\documentclass[12pt]{amsart}
\usepackage{amsrefs}
\usepackage{mathptmx}
\usepackage{amssymb}
\usepackage[all]{xy}
\newif\ifdebug

\debugfalse

\newcommand{\printname}[1]{\smash{\makebox[0pt]{\hspace{-1.0in}\raisebox{8pt}{\tiny #1}}}}

\newcommand{\Label}[1]{\ifdebug{\label{#1}\printname{#1}}\else{\label{#1}}\fi}

\theoremstyle{plain}
\newtheorem{theo}{Theorem}[section]
\newtheorem{lemm}[theo]{Lemma}
\newtheorem{coro}[theo]{Corollary}
\newtheorem{prop}[theo]{Proposition}

\theoremstyle{remark}
\newtheorem{rema}[theo]{Remark}

\theoremstyle{definition}
\newtheorem{defn}[theo]{Definition}
\newtheorem{exam}[theo]{Example}
\newtheorem{prob}[theo]{Problem}
\def\C{\mathbb C}
\def\Z{\mathbb Z}
\def\R{\mathbb R}
\def\D{\mathbb D}

\def\co{\colon\thinspace}
\DeclareMathOperator{\Hom}{\mathrm{Hom}}

\DeclareMathOperator{\pos}{\mathrm{pos}}
\DeclareMathOperator{\coker}{\mathrm{coker}}
\DeclareMathOperator{\link}{\mathrm{link}}
\DeclareMathOperator{\id}{\mathrm{id}}
\DeclareMathOperator{\Aut}{\mathrm{Aut}}

  \makeatletter
    
    \@addtoreset{equation}{section}
  \makeatother

\begin{document}
\title[Complex manifolds with maximal torus actions]{Complex manifolds with maximal torus actions}

\author[H.~Ishida]{Hiroaki Ishida}
\address{Research Institute for Mathematical Science, Kyoto University}
\email{ishida@kurims.kyoto-u.ac.jp}

\date{\today}
\thanks{The author was supported by JSPS Research Fellowships for Young 
Scientists}


\begin{abstract}
	In this paper, we introduce the notion of maximal actions of compact tori on smooth manifolds and study compact connected complex manifolds equipped with maximal actions of compact tori. We give a complete classification of such manifolds, in terms of combinatorial objects, which are  triples $(\Delta, \mathfrak{h}, G)$ of nonsingular complete fan $\Delta$ in $\mathfrak{g}$, complex vector subspace $\mathfrak{h}$ of $\mathfrak{g}^\C$ and compact torus $G$ satisfying certain conditions. We also give an equivalence of categories with suitable definitions of morphisms in these families, like toric geometry.
	
	We obtain several results as applications of our equivalence of categories; complex structures on moment-angle manifolds, classification of holomorphic nondegenerate $\C^n$-actions on compact connected complex manifolds of complex dimension $n$, and construction of concrete examples of non-K\"{a}hler manifolds. 
\end{abstract}

\maketitle 
\tableofcontents 
\section{Introduction} \Label{sec:intro}
	Toric geometry was established around 1970 by Demazure, Miyake-Oda, Mumford etc. It provides examples of concrete algebraic varieties and finds many interesting connections with combinatorics, see \cite {Cox},\cite{Fulton} and \cite{Oda} for history. A toric variety is a normal algebraic variety over the complex numbers $\C$ with an effective algebraic action of an algebraic torus $G^\C$ having an open dense orbit. On the other hand, a fan is a collection of cones in a real vector space with the origin as vertex satisfying certain conditions. A fundamental theorem in toric geometry says 
	that the category of toric varieties is equivalent to the category of rational fans.
	
	In this paper, we study complex manifolds with torus actions, which are siblings of nonsingular complete toric varieties from the viewpoint of compact torus actions. 
	Let $M$ be a connected smooth manifold, equipped with an effective action of a compact torus $G$. Then $\dim G + \dim G_x \leq \dim M$ for any point $x \in M$, where $G_x$ is the isotropy subgroup at $x$ of $G$. This inequality motivates us to consider the case when the equality holds for some $x$. 
	
	We say that the action of $G$ on $M$ is \emph{maximal} if there exists a point $x \in M$ such that 
	\begin{equation*}
		\dim G + \dim G_x = \dim M.
	\end{equation*}
	The main concern in this paper is complex manifolds equipped with maximal actions of compact tori which preserve the complex structures. There are rich examples of such manifolds:
	\begin{exam}[compact complex torus $\C ^n/\Gamma$]\Label{exam:torus} 
		Let $\Gamma $ be a lattice of $\C^n$, that is, a cocompact discrete subgroup of $\C ^n$. Let $M = \C ^n/\Gamma$ and let $G=\C ^n/\Gamma \cong (S^1)^{2n}$ act on $M$ as translations. The $G$-action on $M$ preserves the complex structure on $M$. The isotropy subgroup is trivial at any point in $M$, so the action of $G$ on $M$ is maximal.
	\end{exam}
	\begin{exam}[nonsingular complete toric variety]\Label{exam:toric}
		Let $M$ be a nonsingular complete toric variety and let $G$ be the maximal compact torus of the algebraic torus acting on $M$. The $G$-action on $M$ preserves the complex structure and it is well-known that $M$ has a fixed point. Since $\dim G= \frac{1}{2}\dim M$ and $G_x = G$ at each fixed point $x \in M^G$, the action of $G$ on $M$ is maximal. 
	\end{exam}
	As is shown in \cite{Ishida-Karshon}, a compact connected complex manifold equipped with a maximal action of a compact torus with at least one fixed point is biholomorphic to a nonsingular complete toric variety.
	\begin{exam}[Hopf manifold]\Label{exam:Hopf}
		Take $\alpha \in \C ^*$ such that $|\alpha|\neq 1$.
		Define an equivalence relation $\sim$ on $\C ^n\setminus \{0\}$ so that
		\begin{equation*}
			z \sim w \text{ if and only if there is $k \in \Z$ such that } z=\alpha ^kw. 
		\end{equation*}
		The quotient $M= \C ^n\setminus \{0\}/\sim$ is a complex manifold which is diffeomorphic to $S^{2n-1}\times S^1$. We equip the $G = (S^1)^{n+1}$-action on $M$ as follows. Take an element $\widetilde{\alpha} \in \frac{1}{2\pi \sqrt{-1}}\log \alpha$. Define an $\R ^{n+1}$-action on $M$ by 
		\begin{equation*}
			(v_1,\dots ,v _n,v_{n+1})\cdot [z_1,\dots ,z_n] := [e^{2\pi \sqrt{-1}v_1+2\pi \sqrt{-1}\widetilde{\alpha}v_{n+1}}z_1,\dots ,e^{2\pi\sqrt{-1}v _n+2\pi \sqrt{-1}\widetilde{\alpha}v_{n+1}}z_n],
		\end{equation*}
		where $[z]$ denotes the equivalence class of $z \in \C ^n\setminus \{0\}$. Clearly, the $\R ^{n+1}$-action on $M$ preserves the complex structure and $\R ^{n+1}$ has the global stabilizers $\Z ^{n+1}$. Hence the $\R ^{n+1}$-action descends to an effective $G=(S^1)^{n+1}$-action which preserves the complex structure on $M$. At the point $x=[1,0,\dots, 0] \in M$, the isotropy subgroup is $\{1\} \times (S^1)^{n-1} \times \{1\}$ which is an $(n-1)$-dimensional torus. Therefore the action of $G$ on $M$ is maximal.
	\end{exam}
	
	\begin{exam}[Calabi-Eckmann manifold (see \cite{Calabi-Eckmann} for details)]\Label{exam:Calabi-Eckmann}
		Define a $(\C^*)^m$-action on $(\C^k \setminus \{0\}) \times (\C^{m-k} \setminus \{0\})$ by
		\begin{equation*}
			(g_1,\dots, g_m) \cdot (z_1,\dots, z_k,z_{k+1},\dots, z_m) :=(g_1z_k,\dots, g_mz_m). 
		\end{equation*}
		
		As in Example \ref{exam:Hopf}, take $\widetilde{\alpha} \in \C \setminus \R$. Define a subgroup 
		\begin{equation*}
			H := \left\{ (\underbrace{e^t,\dots, e^t}_k,\underbrace{e^{\widetilde\alpha t},\dots, e^{\widetilde{\alpha} t}}_{m-k}) \in (\C ^*)^m  \mid t \in \C \right\}. 
		\end{equation*}
		The action of $(\C ^*)^m$ restricted to $H$ on $(\C^k \setminus \{0\}) \times (\C^{m-k} \setminus \{0\})$ is free, proper and holomorphic. Therefore, the quotient space $M := (\C^k \setminus \{0\}) \times (\C^{m-k} \setminus \{0\})/H$ becomes a complex manifold. Thinking of $S^{2k-1}$ and $S^{2m-2k-1}$ as the unit spheres in $\C^k$ and $\C^{m-k}$ respectively, the inclusion $S^{2k-1} \times S^{2m-2k-1} \hookrightarrow (\C^k \setminus \{0\}) \times (\C^{m-k} \setminus \{0\})$ induces a diffeomorphism $S^{2k-1} \times S^{2m-2k-1} \to M$. In particular, $M$ is a compact connected complex manifold. We give an action of $G=(S^1)^m$ on $M$ as 
		\begin{equation*}
			(g_1,\dots, g_m) \cdot [z_1,\dots, z_k,z_{k+1},\dots, z_m]:= [g_1z_1,\dots ,g_mz_m]
		\end{equation*}
		for $(g_1,\dots, g_m) \in (S^1)^m$ and $[z_1,\dots, z_m]\in M$, where $[z_1,\dots, z_m]$ denotes the equivalence class of $(z_1,\dots, z_m) \in (\C ^k \setminus \{0\} ) \times (\C ^{m-k} \setminus \{0\})$. This $G$-action preserves the complex structure on $M$. At the point $x=[z_1, 0,\dots, 0, z_{k+1},0,\dots, 0]$ with $z_1,z_{k+1} \neq 0$, the isotropy subgroup $G_x$ is $\{1\} \times (S^1)^{k-1} \times \{1\} \times (S^1)^{m-k-1}$, that is, an $(m-2)$-dimensional torus. Therefore $\dim G+ \dim G_x = 2m-2 = \dim M$, so the action of $G$ on $M$ is maximal. 
	\end{exam}
	More interesting examples are \emph{LVM manifolds}, \emph{LVMB manifolds} and \emph{moment-angle manifolds} with complex structures. LVM manifolds are non-K\"{a}hler compact complex manifolds, constructed in \cite{Meersseman} and the acronyms stand for Lo\`{p}ez de Medrano-Verovsky \cite{LV} and Meersseman. LVMB manifolds are generalizations of LVM manifolds, constructed by Bosio in \cite{Bosio}. LVMB manifolds are obtained as the quotients of complements of coordinate subspace arrangements in projective spaces $\C P^n$ by free, proper, and holomorphic $\C ^m$-actions for some $n$. They are naturally equipped with actions of compact tori which preserve the complex structures and are maximal (see Proposition \ref{prop:LVMBmax}). Moment-angle manifolds are topological manifolds equipped with actions of compact tori, constructed from simplicial complexes. In \cite{Bosio-Meersseman}, the smooth manifolds underlying a large class of LVM manifolds are shown to be equivariantly homeomorphic to moment-angle manifolds coming from simplicial polytopal sphere. In \cite{Panov-Ustinovsky} and \cite{Tambour}, it is shown that even dimensional moment-angle manifolds coming from star-shaped simplicial spheres carry complex structures independently.  The actions of compact tori on them preserve the complex structures and are maximal. 

	We consider the class $\mathcal{C}_1$ which consists of all $(M,G,y)$ of a compact connected complex manifold $M$ equipped with a maximal action of a compact torus $G$ which preserves the complex structure on $M$ and a point $y \in M$ such that $G_y$ is trivial. As a combinatorial counterpart of $\mathcal{C}_1$, we also consider the family $\mathcal{C}_2$ which consists of all triples $(\Delta, \mathfrak{h}, G)$, where $\Delta$ is a nonsingular fan in the Lie algebra $\mathfrak{g}$ of the compact torus $G$ and $\mathfrak{h}$ is a complex vector subspace of $\mathfrak{g}^\C := \mathfrak{g} \otimes \C \cong \mathfrak{g} \oplus \sqrt{-1}\mathfrak{g}$ satisfying the following conditions:
	\begin{enumerate}
		\item the restriction $p|_{\mathfrak{h}}$ of the projection $p \co \mathfrak{g}^\C \to \mathfrak{g}$ is injective and 
		\item the quotient map $q \co \mathfrak{g} \to \mathfrak{g}/p(\mathfrak{h})$ sends the fan $\Delta$ to a complete fan $q(\Delta)$ in $\mathfrak{g}/p(\mathfrak{h})$. 
	\end{enumerate}	
	
	To each $(M,G,y)$ in $\mathcal{C}_1$, we may assign a triple $(\Delta,\mathfrak{h},G)$ which sits in $\mathcal{C}_2$ as follows. Each connected component of the fixed point set of a circle subgroup of $G$ is a closed complex submanifold of $M$. If such a submanifold has complex codimension one, then, in analogy with the toric topology literature, we call it a \emph{characteristic submanifold} of $M$ (cf.~\cite[p.240]{Masuda}). One can see that the number of characteristic submanifolds of $M$ is finite (possibly $0$). Let $N_1,\dots, N_k$ be the characteristic submanifolds of $M$. We define an abstract simplicial complex $\Sigma$ on the vertex set $\{1,\dots, k\}$ as 
	\begin{equation*}
		\Sigma := \left\{ I \subseteq \{1,\dots, k\} \mid \bigcap _{i \in I} N_i \neq \emptyset \right\}.
	\end{equation*}
	Let $G_i$ be the circle subgroup of $G$ which fixes $N_i$ pointwise. Define $\lambda_i \co S^1 \to G$ so that 
	\begin{equation*}
		(\lambda_i(g))_*(\xi)=g\xi \quad \text{for all $g \in S^1, \xi \in TM|_{N_i}/TN_i$}.
	\end{equation*}
	We may regard $\lambda_i$ as a vector in $\mathfrak{g}$ and put $C_I := \{\sum_{i \in I}a_i\lambda_i \mid a_i\geq 0\}$ for $I \in \Sigma$. Then, we have a collection $\Delta$ of cones $C_I$ in $\mathfrak{g}$. One can see that $\Delta$ is a nonsingular fan in $\mathfrak{g}$. Since the action of $G$ on $M$ preserves the complex structure on $M$, we may consider the complexified action $G^\C \times M \to M$. The vector subspace $\mathfrak{h}$ is the Lie subalgebra of the global stabilizers of this complexified action. One can see that $\Delta$ and $\mathfrak{h}$ satisfy conditions (1) and (2) above and hence we have a map $\mathcal{F}_1 \co \mathcal{C}_1 \to \mathcal{C}_2$. $\mathcal{F}_1$ does not depend on the choice of $y$, but we will use $y$ for defining the hom-sets on $\mathcal{C}_1$ later. We often denote by $\mathcal{F}_1(M,G)$ the value $\mathcal{F}_1(M,G,y)$ for $y \in M$ such that $G_y$ is trivial.
	
	Conversely, we can construct a compact connected complex manifold equipped with an action of a compact torus $G$ which is maximal and preserves the complex structure from each triple $(\Delta,\mathfrak{h},G)$  in $\mathcal{C}_2$. Let $X(\Delta)$ be the nonsingular toric variety associated with $\Delta$. One can see that the quotient $X(\Delta)/H$ is a compact connected complex manifold, where $H:= \exp (\mathfrak{h}) \subseteq G^\C$. The action of $G$ on $X(\Delta)$ descends to an action on $X(\Delta)/H$ which is maximal and preserves the complex structure. Since $X(\Delta)$ is a toric variety, by definition, the unit of $G^\C$ can be regarded as a point $1_{X(\Delta)} \in X(\Delta)$. So we have a map $\mathcal{F}_2 \co \mathcal{C}_2 \to \mathcal{C}_1$ given by $\mathcal{F}_2(\Delta,\mathfrak{h},G) := (X(\Delta)/H, G, [1_{X(\Delta)}])$.

	The first theorem in this paper is the following: 
	\begin{theo}[See also Theorem \ref{theo:maintheo}]\Label{intro:maintheo}
		Let $M$ be a compact connected complex manifold equipped with a maximal action of a compact torus $G$ which preserves the complex structure on $M$, and let $(\Delta, \mathfrak{h}, G) = \mathcal{F}_1(M,G)$. Then, $M$ is equivariantly biholomorphic to $X(\Delta)/H$. 
	\end{theo}
	Theorem \ref{intro:maintheo} tells us that a compact connected complex manifold $M$ equipped with a maximal action of a compact torus $G$ which preserves the complex structure on $M$ is equivariantly biholomorphic to the quotient of a nonsingular toric variety by a subgroup $H$ of $G^\C$. 
	\begin{exam}[compact complex torus $\C^n/\Gamma$ (continuation)]
		Let $\gamma_1,\dots, \gamma_{2n}$ be generators of the lattice $\Gamma$. Consider the $\C$-linear map 
		\begin{equation*}
			f \co \C^{2n} \to \C^n, \quad f(e_i) = \gamma_i
		\end{equation*}
		for all $i$, where $e_i$ denotes the $i$-th standard basis vector of $\C^{2n}$. The map $f$ descends to a surjective holomorphic homomorphism $\overline{f} \co \C^{2n}/\Z^{2n} \to \C^n /\Gamma$. Through the exponential function $t \mapsto e^{2\pi\sqrt{-1}t}$, we can identify $\C^{2n}/\Z^{2n}$ with the algebraic torus $(\C^*)^{2n}$. So we have a surjective holomorphic homomorphism $\overline{f}\co (\C^*)^{2n} \to \C^n/\Gamma$. The Lie algebra of $\ker \overline{f}=:H$ is $\ker f=:\mathfrak{h}$ and $\C^n/\Gamma$ is isomorphic to the quotient $(\C^*)^{2n}/H$. Remark that $(\C^*)^{2n}$ is the affine toric variety associated with the fan consisting of only the origin $\{0\}$ in $\R^{2n}$. 
	\end{exam}
	\begin{exam}[nonsingular complete toric variety (continuation)] 
		Let $(\Delta,\mathfrak{h},G) \in \mathcal{C}_2$. In the case when $\mathfrak{h}=\{0\}$, the condition (2) implies that $\Delta$ is a complete fan. Theorem \ref{intro:maintheo} generalizes the correspondence between nonsingular complete toric varieties and nonsingular complete fans through \cite[Theorem 1]{Ishida-Karshon}.
	\end{exam}
	\begin{exam}[Hopf manifold (continuation)]
		Let $\widetilde{\alpha} \in \C \setminus \R$ and let $\alpha = e^{2\pi \sqrt{-1}\widetilde{\alpha}}$. We may think of the Hopf manifold $M=\C^n \setminus \{0\}/\sim$ in Example \ref{exam:Hopf} as a Calabi-Eckmann manifold. In fact, 
		\begin{equation*}
			\C^n /\sim \to (\C^n \setminus \{0\}) \times \C^*/H, \quad [z_1,\dots, z_n] \mapsto [z_1,\dots, z_n,1]
		\end{equation*}
		is a biholomorphism, where $H$ is defined as 
		\begin{equation*}
			H := \left\{ (\underbrace{e^t,\dots, e^t}_n,e^{\widetilde{\alpha}t}) \in (\C^*)^{n+1} \mid t \in \C\right\}.
		\end{equation*}
		So Example \ref{exam:Calabi-Eckmann2} below contains the case of Hopf manifolds. 
	\end{exam}
	\begin{exam}[Calabi-Eckmann manifold (continuation)]\Label{exam:Calabi-Eckmann2}
		Example \ref{exam:Calabi-Eckmann} is suitable to explain Theorem \ref{intro:maintheo}. Consider the Calabi-Eckmann manifold $M = (\C^k \setminus \{0\}) \times (\C^{m-k}\setminus \{0\})/H$. There are $m$ characteristic submanifolds of $M$ and each of them is 
	\begin{equation*}
		N_i := \{ [z_1,\dots, z_m] \in M \mid z_i=0\} 
	\end{equation*}
	for $i=1,\dots, m$. The circle subgroup $G_i$ which fixes $N_i$ pointwise is the $i$-th coordinate $1$-dimensional subtorus of $G=(S^1)^m$. So one can see that the corresponding fan $\Delta$ is 
	\begin{equation*}
		\Delta := \{ \pos (e_i \mid i \in I \cup J) \mid I \subsetneq \{1,\dots, k\}, J \subsetneq \{k+1,\dots, m\}\}
	\end{equation*}
	and the associated toric variety $X(\Delta)$ is $(\C^k \setminus \{0\}) \times (\C^{m-k}\setminus \{0\})$, where $\pos(A)$ denotes the cone spanned by elements in $A$. The complexified action $(\C^*)^m \times M \to M$ is given by 
	\begin{equation*}
		(g_1,\dots, g_m)\cdot [z_1,\dots, z_m] = [g_1z_1,\dots, g_mz_m]
	\end{equation*}
	for $(g_1,\dots, g_m) \in (\C^*)^m$ and $[z_1,\dots, z_m] \in M$, so the global stabilizers coincide with $H$. Finally, we shall see that $\Delta$ and the Lie algebra $\mathfrak{h}$ of $H$ satisfies the conditions (1) and (2). By direct computation, $p(\mathfrak{h}) \subseteq \R^m$ is spanned by the vectors $v_1=(\underbrace{1,\dots, 1}_k,\underbrace{0,\dots, 0}_{m-k})$ and $v_2=(\underbrace{0,\dots, 0}_k,\underbrace{1,\dots, 1}_{m-k})$. So we may identify $\R^m/p(\mathfrak{h})$ with $\R^{m-2}$ by 
	\begin{equation*}
		[x_1,\dots, x_m] \mapsto (x_1-x_k,\dots, x_{k-1}-x_k, x_{k+1}-x_m,\dots, x_{m-1}-x_m).
	\end{equation*}
	Therefore the image of the fan $\Delta$ by the map $q \co \R^m \to \R^{m-2}$ is the product of the fans associated with projective spaces $\C P^{k-1}$ and $\C P^{m-k-1}$. In particular, $q(\Delta)$ is a complete fan in $\R^{m-2}$. 
	\end{exam}
	
	We also consider morphisms in both $\mathcal{C}_1$ and $\mathcal{C}_2$ and show that $\mathcal{C}_1$ and $\mathcal{C}_2$ are equivalent. For $(M_1,G_1,y_1), (M_2,G_2,y_2) \in \mathcal{C}_1$, a morphism $\psi \co (M_1,G_1,y_1) \to (M_2,G_2,y_2)$ is an $\alpha$-equivariant\footnote{A map $f \co X_1 \to X_2$ between a $G_1$-set $X_1$ and a $G_2$-set $X_2$ is said to be \emph{$\alpha$-equivariant} if $\alpha(g) \circ f = f \circ g$ for all $g \in G_1$.} holomorphic map $\psi \co M_1 \to M_2$ such that $\psi(y_1) = y_2$ for some group homomorphism $\alpha \co G_1 \to G_2$. Obviously $\mathcal{C}_1$ with the hom-sets $\Hom _{\mathcal{C}_1}((M_1,G_1,y_1),(M_2,G_2,y_2))$ becomes a category. 
	
	For $(\Delta_1,\mathfrak{h}_1, G_1), (\Delta_2,\mathfrak{h}_2,G_2) \in \mathcal{C}_2$, We define a hom-set $\Hom _{\mathcal{C}_2} ((\Delta_1,\mathfrak{h}_1, G_1), (\Delta_2,\mathfrak{h}_2,G_2))$ to be the set of all Lie group homomorphism $\alpha \co G_1 \to G_2$ such that the linear map $(d\alpha)_1 \co \mathfrak{g}_1 \to \mathfrak{g}_2$ induces a morphism from $\Delta_1$ to $\Delta_2$ and the linear map $(d\alpha)_1 \otimes \id_\C \co \mathfrak{g}_1^\C \to \mathfrak{g}_2^\C$ satisfies $(d\alpha)_1 \otimes \id_\C(\mathfrak{h}_1) \subseteq \mathfrak{h}_2$. For morphisms, the composition is defined by the composition of group homomorphisms. Then, we can see that $\mathcal{C}_2$ with the hom-sets $\Hom _{\mathcal{C}_2} ((\Delta_1,\mathfrak{h}_1, G_1), (\Delta_2,\mathfrak{h}_2,G_2))$ forms a category, too. 
	
	We define mapping functions and use the same notations $\mathcal{F}_1$ and $\mathcal{F}_2$. 
	For a morphism $\psi \in \Hom _{\mathcal{C}_1}((M_1,G_1,y_1), (M_2,G_2,y_2))$, we denote by $\mathcal{F}_1(\psi)$ the group homomorphism $\alpha \co G_1 \to G_2$ if $\psi$ is $\alpha$-equivariant. Conversely, for a morphism $\alpha \in \Hom _{\mathcal{C}_2}((\Delta_1,\mathfrak{h}_1, G_1), (\Delta_2,\mathfrak{h}_2,G_2))$, let $\widetilde{\psi_{\alpha}} \co X(\Delta_1) \to X(\Delta_2)$ denote the toric morphism associated with $(d\alpha)_1 \co \mathfrak{g}_1 \to \mathfrak{g}_2$. $\widetilde{\psi_\alpha}$ descends to an $\alpha$-equivariant holomorphic map $\psi _\alpha \co X(\Delta_1)/H_1 \to X(\Delta_2)/H_2$, where $H_i$ denotes the image of $\mathfrak{h}_i$ by the exponential map for $i=1,2$. We define $\mathcal{F}_2(\alpha)$ to be $\psi_\alpha$. We can see that both $\mathcal{F}_1$ and $\mathcal{F}_2$ are covariant functors. 
	\begin{theo}[See also Theorem \ref{theo:equivalence}]\Label{intro:equivalence}
		The covariant functors $\mathcal{F}_1 \co \mathcal{C}_1 \to \mathcal{C}_2$ and $\mathcal{F}_2 \co \mathcal{C}_2 \to \mathcal{C}_1$ are weak inverse to each other. In particular, the categories $\mathcal{C}_1$ and $\mathcal{C}_2$ are equivalent. 
	\end{theo}

	As an application of Theorem \ref{intro:equivalence}, we show that $\mathcal{F}_1$ gives a complete invariant as complex manifolds. Namely,
	\begin{theo}[See also Theorem \ref{theo:biholo}]\Label{intro:biholo}
		Let $M_1$ and $M_2$ be compact connected complex manifolds, equipped with maximal actions of compact tori $G_1$ and $G_2$ which preserve the complex structures on $M_1$ and $M_2$, respectively. Then, $M_1$ is biholomorphic to $M_2$ if and only if $\mathcal{F}_1(M_1,G_1)$ is isomorphic to $\mathcal{F}_1(M_2,G_2)$. 
	\end{theo}
	We also study on equivariant princiapal holomorphic bundles in complex manifolds with maximal torus actions. We show that LVMB manifolds appear as building blocks of complex manifolds with maximal torus actions.  
	\begin{theo}[See also Theorem \ref{theo:universal}]\Label{intro:universal}
		For any compact connected complex manifold $M$ equipped with a maximal action of a compact torus $G$ which preserves the complex structure on $M$, there exists an LVMB manifold $X_{n,m}$ and an equivariant principal holomorphic bundle $\psi \co X_{n,m} \to M$. In particular, every compact complex manifold with a maximal torus action is a quotient of an LVMB manifold. 
	\end{theo}
	
	As another application, we also give a complete classification of compact connected complex $n$-manifolds with holomorphic nondegenerate $\C^n$-actions. The notion of nondegenerate action is introduced by Zung and Minh in \cite{Zung-Minh}. The classification of compact connected complex manifolds which admit nondegenerate $\C^n$-actions can be obtained by Theorem \ref{intro:biholo} and by the following:
	\begin{theo}[See also Theorem \ref{theo:nondeg}]\Label{intro:nondeg}
		Let $M$ be a compact connected complex manifold of complex dimension $n$. Then, $M$ admits a holomorphic nondegenerate $\C^n$-action if and only if $M$ admits an action of a compact torus $G$ which is maximal and preserves the complex structure on $M$. 
	\end{theo}
	
	We give a necessary condition for a compact connected complex manifold  $M$ with a maximal action of a compact torus $G$ preserving the complex structure on $M$ to admit a K\"{a}hler metric in terms of $\mathcal{F}_1(M,G)$ (Theorem \ref{theo:kaehler}). It turns out that we can obtain concrete examples of non-K\"{a}hler manifolds from elements in $\mathcal{C}_2$ which do not satisfy the necessary condition. 

	This paper is organized as follows. In Section \ref{sec:maximal}, we see some properties of maximal torus actions. In Section \ref{sec:complexified}, we briefly recall how to construct complexified action and prepare some lemmas for later use. In Section \ref{sec:minimal}, we study the complex structure around minimal orbits. Especially, we give an explicit presentation of the complex structure around a minimal orbit. In Section \ref{sec:fan}, we assign a nonsingular fan $\Delta$ in $\mathfrak{g}$ and a complex vector subspace $\mathfrak{h}$ of $\mathfrak{g}^\C$ to a compact connected complex manifold $M$ equipped with a maximal torus action which preserves the complex structure. We also characterize the pair $(\Delta,\mathfrak{h})$ which comes from $(M,G)$. In Section \ref{sec:bundle}, we construct a principal bundle over an open subset of $M$ whose total space is a toric variety. This construction is the key of the proof of Theorem \ref{intro:maintheo}. In Section \ref{sec:maintheo}, we study the quotients of toric varieties by subgroups of algebraic tori, and prove Theorem \ref{intro:maintheo}. Throughout Sections \ref{sec:minimal} -- \ref{sec:maintheo}, we use arguments similar to \cite{Ishida-Karshon}. In Section \ref{sec:category}, we define morphisms in $\mathcal{C}_1$ and $\mathcal{C}_2$ and show that $\mathcal{F}_1$ and $\mathcal{F}_2$ are equivalence of categories. In Section \ref{sec:LVMBmam}, we clarify relationships between complex manifolds with maximal torus actions, LVMB manifolds and moment-angle manifolds with invariant complex structures. In Section \ref{sec:epb}, we study on the equivariant principal bundles in complex manifolds with maximal torus actions and show Theorem \ref{intro:universal}. 
	In Section \ref{sec:nondeg}, we clarify the relation with holomorphic nondegenerate $\C^n$-actions and give a complete classification of compact connected complex $n$-manifolds which admit holomorphic nondegenerate $\C^n$-actions. In Section \ref{sec:kaehler}, we give a necessary condition for a compact connected complex manifold equipped with a maximal action of a compact torus to admit a K\"{a}hler metric.

\section{Maximal torus actions and  minimal orbits} \Label{sec:maximal}
	Let $M$ be a connected manifold equipped with an effective action of a compact torus $G$. 
	Since $M$ is connected, $G$ is abelian and the action on $M$ of $G$ is effective, the normal subspace $T_xM/T_x(G\cdot x)$ must be a faithful representation of the isotropy subgroup $G_x$ at $x$. Since the identity component $G_x^0$ of $G_x$ is a subtorus of $G$, it follows from the faithfulness of $T_xM/T_x(G\cdot x)$ that 
	\begin{equation} \Label{eq:inequality1}
		\dim G_x \leq \frac{1}{2}(\dim M-(\dim G -\dim G_x))
	\end{equation}
	for all $x \in M$. The inequality \eqref{eq:inequality1} can be simplified to 
	\begin{equation} \Label{eq:inequality2}
		\dim G + \dim G_x \leq \dim M
	\end{equation}
	and this inequality \eqref{eq:inequality2} is equivalent to 
	\begin{equation} \Label{eq:inequality3}
		\dim G\cdot x \geq 2\dim G -\dim M. 
	\end{equation}
	
	\begin{defn}
		Let $M$ be a connected manifold equipped with an effective action of a compact torus $G$. We say that the action of $G$ on $M$ is \emph{maximal} if there is a point $x \in M$ such that 
		\begin{equation*} 
			\dim G + \dim G_x=\dim M,
		\end{equation*}
		that is, the equality of \eqref{eq:inequality2} holds.
		We say that an orbit $G\cdot x$ through $x$ is \emph{minimal} if 
		\begin{equation*}
			\dim G\cdot x = 2\dim G -\dim M,
		\end{equation*}
		that is, the equality of \eqref{eq:inequality3} holds. 
	\end{defn}
	It imediatly follows that the action of $G$ on $M$ is maximal if and only if there exists a minimal orbit $G \cdot x$. 
	
	Roughly speaking, the maximal action on $M$ means that there is no larger compact torus which acts on $M$ effectively. We shall state this for later use.
	\begin{lemm} \Label{lemm:maximal}
		Let $M$ be a connected manifold equipped with an effective action of a compact torus $G'$. Let $G$ be a subtorus of $G'$. Suppose that the action of $G'$ restricted to $G$ on $M$ is maximal. Then, $G=G'$. 
	\end{lemm}
	\begin{proof}
		It suffices to show that $\dim G=\dim G'$. Since the action of $G$ on $M$ is maximal, there exists a point $x \in M$ such that $\dim G+\dim G_x = \dim M$. Since $\dim G \leq \dim G'$ and $\dim G_x \leq \dim G'_x$, we have that 
		\begin{equation*} 
				\dim M = \dim G + \dim G_x \leq \dim G'+ \dim G'_x \leq \dim M
		\end{equation*} by \eqref{eq:inequality2}. Therefore $\dim G+\dim G_x= \dim G'+\dim G_x'$. Using $\dim G \leq \dim G'$ and $\dim G_x \leq \dim G'_x$ again, we have that $\dim G= \dim G'$, proving the lemma.
	\end{proof}
	We shall state the following for later use, too. 
	\begin{lemm}\Label{lemm:cfi}
		Let $M$ be a connected manifold equipped with a maximal action of a compact torus $G$. Let $G\cdot x$ be a minimal orbit. Then, the followings hold: 
		\begin{enumerate}
			\item The isotropy subgroup $G_x$ of $G$ at $x$ is connected. 
			\item $G\cdot x$ is a connected component of the fixed point set of the action of $G$ restricted to $G_x$ on $M$.
			\item Each minimal orbit is isolated. In particular, there are finitely many minimal orbits if $M$ is compact. 
		\end{enumerate}
	\end{lemm}
	\begin{proof}
		Since the action of $G$ is maximal, the representation $T_xM/T_x(G\cdot x)$ of $G_x$ is faithful and 
		\begin{equation} \Label{eq:equality1}
			\dim G_x = \frac{1}{2}\dim T_xM/T_x(G\cdot x). 
		\end{equation}
		Consider the injective homomorphism $G_x \to GL(T_xM/T_x(G\cdot x))$ associated with the representation $T_xM/T_x(G\cdot x)$. Since $G_x$ is compact and abelian, it follows from \eqref{eq:equality1} that the image of the whole $G_x$ by the homomorphism is a maximal compact torus of $GL(T_xM/T_x(G\cdot x))$. Therefore $G_x$ should be connected, proving Part (1).
		
		Since the image of $G_x$ by the homomorphism $G_x \to GL(T_xM/T_x(G\cdot x))$ is a maximal compact torus of $GL(T_xM/T_x(G\cdot x))$, we have that 
		\begin{equation*}
			(T_xM/T_x(G\cdot x))^{G_x}=\{0\}.
		\end{equation*}
		This shows that 
		\begin{equation*}
			T_xM^{G_x} = T_x(G\cdot x).
		\end{equation*}
		This together with the slice theorem shows Part (2).
		
		It follows from Part (2) and the slice theorem that each minimal orbit is isolated, proving Part (3). 
	\end{proof}
	
	\section{Complexified actions} \Label{sec:complexified}
	Let $M$ be a complex manifold and let a connected abelian group $G$ act on $M$ preserving the complex structure $J$. Then, under a certain condition, we can construct the \emph{complexified action} on $M$ of a complex Lie group $G^\C$ such that the action map is holomorphic. 
	
	For a vector $v$ in the Lie algebra $\mathfrak{g}$ of $G$, we denote by $X_v$ the fundamental vector field generated by $v$. For a vector field $X$ on $M$, we denote by $\exp (X)$ the flow at time $1$ which is a diffeomorphism from $M$ to $M$ itself when $\exp(X)$ can be defined. 
	
	Any element in $\mathfrak{g}^\C := \mathfrak{g}\otimes_\R \C$ can be represented as $u\otimes 1 + v\otimes{\sqrt{-1}}$ for unique $u,v \in \mathfrak{g}$ and hence it is isomorphic to $\mathfrak{g}\oplus \sqrt{-1}\mathfrak{g}$ as real vector spaces. For simplicity, we denote by $u+\sqrt{-1}v$ the element $u\otimes 1 + v\otimes{\sqrt{-1}}$. 
	
	Let $J$ be the complex structure on $M$. Suppose that $JX_v$ is complete for any $v \in \mathfrak{g}$. Then we define a map $\mathfrak{g}^\C \times M \to M$ to be $(u+\sqrt{-1}v,x)=\exp(X_u+JX_v)(x)$ for $u,v \in \mathfrak{g}$. We claim that the map $\mathfrak{g}^\C \times M \to M$ is a holomorphic action. 
	
	First, we see that the map $\mathfrak{g}^\C \times M \to M$ is an action. To see this, we see that the vector fields $X_u, X_v, JX_u, JX_v$ commute with each other. 
	Since  $\mathfrak{g}$ is commutative, $[X_u,X_v]=0$. Since the flow of $X_u$ preserves $J$ and $X_v$, it also preserves $JX_u$ and $JX_v$. So $[X_u, JX_u]=[X_u,JX_v]=0$. Since $J$ is the complex structure, its Nijenhuis tensor, $N(Z,W):=2([JZ,JW]-J[JZ,W]-J[Z,JW]-[Z,W])$ vanishes for any vector fields $Z,W$ on $M$. Setting $Z=X_u$ and $W=X_v$ we have that $[JX_u,JX_v]=J[JX_u,X_v]+J[X_u,JX_v]+[X_u,X_v]$ and each of the three terms on the right hand side is zero. So the vector fields $X_u,X_v,JX_u$ and $JX_v$ commute with each other. This shows that the map $\mathfrak{g}^\C \times M \to M$ is an action. 
	
	Now we see that the flow of $JX_v$ preserves the complex structure $J$ on $M$. 
	The flow of a vector field $Y$ preserves $J$ if and only if $[Y,JW]=J[Y,W]$ for each vector field $W$ (see \cite[Proposition 2.10 in Chapter IX]{Kobayashi-Nomizu}). Setting $Y=JX_u$ and $W$ arbitrary, it follows from the vanishing of the Nijenhuis tensor and the assumption that the flow of $X_u=-JY$ preserves $J$ that 
	\begin{equation*}
		\begin{split}
			0=N(JY,W)&=2([-Y,JW]-J[-Y,W]-J[JY,JW]-[JY,W])\\
			&=2([-Y,JW]-J[-Y,W]).
		\end{split}
	\end{equation*} 
	Hence the flow of $Y=JX_u$ preserves $J$. 
	
	Finally, the action map $\mathfrak{g}^\C \times M \to M$ is holomorphic, because its differential, which at the point $(X_u + \sqrt{-1}X_v, x)$ is the map $\mathfrak{g}^\C \times T_xM \to T_{\exp (X_u+JX_v)(x)}M$ that takes $(u'+\sqrt{-1}v', X')$ to $X_{u'}(\exp(X_u+JX_v)(x))+JX_{v'}(\exp(X_u+JX_v)(x))+d(\exp (X_u+JX_v))(X')$, is complex linear. 
	
	We remark that the exponential map $\mathfrak{g} \to G$ is a surjective homomorphism because $G$ is connected and abelian. So $G$ is isomorphic to $\mathfrak{g}/\Hom (S^1,G)$, where we think of the set $\Hom (S^1,G)$ of homomorphisms $\lambda \co S^1 \to G$ as a subset of $\mathfrak{g}$ via the image of $1$ by the differential of the composition $\lambda \circ e^{2\pi \sqrt{-1}\bullet} \co \R \to S^1 \to G$. The action of $\mathfrak{g}^\C$ on $M$ descends to an action of $G^\C := \mathfrak{g}^\C/\Hom (S^1, G)$. The $G^\C$-action on $M$ is said to be the complexified action. 
	\begin{exam}\Label{exam:rep}
		Let $G=S^1$ act on a vector space $\C$ as a representation and let $\alpha \in \Hom (G,\C^*)$ be its character. Namely, the action $G \times \C \to \C$ is given as 
		\begin{equation*}
			g\cdot v = \alpha (g)v = g^kv \quad \text{for some $k \in \Z$}.
		\end{equation*}
		The complexified action $G^\C \times \C \to \C$ is given by the same formulation 
		\begin{equation*}
			g \cdot v = g^kv \quad \text{for $g \in G^\C = \C^*$}.
		\end{equation*}
		We also denote by $\alpha$ the character of $G^\C \times \C \to \C$. 
	\end{exam}
	\begin{rema}
		Let $G_1$ and $G_2$ be connected abelian Lie groups and let $\alpha \co G_1 \to G_2$ be a Lie group homomorphism. The differential of $\alpha$ at the unit $1$ of $G_1$, the linear map $(d\alpha)_1 \co \mathfrak{g}_1 \to \mathfrak{g}_2$, sends elements in $\Hom (S^1,G_1)$ to elements in $\Hom (S^1,G_2)$. So the $\C$-linear map $(d\alpha)_1^\C := (d\alpha)_1 \otimes \id _\C \co \mathfrak{g}_1^\C \to \mathfrak{g}_2^\C$ induces a complex Lie group homomorphism $G_1^\C \to G_2^\C$. We also denote by $\alpha$ this complex Lie group homomorphism $G_1^\C \to G_2^\C$. 
		
		Let $M_1$ and $M_2$ be complex manifolds equipped with actions of $G_1$ and $G_2$ preserving the complex structures on $M_1$ and $M_2$, respectively. Let $\psi \co M_1 \to M_2$ be an $\alpha$-equivariant holomorphic map, that is, the diagram 
		\begin{equation*}
			\xymatrix{
				G_1 \times M_1 \ar[r]^{\alpha \times \psi} \ar[d] & G_2 \times M_2 \ar[d]\\
				M_1 \ar[r]^\psi & M_2
			}
		\end{equation*}
		commutes, where the vertical arrows are the action maps. Suppose that the complexified actions $G_i^\C \times M_i \to M_i$ can be defined for both $i =1,2$. Then, it follows from the definition of the complexified actions and the holomorphicity of $\psi$ that the diagram 
		\begin{equation*}
			\xymatrix{
				G_1^\C \times M_1 \ar[r]^{\alpha \times \psi} \ar[d] & G_2^\C \times M_2 \ar[d]\\
				M_1 \ar[r]^\psi & M_2
			}
		\end{equation*}
		commutes. So $\psi$ is an $\alpha$-equivariant holomorphic map with respect to the complexified homomorphism $\alpha \co G_1^\C \to G_2^\C$. 
	\end{rema}
	
	The complexified action $G^\C \times M \to M$ might not be effective. We denote by $Z_{G^\C}$ the subgroup which consists of global stabilizers, that is, 
	\begin{equation*}
		Z_{G^\C} := \{ g \in G^\C \mid g\cdot x = x \text{ for all $x \in M$}\}.
	\end{equation*}
	We denote by $G^M$ the quotient complex Lie group $G^\C/Z_{G^\C}$. We  regard $G$ as a subgroup of $G^\C$ and of $G^M$ when $G$ acts on $M$ effectively. 
	The dimension of $G^M$ should be at most the dimension of the manifold $M$. This follows from the following lemma and its corollary.
	\begin{lemm}\Label{lemm:opendense}
		Let $M$ be a connected complex manifold of complex dimension $n$ and let $Y_1,\dots, Y_n$ be holomorphic vector fields such that $Y_1\wedge \dots \wedge Y_n$ is not zero, that is, there exists a point $x \in M$ such that $Y_1(x),\dots ,Y_n(x)$ are linearly independent over $\C$. Then, the open subset
		\begin{equation*}
			S:=\{ x \in M \mid Y_1(x),\dots, Y_n(x) \text{ are linearly independent over $\C$}\}
		\end{equation*}
		is dense in $M$ and connected. 
	\end{lemm}
	\begin{proof}
		Let $\{\varphi _\alpha \co U_\alpha \to \D^n\}_\alpha$ be local holomorphic coordinate systems, that is, $U_\alpha$'s are open subsets of $M$ such that $\bigcup_{\alpha}U_\alpha=M$, $\D^n$ is the polydisc in $\C^n$ and each $\varphi _\alpha$ is a biholomorphism. It suffices to show that $U_\alpha \cap S$ is dense in $U_\alpha$ and connected for each index $\alpha$. Let $z_i$ for $i=1,\dots ,n$ be the $i$-th factor of $\varphi_\alpha$. Using these coordinates we can represent $Y_1\wedge\dots  \wedge Y_n$ as 
		\begin{equation*}
			Y_1\wedge \dots \wedge Y_n = f\frac{\partial}{\partial z_1}\wedge \dots \wedge \frac{\partial}{\partial z_n},
		\end{equation*}
		where $f$ is a holomorphic function defined on $U_\alpha$. For $y \in U_\alpha$, $Y_1(y), \dots ,Y_n(y)$ are linearly independent over $\C$ if and only if $f(y)\neq 0$. Since $f$ is holomorphic, the open subset $\{ y \in U_\alpha \mid f(y) \neq 0\}=U_\alpha \cap S$ is dense and connected unless empty. Suppose $U_{\alpha _0}$ contains $x \in M$ such that $Y_1(x),\dots ,Y_n(x)$ are linearly independent over $\C$. For any index $\alpha$, we can find a sequence of indices $\alpha _1,\dots ,\alpha _k$ such that $U_{\alpha _j}\cap U_{\alpha _{j+1}} \neq \emptyset$ for $j=0,\dots ,k-1$ and $U_{\alpha _k}=U_{\alpha}$ because $M$ is connected. By induction on $k$, we can show that $U_\alpha \cap S$ is dense and connected. By the choice of $\alpha_0$, $U_{\alpha _0} \cap S$ is not empty and hence $U_{\alpha _0} \cap S$ is dense and connected. Suppose that $U_{\alpha _k}\cap S$ is dense in $U_{\alpha _k}$ and connected and $U_{\alpha _k}\cap U_{\alpha _{k+1}}$ is not empty. Then, $U_{\alpha _k}\cap U_{\alpha _{k+1}}\cap S$ is not empty. So $U_{\alpha _{k+1}} \cap S$ is not empty and hence $U_{\alpha _{k+1}} \cap S$ is dense in $U_{\alpha _{k+1}}$ and connected.  
		
		Therefore, for any index $\alpha$, $U_\alpha \cap S$ is dense in $U_\alpha$ and connected, proving the lemma.
	\end{proof}
	\begin{coro}\Label{coro:opendense}
		Let $M$ be a compact connected complex manifold of complex dimension $n$. Let an abelian Lie group $G$ act on $M$ preserving the complex structure and let $\mathfrak{g}^\C$ be as above. 
		Suppose that there exist a point $x \in M$ and $v_1,\dots ,v_n \in \mathfrak{g}^\C$ such that $X_{v_1}(x),\dots ,X_{v_n}(x)$ are linearly independent over $\C$. Then, $G^\C \cdot x$ is an open dense orbit. 
	\end{coro}
	\begin{proof}
		The corollary follows from Lemma \ref{lemm:opendense} immediately. 
	\end{proof}
	For later use, we state the fact that the dimension of $G^M$ is at most the dimension of $M$ as a corollary.
	\begin{coro}\Label{coro:opendense2}
		Let $M$ and $G$ be as above. Then, the dimension of $G^M$ is at most the dimension of $M$.
	\end{coro}
	\begin{proof}
		The corollary follows from Corollary \ref{coro:opendense} immediately.
	\end{proof}

\section{Structure around minimal orbits}\Label{sec:minimal}
	Let $M$ be a compact connected complex manifold equipped with a maximal action of a compact torus $G$. We assume that each element of $G$ preserves the complex structure $J$ on $M$. As we saw in Section \ref{sec:complexified}, the action of $G$ on $M$ extends to an action of $G^\C$ and of $G^M$ because any vector fields are complete. 
	\begin{lemm}\Label{lemm:minimal}
		Let $M$ be a connected complex manifold equipped with a maximal action of a compact torus $G$. Suppose that each element of $G$ preserves the complex structure on $M$. Then, each minimal orbit is a complex submanifold of $M$. 
	\end{lemm}
	\begin{proof}
		Let $G\cdot x$ be a minimal orbit. Since each element of $G$ preserves the complex structure $J$ on $M$, each element of $G_x$ also preserves $J$. Hence each connected component of fixed points of the action of $G_x$ is a complex submanifold. This together with Lemma \ref{lemm:cfi} shows that $G\cdot x$ is a complex submanifold of $M$, as required.
	\end{proof}
	Let $G\cdot x$ be a minimal orbit of $M$. For any $v \in \mathfrak{g}$, the fundamental vector field $X_v|_{G\cdot x}$ generated by $v$ restricted to $G\cdot x$ is tangent to $G\cdot x$. It follows from Lemma \ref{lemm:minimal} that the vector field $JX_v|_{G\cdot x}$ is also tangent to $G\cdot x$. This together with the definition of complexified action yields that $G\cdot x$ is invariant under the action of $G^\C$ and hence $G\cdot x$ is invariant under the action of $G^M$. Since $G^M$ acts on $G\cdot x$ transitively, $G\cdot x$ is $G^M$-equivariantly biholomorphic to $G^M/(G^M)_x$. In order to describe the complex structure of $G\cdot x$ and its tubular neighborhood, we investigate the global stabilizers $Z_{G^\C}$ and the isotropy subgroup $(G^M)_x$. 
	\begin{lemm}\Label{lemm:stabilizers}
		Let $M$ and $G$ as above. The global stabilizers
		\begin{equation*}
			Z_{G^\C} = \{ g\in G^\C \mid g\cdot x = x \text{ for all $x \in M$}\}
		\end{equation*}
		are connected and have no element of finite order . 
	\end{lemm}
	\begin{proof}
		Any element of finite order in $G^\C = \mathfrak{g}^\C/\Hom (S^1,G)$ is an element of $G$. Since the action of $G$ on $M$ is effective, $Z_{G^\C}$ has no element of finite order. 
		
		Suppose that $Z_{G^\C}$ is not connected. Then, there exists an element $[v] \in Z_{G^\C}$ such that $[v]$ does not sit in the identity component of $Z_{G^\C}$, where $[v]$ denotes the equivalence class of $v \in \mathfrak{g}^\C$. We fix the representative $v \in \mathfrak{g}^\C$ of $[v] \in Z_{G^\C}$. If $v \in \mathfrak{g}$, then the action of $G$ on $M$ is not  effective. Thus $v \notin \mathfrak{g}$. We have a circle action on $M$ via
		\begin{equation*}
			([t],x) \mapsto \exp (tX_v)(x) \quad \text{for $[t] \in \R/\Z$}.
		\end{equation*}
		By definition, the action of $\R/\Z$ on $M$ commute with the action of $G$. So we have the action of the compact torus $\R/\Z \times G$ on $M$. Since $v$ does not sit in $\mathfrak{g}$ and $[v]$ does not sit in the identity component of $Z_{G^\C}$, the action of $\R/\Z \times G$ on $M$ has at most  finite global stabilizers. Thus we have an effective action of a compact torus of dimension $\dim G+1$. This contradicts Lemma \ref{lemm:maximal}. Therefore $Z_{G^\C}$ is connected, proving the lemma.
	\end{proof}
	
	\begin{lemm}\Label{lemm:chart}
		Let $M, G$ and $x$ as above.
		There exists a $G_x$-equivariant biholomorphism $\varphi _x : U \to \widetilde{U}$, where $U$ is a $G_x$-invariant open subset containing $x \in M$ and $\widetilde{U}$ is a $G_x$-invariant open subset of $T_xM$ (with the natural complex structure) containing the origin.
	\end{lemm}
	\begin{proof}
		This can be shown by an almost same argument as \cite[Proof of Lemma 5]{Ishida-Karshon}. For convenience of the reader, we give a proof of the lemma here.
		
		Suppose that the complex dimension of $M$ is $n$. Let $\varphi : U \to \widetilde{U} \subseteq \C^n$ be a holomorphic local chart near $x$ with $\varphi (x)=0$. Identifying $\C^n$ with $T_xM$ via the differential
	\begin{equation*}
		(d\varphi)_x : T_xM \to T_0\C^n \cong \C^n,
	\end{equation*}
	we have a biholomorphism
	\begin{equation*}
		\varphi ' \co U \to \widetilde{U}' \subseteq T_xM
	\end{equation*}
	whose differential at $x$ is the identity map on $T_xM$. We want to obtain such a biholomorphism that is also equivariant. 
	
	Set 
	\begin{equation*}
		U' := \bigcap _{g \in G_x}gU.
	\end{equation*}
	By definition of $U'$, $U'$ contains $x$ and hence is not empty. We now show that $U'$ is open. The complement of $U'$ is the image of the closed subset $G_x \times (M\setminus U)$ by the action map $G_x \times M \to M$. Since $M$ is compact, so is $G_x \times (M \setminus U)$. Therefore, $M\setminus U'$ is compact and hence $U'$ is open.
	
	To obtain an equivariant map, we average $\varphi'|_{U'}$: let
	\begin{equation*}
		\widetilde{\varphi} := \int _{g \in G_x}(g\circ \varphi'\circ g^{-1})dg : U'\to T_xM,
	\end{equation*}
	where $dg$ is a Haar measure on $G_x$. By definition of $\widetilde{\varphi}$, the map $\widetilde{\varphi}$ is holomorphic, $G_x$-equivariant and its differential at $x$ is the identity map on $T_xM$, but no longer biholomorphic in general. However, the implicit function theorem tells us that the restriction of $\widetilde{\varphi}$ to some smaller open subset $U''$ containing $x$ is a biholomorphism onto the subset $\widetilde{\varphi}(U'')$ of $T_xM$. The restriction of $\widetilde{\varphi}$ to the $G_x$-invariant open subset $U_x := \bigcap _{g \in G_x}gU''$ is what we wanted. The lemma is proved.
	\end{proof}
	Let $2n$ be the real dimension of $M$ (that is, $n$ is the complex dimension of $M$) and let $m$ be the dimension of $G$. 
	It follows from Lemma \ref{lemm:minimal} that $T_x(G\cdot x)$ is a complex vector space of complex dimension $m-n$ and invariant under the $G_x$-action. So $T_xM/T_x(G\cdot x)$ is a complex $(2n-m)$-dimensional representation of $G_x$. We decompose the representation $T_xM/T_x(G\cdot x)$ into the direct sum of irreducible representations 
	\begin{equation*}
		\C_{\alpha_1}\oplus \dots \oplus \C_{\alpha _{2n-m}},
	\end{equation*}
	where $\C_{\alpha _j}$ denotes the complex $1$-dimensional representation of $G_x$ whose character is $\alpha _j \in \Hom (G_x,S^1)$. 
	The following follows from Lemma \ref{lemm:chart} immediately.
	\begin{coro}\Label{coro:discs}
		There exists a $G_x$-equivariant local holomorphic chart 
		\begin{equation*}
			\varphi _x : U_x \to \D^{m-n}\times\D^{2n-m}
		\end{equation*}
		from an invariant open neighborhood $U_x$ of $x$ to a polydisc $\D^{m-n}\times \D^{2n-m}$, where the former factor $\D^{m-n}$ is a polydisc in $\C ^{m-n}$ with the trivial $G_x$-action and the latter factor $\D^{2n-m}$ is a polydisc in $\bigoplus _{j=1}^{2n-m}\C _{\alpha _j}$.
	\end{coro}

	Let $z_1,\dots, z_{m-n}, w_1,\dots ,w_{2n-m}$ be the standard coordinate functions of $\D^{m-n} \times \D^{2n-m}\subset \C^{m-n} \times \bigoplus _{j=1}^{2n-m}\C _{\alpha _j}$. 
	Corollary \ref{coro:discs} implies that there exists a basis $(v'_1,\dots, v'_{2n-m})$ of $\mathfrak{g}_x$ such that the fundamental vector fields generated by them can be written as 
	\begin{equation}\Label{eq:vectors}
		X_{v'_j}= \sqrt{-1}w_j\frac{\partial}{\partial w_j} \quad \text{for $j=1,\dots, 2n-m$}
	\end{equation}
	through the coordinates $z_1,\dots ,z_{m-n}, w_1,\dots ,w_{2n-m}$. Here, we identify vector fields whose flows preserve the complex structure with holomorphic vector fields via $X \mapsto \frac{1}{2}(X -\sqrt{-1}JX)$.
	\begin{lemm}\Label{lemm:existence-of-slice}
		There exists a $(G_x)^\C$-equivariant local holomorphic chart
		\begin{equation*}
			\varphi_x \co U_x \to \D ^{m-n} \times \bigoplus _{j=1}^{2n-m}\C_{\alpha _j}
		\end{equation*}
		for an invariant open neighborhood $U_x$ of $x$, where $(G_x)^\C$ acts on the former factor $\D^{m-n}$ trivially and on the latter factor $\bigoplus _{j=1}^{2n-m}\C_{\alpha _j}$ as the complexified action on $\bigoplus_{j=1}^{2n-m}\C_{\alpha _j}$ (see Section \ref{sec:complexified}). 
	\end{lemm}
	\begin{proof}
		Let $\varphi _x \co U_x \to \D^{m-n} \times \D ^{2n-m}$ be as in Corollary \ref{coro:discs}. We sweep the second factor $\D ^{2n-m}$ by the complexified action of $(G_x)^\C$. Because $\varphi _x$ is $G_x$-equivariant and holomorphic, it intertwines the restriction to $U_x$ of the vector fields that generate the complexified $(G_x)^\C$-action on $M$ with the restriction to $\D^{m-n}\times \D ^{2n-m}$ of the vector fields that generate the complexified $(G_x)^\C$-action on $\D^{m-n}\times \bigoplus_{j=1}^{2n-m}\C_{\alpha _j}$. This together with the fact that $\varphi _x$ is diffeomorphism between $U_x$ and $\D ^{m-n} \times \D^{2n-m}$ implies that $\varphi_x$ also intertwines the partial flows on $U_x$ and on $\D^{m-n}\times \D^{2n-m}$ that are generated by these vector fields; in particular, it intertwines the domains of definition of these partial flows. 
		
		Let $v$ be an element in $(\mathfrak{g}_x)^\C$. Since $\varphi _x$ intertwines the partial flows on $U_x$ and $\D^{m-n}\times \D^{2n-m}$ that are generated by $X_v$, 
		\begin{equation*}
			\varphi _x \circ \exp (X_v) = \exp ((\varphi _x)_*X_v)\circ \varphi _x
		\end{equation*}
		on a sufficiently small neighborhood of $x$. This means that 
		\begin{equation}\Label{eq:intertwine}
			\varphi _x^{[v]} := \exp(X_v)\circ \varphi _x \circ \exp(X_v)^{-1} = \varphi _x
		\end{equation}
		on a sufficiently small neighborhood of $x$, where $[v] \in (G_x)^\C=(\mathfrak{g}_x)^\C/\Hom (S^1,G_x)$ denotes the equivalence class of $v \in (\mathfrak{g}_x)^\C$. Since both $\varphi_x^{[v]}$ and $\varphi _x$ are holomorphic and \eqref{eq:intertwine}, they coincide on the intersection of the domains of definition for $\varphi _x$ and $\varphi ^{[v]}$ because of the identity theorem. Namely, we have that 
		\begin{equation*}\Label{eq:intertwine2}
			\varphi _x^{[v]} = \varphi _x \quad \text{on $U_x \cap \exp(X_v)(U_x)$}.
		\end{equation*}
		Let $u$ be another element in $(\mathfrak{g}_x)^\C$. Applying the identity theorem again, we have that
		\begin{equation*}
			\varphi _x^{[u]}=\varphi _x^{[v]} \quad \text{on $\exp(X_u)(U_x) \cap \exp (X_v)(U_x)$}.
		\end{equation*}
		We glue all $\varphi _x^{[v]}$ for $[v] \in (G_x)^\C$ and get a holomorphic map
		\begin{equation*}
			\bigcup _{[v] \in (G_x)^\C} \varphi_x^{[v]} \co \bigcup _{[v] \in (G_x)^\C}\exp(X_v)(U_x) \to \D^{m-n}\times \bigoplus _{j=1}^{2n-m}\C _{\alpha _j}.
		\end{equation*}
		Since each $\varphi_x^{[v]}$ is biholomorphic, the map $\bigcup_{[v] \in (G_x)^\C}\varphi _x^{[v]}$ is a local biholomorphism. Moreover, $\bigcup_{[v] \in (G_x)^\C}\varphi _x^{[v]}$ is surjective. In fact, for any point $p \in \D^{m-n}\times \bigoplus _{j=1}^{2n-m}\C _{\alpha _j}$, take $\ell \in \R$ to be
		\begin{equation*}
			e^\ell > \max \{ |w_j(p)| \mid j=1,\dots ,2n-m\}
		\end{equation*}
		and put
		\begin{equation*}
			v = \ell \left( \sum _{j=1}^{2n-m}-\sqrt{-1}{v'_j} \right) \quad \text{where $v'_j$ is as \eqref{eq:vectors}},
		\end{equation*}
		then we have $p \in \varphi_x^{[v]}(U_x)$.
		
		Since $\bigcup_{[v] \in (G_x)^\C}\varphi _x^{[v]}$ is a surjective local biholomorphism, the range $\D^{m-n}\times \bigoplus _{j=1}^{2n-m}\C_{\alpha _j}$ is simply connected and the domain $\bigcup_{[v] \in (G_x)^\C}\exp(X_v)(U_x)$ of definition is connected, the map $\bigcup_{\xi \in (\mathfrak{g}_x)^\C}\varphi _x^{(\xi)}$ is a biholomorphism. By the construction of the map $\bigcup_{[v] \in (G_x)^\C}\varphi _x^{[v]}$, it is $(G_x)^\C$-equivariant. The lemma is proved. 
	\end{proof}
	
	We set a $G^\C$-invariant open subset
	\begin{equation*}
		N(G\cdot x) := \bigcup_{[v]\in G^\C}\exp (X_v)(U_x),
	\end{equation*}
	where $U_x$ is as in Lemma \ref{lemm:existence-of-slice}. $N(G\cdot x)$ is the minimal $G^\C$-invariant open subset which contains $G\cdot x$. Thus, $N(G\cdot x)$ does not depend on a choice of $U_x$ and hence $N(G\cdot x)$ is unique. 
	
	\begin{lemm}\Label{lemm:opendense2}
		Let $N(G\cdot x)$ be as above. Then $N(G\cdot x)$ contains an open dense $G^\C$-orbit. In particular, $M$ contains an open dense $G^\C$-orbit. 
	\end{lemm}
	\begin{proof}
		By Corollary \ref{coro:opendense}, it suffices to show the existence of a point $y \in N(G\cdot x)$ and elements $v_1,\dots ,v_n \in \mathfrak{g}^\C$ such that  $X_{v_1}(y), \dots ,X_{v_n}(y)$ are linearly independent over $\C$.

		Let $\varphi _x : U_x \to \D ^{m-n}\times \bigoplus _{j=1}^{2n-m}\C _{\alpha _j}$ be as in Lemma \ref{lemm:existence-of-slice}. Let $z_1,\dots, z_{m-n}, w_1,\dots ,w_{2n-m}$ be the standard coordinate functions of $\D^{m-n} \times \bigoplus _{j=1}^{2n-m}\C _{\alpha _j}$. Clearly, there exist elements $v_1,\dots ,v_{m-n} \in \mathfrak{g} \subset \mathfrak{g}^\C$ such that 
		\begin{equation*}
			(dz_i(X_{v_j}))(x)=\delta _{ij}
		\end{equation*}
		where $\delta _{ij}$ is the Kronecker delta. Hence we have that a function on $U_x$ defined by 
		\begin{equation*}
			dz_1 \wedge \dots \wedge dz_{m-n}(X_{v_1}, \dots ,X_{v_{m-n}})
		\end{equation*}
		is not zero on an open subset $U'_x$ containing $x$. It follows from Lemma \ref{lemm:existence-of-slice} that there exist elements $v'_1,\dots ,v '_{2n-m} \in \mathfrak{g}_x \subset \mathfrak{g}^\C$ such that the fundamental vector fields are locally represented as
		\begin{equation*}
			X_{v'_j} = \sqrt{-1}w_j\frac{\partial}{\partial w_j}
		\end{equation*}
		for $j=1,\dots ,2n-m$. And then, a function 
		\begin{equation*}
			\begin{split}
			& dz_1\wedge \dots \wedge dz_{m-n} \wedge dw_1\wedge \dots \wedge dw_{2n-m} (X_{v_1},\dots , X_{v_{m-n}},X_{v'_1},\dots ,X_{v'_{2n-m}})\\
			& = \frac{(\sqrt{-1}) ^{2n-m}(m-n)!}{n!}w_1\cdots w_{2n-m}dz_1\wedge \dots \wedge dz_{m-n}(X_{v_1},\dots ,X_{v_{m-n}})
			\end{split}
		\end{equation*}
		is not zero on the open subset
		\begin{equation*}
			U''_x = \{ y \in U'_x\mid w_j(y)\neq 0 \text{ for all $j=1,\dots ,2n-m$}\}.
		\end{equation*}
		This shows that $X_{v_1}(y),\dots ,X_{v_{m-n}}(y),X_{v'_1}(y),\dots , X_{v'_{2n-m}}(y)$ are linearly independent over $\C$ for $y \in U''_x$, proving the lemma.  
	\end{proof}
	Let $\mathfrak{h}$ be the Lie algebra of $Z_{G^\C}$. The Lie algebra $\mathfrak{g}^M$ of $G^M = G^\C/Z_{G^\C}$ is $\mathfrak{g}^\C/\mathfrak{h}$ by definition. It follows from Lemma \ref{lemm:opendense2} that the dimension of $\mathfrak{g}^M$ is equal to the dimension of $M$. 
	
	It follows from Lemma \ref{lemm:existence-of-slice} that $(G_x)^\C$ acts on $M$ effectively. Hence we may regard $(G_x)^\C$ as a subgroup of $G^M$. Also, we may regard $(\mathfrak{g}_x)^\C$ as a Lie subalgebra of $\mathfrak{g}^M$. 
	
	\begin{lemm}\Label{lemm:GM}
	 The vector subspaces $\mathfrak{g}$ and $\sqrt{-1}\mathfrak{g}_x$ span the whole space  $\mathfrak{g}^M$ and $\mathfrak{g} \cap \sqrt{-1}\mathfrak{g}_x = \{0\}$. 
	\end{lemm}
	\begin{proof}
		We shall show that there exist elements $v_1,\dots, v_{2m-2n} \in \mathfrak{g}$, $v'_1,\dots, v'_{4n-2m} \in (\mathfrak{g}_x)^\C$ and a point $y \in M$ such that $X_{v_1}(y),\dots ,X_{v_{2m-2n}}(y), X_{v'_1}(y), \dots , X_{v'_{4n-2m}}(y)$ are linearly independent over $\R$.  Because if we got such elements, then we have that $\mathfrak{g} + \sqrt{-1}\mathfrak{g}_x = \mathfrak{g}^M$. Moreover, thanks to the dimensions of $\mathfrak{g}, \mathfrak{g}_x$ and $\mathfrak{g}^M$, we will have that $\mathfrak{g} \cap \sqrt{-1}\mathfrak{g}_x = \{0\}$. 
		
		We will use an argument similar to the proof of Lemma \ref{lemm:opendense2}. Let $\varphi _x \co U_x \to \D^{m-n} \times \bigoplus _{j=1}^{2n-m}\C_{\alpha _j}$ be as in Lemma \ref{lemm:existence-of-slice}. Let $z_1,\dots ,z_{m-n}, w_1,\dots ,w_{2n-m}$ be the standard holomorphic coordinates of $\D ^{m-n} \times \bigoplus_{j=1}^{2n-m}\C_{\alpha _j}$. Then, 
		\begin{equation*}
			x_i := \frac{z_i+\overline{z_i}}{2}, \quad y_i:=-\sqrt{-1}\frac{z_i-\overline{z_i}}{2}
		\end{equation*}
		for $i=1,\dots ,m-n$ form real coordinates of the first factor $\D ^{m-n}$. 
		\begin{equation*}
			x'_j := \frac{w_j+\overline{w_j}}{2}, \quad y'_j:=-\sqrt{-1}\frac{w_j-\overline{w_j}}{2}
		\end{equation*}
		for $j=1,\dots , 2n-m$ form real coordinates of the second factor $\bigoplus _{j=1}^{2n-m}\C _{\alpha _j}$. Since $\varphi _x(G\cdot x \cap U_x)$ is the subset of points such that $x'_j=y'_j=0$ for all $j$, there exist elements $v_1,\dots ,v _{2m-2n} \in \mathfrak{g}$ such that 
		\begin{equation*}
			(dx_i(X_{v_k}))(x)= \begin{cases}
				1 & \text{if $k=i$},\\
				0 & \text{otherwise}
			\end{cases}
		\end{equation*}
		and 
		\begin{equation*}
			(dy_i(X_{v_k}))(x) = \begin{cases}
				1 & \text{if $k=m-n+i$},\\
				0 & \text{otherwise}
			\end{cases}
		\end{equation*}
		for $i=1,\dots ,m-n$ and $k=1,\dots ,2m-2n$. By Lemma \ref{lemm:existence-of-slice}, there exist elements $v'_1,\dots ,v'_{2n-m} \in \mathfrak{g}_x$ such that 
		\begin{equation*}
			X_{v'_j} = -y'_j\frac{\partial}{\partial x'_j}+x'_j\frac{\partial}{\partial y'_j}
		\end{equation*}
		for $j=1,\dots, 2n-m$, where we represent $X_{v'_j}$ with the real coordinates $x_i,y_i,x'_j,y'_j$. 
		We set $v'_{2n-m+j} = -\sqrt{-1}v'_j$ for $j=1,\dots ,2n-m$. Then, 
		\begin{equation*}
			X_{v'_{2n-m+j}}=x'_j\frac{\partial}{\partial x'_j}+y'_j\frac{\partial}{\partial y'_j}
		\end{equation*}
		and hence we have
		\begin{align*}
			 &dx_1\wedge dy_1\wedge \dots \wedge dx_{m-n}\wedge dy_{m-n} \wedge dx'_1\wedge dy'_1 \wedge \dots \wedge dx'_{2n-m}\wedge dy'_{2n-m}\\
			& \qquad \qquad (X_{v_1},X_{v_{m-n+1}},\dots, X_{v_{m-n}}, X_{v_{2m-2n}}, X_{v'_1},X_{v'_{2n-m+1}},\dots ,X_{v'_{2n-m}},X_{v'_{4n-2m}}) \\
			&= \frac{(-1)^{2n-m}(2m-2n)!}{2n!}|w_1|^2\cdots |w_{2n-m}|^2dx_1\wedge dy_1\wedge \dots \wedge dx_{m-n}\wedge dy_{m-n}\\
			&\qquad \qquad \qquad \qquad \qquad \qquad \qquad \qquad (X_{v_1},X_{v_{m-n+1}},\dots, X_{v_{m-n}}, X_{v_{2m-2n}}).
		\end{align*}
		Since the function 
		\begin{equation*}
			dx_1\wedge dy_1\wedge \dots \wedge dx_{m-n}\wedge dy_{m-n}(X_{v_1},X_{v_{m-n+1}},\dots, X_{v_{m-n}}, X_{v_{2m-2n}})
		\end{equation*}
		takes the value $\frac{1}{(2m-2n)!}$ at $x$, it is not zero on an open subset $U'_x$ containing $x$. Thus, the vectors $X_{v_1}(y),\dots , X_{v_{2m-2n}}(y),X_{v'_1}(y),\dots ,X_{v'_{4m-2n}}(y)$ are linearly independent over $\R$ for $y \in \{y\in U'_x \mid w_j(y)\neq 0 \text{ for all j}\}$. The lemma is proved. 
	\end{proof}
	By Lemma \ref{lemm:GM}, we have a decomposition
	\begin{equation}\Label{eq:decomposition}
		\mathfrak{g}^\C = \mathfrak{g}\oplus\sqrt{-1}\mathfrak{g}_x\oplus\mathfrak{h}
	\end{equation}
	and 
	\begin{equation*}
		\mathfrak{g}^M = \mathfrak{g}\oplus \sqrt{-1}\mathfrak{g}_x.
	\end{equation*}
	
	Now we are in a position to clarify the relation between $(G_x)^\C$ and $(G^M)_x$. Remark that $(G_x)^\C$ can be regarded as a subgroup of $(G^M)_x$ because $(G_x)^\C$ acts on $M$ effectively and fixes $x$. 
	\begin{lemm}\Label{lemm:GxM}
		$(G_x)^\C = (G^M)_x$.
	\end{lemm}
	\begin{proof}
		Let $v \in \mathfrak{g}^M$ satisfy $\exp(X_v)(x)=x$. By Lemma \ref{lemm:GM}, there exist unique $v' \in \mathfrak{g}$ and $v'' \in \mathfrak{g}_x$ such that $v=v'+\sqrt{-1}v''$.  Since $\exp(X_v)(x)=x$, $\exp (X_{v'})\circ \exp(X_{\sqrt{-1}v''})(x)=x$. Thus $\exp(X_{v'})(x)=x$. This together with Lemma \ref{lemm:cfi} shows that $v' \in \mathfrak{g}_x$. Therefore $v \in (\mathfrak{g}_x)^\C$. This shows that $(G_x)^\C=(G^M)_x$, proving the lemma.
	\end{proof}
	Now we can show the following which is an analogue of the slice theorem to our situation. 
	\begin{prop}\Label{prop:slicetheorem}
		Let $\varphi _x \co U_x \to \D^{m-n} \times \bigoplus _{j=1}^{2n-m}\C_{\alpha _j}$ be as in Lemma \ref{lemm:existence-of-slice}. Define $\psi \co G^M \times_{(G_x)^\C}\bigoplus _{j=1}^{2n-m}\C_{\alpha _j} \to N(G\cdot x)$ to be  
		\begin{equation*}
			[g, w] \mapsto g \cdot \varphi _x^{-1}(0, w).
		\end{equation*}
		Then, $\psi$ is well-defined and a $G^M$-equivariant biholomorphism. 
	\end{prop}
	\begin{proof}
		Let $\widetilde{\psi} \co G^M \times \bigoplus _{j=1}^{2n-m}\C_{\alpha _j} \to N(G\cdot x)$ be the map defined by 
		\begin{equation*}
			(g,w) \mapsto g\cdot \varphi_x^{-1}(0,w)
		\end{equation*} 
		for $g \in G^M$ and $w \in \bigoplus _{j=1}^{2n-m}\C_{\alpha _j}$. By definition of $\widetilde{\psi}$, $\widetilde{\psi}$ is $G^M$-equivariant. To show the well-definedness of $\psi$, it suffices to show that 
		\begin{equation*}
			\widetilde{\psi}(1,w)=\widetilde{\psi}(h,h^{-1}\cdot w)
		\end{equation*}
		for all $h \in (G_x)^\C$ and $w \in \bigoplus _{j=1}^{2n-m}\C_{\alpha _j}$ because $\widetilde{\psi}$ is $G^M$-equivariant. Since $\varphi _x$ is $(G_x)^\C$-equivariant, 
		\begin{equation*}
			\widetilde{\psi}(h,h^{-1}\cdot w)=h \cdot \varphi^{-1}_x(0,h^{-1}\cdot w)= \varphi_x^{-1}(0,w)=\widetilde{\psi}(1,w).
		\end{equation*}
		and hence $\psi$ is well-defined. 
		
		Now we show that $\widetilde{\psi}$ is surjective. To see this, we define a map
$\pi \co N(G\cdot x) \to G\cdot x$ to be 
		\begin{equation*}
			y \mapsto \lim _{t\to -\infty}\exp(tX_v)(y),
		\end{equation*}
		where $v$ is the element in $\mathfrak{g}^\C$ defined by 
		\begin{equation*}
			v = \sum _{j=1}^{2n-m}-\sqrt{-1}v'_j \quad \text{and $v'_j$'s are as in (\ref{eq:vectors})}.
		\end{equation*}
		Clearly, $\pi$ is surjective and $G^M$-equivariant. In particular, $\pi^{-1}(G \cdot x) = N(G \cdot x)$. We will show that $\pi^{-1}(g \cdot x) = \widetilde{\psi}(g, \bigoplus _{j=1}^{2n-m}\C_{\alpha_j})$ for all $g \in G^M$. 
		Suppose that $\pi \circ \widetilde{\psi}(g,w) =x$. Then by definition of $\widetilde{\psi}$, $\pi (g\cdot \varphi_x^{-1}(0,w)) = x$. Since $\pi$ is $G^M$-equivariant, we have that $x = g \cdot \pi(\varphi_x^{-1}(0,w))$. Since $\varphi_x$ is $(G_x)^\C$-equivariant biholomorphic map, we have that $x = g \cdot x$ and hence $g \in (G_x)^\C$. Since $\widetilde{\psi}(g,w) = \widetilde{\psi}(1,g\cdot w)$ for $g \in (G_x)^\C$, we have that $\pi^{-1}(x) = \widetilde{\psi}(1,\bigoplus _{j=1}^{2n-m}\C_{\alpha_j})$. We shall show that $g\cdot \pi^{-1}(x) = \pi^{-1}(g\cdot x)$ for $g \in G^M$. 
		Take $y \in \pi^{-1}(g\cdot x)$. Since $\pi$ is $G^M$-equivariant, $\pi (g^{-1}\cdot y)=g^{-1}\cdot \pi (y)=g^{-1}g\cdot x=x$. This shows that $g^{-1}\cdot y \in \pi ^{-1}(x)$. Thus, $g^{-1}\cdot \pi^{-1}(g\cdot x) \subseteq \pi ^{-1}(x)$ and hence $\pi^{-1}(g \cdot x) \subseteq g \cdot \pi ^{-1}(x)$. Conversely, for $y \in \pi^{-1}(x)$, $\pi(g \cdot y) = g\cdot \pi(y) = g\cdot x$ and hence $\pi^{-1}(g \cdot x) \supseteq g \cdot \pi^{-1}(x)$. Therefore $\pi^{-1}(g \cdot x ) = g \cdot \pi^{-1}(x)$. This together with that $\pi^{-1}(x) = \widetilde{\psi}(1,\bigoplus_{j=1}^{2n-m}\C_{\alpha_j})$ shows that 
		\begin{equation*}
			\pi^{-1}(g \cdot x) = g \cdot \pi^{-1}(x) = g \cdot \widetilde{\psi}\left(1,\bigoplus _{j=1}^{2n-m}\C_{\alpha_j}\right)=\widetilde{\psi}\left(g,\bigoplus _{j=1}^{2n-m}\C_{\alpha_j}\right)
		\end{equation*}
		for all $g \in G^M$, showing that $\widetilde{\psi}$ is surjective. 
		
		Now we show that if $\widetilde{\psi}(g_1,w_1)=\widetilde{\psi}(g_2,w_2)$ then there exists $h \in (G_x)^\C$ such that $g_2=g_1h$ and $w_2=h^{-1}\cdot w_1$. Suppose that $\widetilde{\psi}(g_1,w_1)=\widetilde{\psi}(g_2,w_2)$. Then, $\pi (\widetilde{\psi}(g_1,w_1))=\pi (\widetilde{\psi}(g_2,w_2))$. Thus, $g_1\cdot x=g_2\cdot x$ and hence $h:=g_1^{-1}g_2 \in (G_x)^\C$. And we have that 
		\begin{equation*}
			\widetilde{\psi}(g_1,w_1)=\widetilde{\psi}(g_1h,w_2).
		\end{equation*}
		Since $\widetilde{\psi}$ is $G^M$-equivariant and $h \in (G_x)^\C$, 
		\begin{equation}\Label{eq:psitilde}
			\widetilde{\psi}(1,w_1)=g_1^{-1}\cdot \widetilde{\psi}(g_1,w_1)=g_1^{-1}\cdot \widetilde{\psi}(g_1h,w_2)=\widetilde{\psi}(h,w_2)=\widetilde{\psi}(1,h\cdot w_2).
		\end{equation}
		By applying $\varphi _x$ to \eqref{eq:psitilde}, we have that $w_1=h\cdot w_2$. Thus, $w_2=h^{-1}\cdot w_1$.  This together with the well-definedness of $\psi$ tells us that $\widetilde{\psi}$ descends to a holomorphic injective map $\psi$. 
		
		Since $\widetilde{\psi}$ is surjective $G^M$-equivariant, so is $\psi$. So we have that $\psi$ is a holomorphic bijective map. Because an injective holomorphic map is a biholomorphic map into its image, $\psi$ is a $G^M$-equivariant biholomorphism, proving the proposition.
	\end{proof}

\section{Obtaining fans and their characterization}\Label{sec:fan}
	Let us recall our setting in Section \ref{sec:minimal}; $M$ is a compact connected manifold of complex dimension $n$ equipped with a maximal action of a compact torus $G$ of real dimension $m$. For such a manifold, we have shown that there exists a ``nice" tubular neighborhood $N(G\cdot x)$ of a minimal orbit $G\cdot x$ (see Proposition \ref{prop:slicetheorem}). 

	We set
	\begin{equation*}
		N := \bigcup_{G\cdot x : \text{ minimal orbit}}N(G\cdot x).
	\end{equation*}
	By definition of $N$, $N$ is a $G^M$-invariant open dense submanifold of $M$ because $N$ contains the open dense orbit of $G^M$ (we will show that $N$ is the whole manifold $M$ in Section \ref{sec:maintheo}). For any subgroup $G'$ of $G \subset G^M$, each connected component of fixed point set $N^{G'}$ is a complex submanifold of $N$ because $G'$ is compact and its action preserves the complex structure on $N$. 

	\begin{defn}
		A \emph{characteristic submanifold} $N'$ of $N$ is a complex $1$-codimensional connected component of fixed points $N^{G_{N'}}$ for some $1$-dimensional subtorus $G_{N'}$ of $G$.
	\end{defn}
	
	\begin{rema}In case when $m=2n$, $M$ has no characteristic submanifold because $G$ acts on $M$ simply transitively and hence $M$ is a compact complex torus. In this case, all assertions in this section are tautological or do not have any sense. From here until the end of this section, we assume that $m<2n$ unless otherwise stated.
	\end{rema}
	
	To clarify the characteristic submanifolds, we shall see the isotropy subgroup at a point in $N(G\cdot x)$. By Proposition \ref{prop:slicetheorem}, each $N(G\cdot x)$ is $G^M$-equivariantly biholomorphic to 
	\begin{equation*}
		G^M \times_{(G_x)^\C} \bigoplus _{j=1}^{2n-m}\C_{\alpha _j}.
	\end{equation*}
	For $[g,(z_1,\dots ,z_{2n-m})] \in G^M\times _{G_x^{\C}}\bigoplus _{j=1}^{2m-n}\C_{\alpha _j}$ and $g' \in G$, if 
	\begin{equation*}
		g'\cdot [g,(z_1,\dots ,z_{2m-n})]=[g,(z_1,\dots ,z_{2m-n})]
	\end{equation*}
	then $g'$ should be an element of $G_x$.
	Suppose $g' \in G_x$. Then, 
	\begin{equation*}
		g'\cdot [g,(z_1,\dots ,z_{2n-m})]= [g,(\alpha _1(g'^{-1})z_1, \dots ,\alpha _{2m-n}(g'^{-1})z_{2n-m})].
	\end{equation*}
	Therefore, the isotropy subgroup at $[g,(z_1,\dots ,z_{2n-m})]$ of $G$ is 
	\begin{equation*}
		\{ g' \in G_x \mid \alpha _j(g')=1 \text{ if $z_j\neq 0$}\}.
	\end{equation*}
	Conversely, for any subgroup $F$ of $G_x$, the fixed point set $N(G\cdot x)^F$ is $G^M$-equivariantly biholomorphic to 
	\begin{equation*}
		\{ [g,(z_1,\dots ,z_{2n-m})] \mid z_j=0 \text{ if there is some $g' \in F$ such that $\alpha _j(g') \neq 1$}\}.
	\end{equation*}
	So $N(G\cdot x)^F$ has complex codimension $1$ if and only if $\alpha _j|_F =1$ for all $j$ except one $j_0$.
	Recall that $\bigoplus _{j=1}^{2m-n} \C _{\alpha _j} \cong T_xM/T_x(G\cdot x)$ is a faithful $G_x$-representation. Since $\bigoplus _{j=1}^{2n-m} \C _{\alpha _j}$ is faithful, we have an isomorphism $(\alpha _1,\dots ,\alpha _{2n-m}) \co G_x \to (S^1)^{2n-m}$ via 
	\begin{equation*}
		(\alpha _1,\dots ,\alpha _{2n-m})(g)=(\alpha _1(g),\dots ,\alpha _{2n-m}(g)).
	\end{equation*}
	Now we suppose that the subgroup $F$ is $1$-dimensional subtorus of $G$ and consider the codimension of $N(G\cdot x)^F$. Let $T_j$ be the $j$-th $1$-dimensional coordinate subtorus of $(S^1)^{2n-m}$. The observation above implies that $N(G\cdot x)^F$ has complex codimention $1$ if and only if $F=(\alpha _1,\dots ,\alpha _{2n-m})^{-1}(T_j) \subset G_x \subset G$ for some $j=1,\dots ,2n-m$. It turns out that $N(G\cdot x)$ intersects exactly $2n-m$ characteristic submanifolds of $N$. It follows from the compactness of $M$ and Lemma \ref{lemm:cfi} that there are at most finitely many characteristic submanifolds of $N$. 
	
	Let $N_1,\dots ,N_k$ be the characteristic submanifolds of $X$ and let $G_i$ be the $1$-dimensional subtorus of $G$ which fixes $N_i$ pointwise for each $i$. We define a simplicial complex $\Sigma$ on the vertex set $\{1,\dots ,k\}$ and vectors $\lambda _i \in \Hom (S^1,G) \subset \mathfrak{g}$ as follows.  
	We set 
	\begin{equation*}
		\Sigma = \left\{ I \subset \{1,\dots ,k\} \mid \bigcap _{i \in I}N_i \neq \emptyset\right\}.
	\end{equation*}
	We assign a homomorphism $\lambda _i \co S^1 \to G_i$ to each characteristic submanifold $N_i$ to be 
	\begin{equation}\Label{eq:lambda}
		(\lambda _i(h))_*(v)=hv,
	\end{equation}
	where $v$ is any normal vector $TN|_{N_i}/T{N_i}$ and $h \in S^1$. 
	
	\begin{lemm}\Label{lemm:unimodular}
		For each $I \in \Sigma$, $(\lambda _i)_{i\in I}$ is a part of a $\Z$-basis of $\Hom (S^1,G)$. 
	\end{lemm}
	\begin{proof}
		Let $I \in \Sigma$. By definition of $\Sigma$, $\bigcap _{i \in I}N_i$ is nonempty. Let $p \in \bigcap _{i \in I}N_i$. By definition of $N$, there exists a minimal orbit $G\cdot x$ such that $N(G\cdot x) \ni p$. Let $\psi \co G^M \times _{(G_x)^\C}\bigoplus _{j=1}^{2n-m}\C_{\alpha _j} \to N(G\cdot x)$ be as in Proposition \ref{prop:slicetheorem}. By definition of $\lambda _i$, there exists an injection $\rho \co I \to \{1,\dots ,2n-m\}$ such that 
		\begin{equation}\Label{eq:alphalambda}
			\alpha _j\circ \lambda_i (h)=\begin{cases}
				1 & \text{if $\rho (i) \neq j$},\\
				h & \text{if $\rho (i) = j$}.
			\end{cases} 
		\end{equation}
		Since $(\alpha _1,\dots ,\alpha _{2n-m}) \co G_x \to (S^1)^{2n-m}$ is an isomorphism, $(\alpha_1,\dots ,\alpha _{2n-m})$ is a $\Z$-basis of $\Hom (G_x,S^1)$. Thus, $(\lambda _i)_{i\in I}$ is a part of the dual basis of $\alpha _1,\dots ,\alpha _{2n-m}$ and hence it is a part of a $\Z$-basis of $\Hom (S^1,G_x)$. Since $G_x$ is a subtorus of $G$, $\Hom (S^1,G_x)$ is a direct summand of $\Hom (S^1,G)$. Therefore, $(\lambda _i)_{i \in I}$ is a part of $\Z$-basis of $\Hom (S^1,G)$, proving the lemma. 
	\end{proof}
	By Lemma \ref{lemm:unimodular}, for each $I \in \Sigma$, the cone 
	\begin{equation*}
		C_I:= \pos (\lambda _i \mid i \in I) :=\left\{ \sum _{i \in I}a_i \lambda _i\mid a_i \geq 0\right\} \subset \mathfrak{g}
	\end{equation*}
	in $\mathfrak{g}$ spanned by $\lambda _i$ is nonsingular with respect to the lattice $\Hom (S^1,G)$. We will see that  the collection $\Delta =\{C_I \mid {I \in \Sigma}\}$ of nonsingular cones becomes a fan later. 
	\begin{rema}
		In case $m=2n$, there is no characteristic submanifold and hence $\Sigma =\{\emptyset\}$ and $C_\emptyset =\{0\}$. Thus, $\Delta$ is a fan consisting of only the origin $\{0\}$ in $\mathfrak{g}$. 
	\end{rema}

	As well as toric varieties, Hausdorff-ness of $N$ tells us that the collection $\Delta$ of nonsingular cones does not have overlaps, that is, $\Delta$ is a fan. 
	Let $y$ be a point sitting in the free $G^M$-orbit in $N$. For each $v \in \mathfrak{g}$, we consider the curve 
	\begin{equation*}
		c_y^v \co \R \to N
	\end{equation*}
	that is given by 
	\begin{equation*}
		c_y^v (r) := \exp (-rJX_v) (y) \quad \text{for $r \in \R$}.
	\end{equation*}
	For each $I \in \Sigma$, we set 
	\begin{equation*}
		N_I := \bigcap _{i \in I}N_i,
	\end{equation*}
	\begin{equation*}
		N_I^0 := \bigcap _{i\in I}N_i \setminus \bigcup_{j \notin I}N_j
	\end{equation*}
	and 
	\begin{equation*}
		C_I^0 := \left\{ \sum_{i \in I}a_i\lambda _i \mid a_i >0\right\} \subset \mathfrak{g}.
	\end{equation*}
	\begin{lemm}\Label{lemm:limit}
		The curve $c_y^v(r)$ converges to a point in $N_I^0$ as $r$ approaches to $-\infty$ if and only if $v \in C_I^0$. Moreover, in this case the limit point $q'$ belongs to $N(G\cdot x)$ for every minimal orbit $G\cdot x$ such that $N(G\cdot x) \cap N_I \neq \emptyset$. 
	\end{lemm}
	\begin{proof}
		Suppose that $v \in C_I^0$. By definition of $\Sigma$, $N_I$ is nonempty. By renumbering characteristic submanifolds, we may assume that $I=\{1,\dots ,\ell\}$ without loss of generality. Let $G\cdot x$ be a minimal orbit such that $N(G\cdot x)$ meets $N_1,\dots ,N_{\ell}$. We also assume that characteristic submanifolds $N_{\ell+1}, \dots, N_{2n-m}$ also intersects with $N(G \cdot x)$. By Lemma \ref{lemm:unimodular}, $\lambda _1,\dots ,\lambda _{2n-m}$ form a basis of $\Hom (S^1,G_x)$ because $\dim G_x=2n-m$. 	Let $\psi \co G^M \times _{(G_x)^\C} \bigoplus _{j=1}^{2n-m}\C_{\alpha _j} \to N(G\cdot x)$ be as in Proposition \ref{prop:slicetheorem}. By renumbering $\alpha_j$'s, we may assume that $(\alpha _1,\dots ,\alpha _{2n-m})$ is the dual basis of $(\lambda _1,\dots ,\lambda _{2n-m})$. Then, 
		\begin{equation*}
			\psi^{-1}(N(G\cdot x)\cap N_j)= \left\{ [g,w_1,\dots ,w_{2n-m}] \in G^M \times _{(G_x)^\C} \bigoplus _{j=1}^{2n-m}\C_{\alpha _j} \mid w_j=0 \right\}
		\end{equation*}
		and hence 
		\begin{equation*}
			\psi^{-1}(N(G\cdot x)\cap N_I^0) =  \left\{ [g,w_1,\dots ,w_{2n-m}] \mid w_j\begin{cases}
				=0 & \text{for $j=1,\dots, \ell$},\\
				\neq 0 & \text{for $j=\ell+1,\dots ,2n-m$}
			\end{cases} \right\}.
		\end{equation*}
		Suppose that $\psi ^{-1}(y)$ is represented as 
		\begin{equation*}
			[g,y_1,\dots ,y_{2n-m}] \in G^M \times _{(G_x)^\C} \bigoplus _{j=1}^{2n-m}\C_{\alpha _j}.
		\end{equation*}
		Then, 
		\begin{equation*}
			\begin{split}
				\psi^{-1}(c_y^v(r)) &= [-r\sqrt{-1}v]\cdot [g,y_1,\dots ,y_{2n-m}]\\
				&= [g-[r\sqrt{-1}v],y_1,\dots ,y_{2n-m}]\\
				&= [g,e^{2\pi r\langle \alpha _1,v\rangle}y_1, \dots ,e^{2\pi r\langle \alpha _{2n-m},v\rangle}y_{2n-m}],
			\end{split}
		\end{equation*}
		where $[-\sqrt{-1}v] \in (G_x)^\C=(\mathfrak{g}_x)^\C/\Hom (S^1,G_x)$ denotes the equivalence class of $-\sqrt{-1}v \in \sqrt{-1}\mathfrak{g}_x$. Since $v \in C_I^0$, the pairing $\langle \alpha_j,v\rangle$ is positive for $j=1,\dots ,\ell$ and $0$ for $j=\ell +1,\dots ,2n-m$. Since $y$ sits in the free $G^M$-orbit, $y_1,\dots ,y_{2n-m}$ are not zero. Therefore,
		\begin{equation*}
			\lim _{r \to -\infty}\psi^{-1}(c_y^v(r)) = [g,0,\dots ,0,y_{\ell +1},\dots ,y_{2n-m}] 
		\end{equation*}
		and hence $\lim_{r\to -\infty}c_y^v(r) \in N_I^0$.
		
		Suppose that the curve $c_y^v(r)$ converges to a point $y'$ in $N_I^0$. Let $G\cdot x$ be a minimal orbit such that $N(G\cdot x)$ contains $y'$. As before, we may assume that $I=\{1,\dots ,\ell\}$ and the characteristic submanifolds that meet $N(G\cdot x)$ are exactly $N_1,\dots ,N_{2n-m}$. We also sort $\alpha_j$'s so that $\alpha _1,\dots ,\alpha_{2n-m}$ form the dual basis of $\lambda_1,\dots ,\lambda_{2n-m}$. We show that $v$ should be in $\mathfrak{g}_x$ first. Suppose that $\psi^{-1}(y')$ is represented as 
		\begin{equation*}
			[g',0,\dots, 0,y'_{\ell+1},\dots ,y'_{2n+m}] \in G^M \times _{(G_x)^\C} \bigoplus _{j=1}^{2n-m}\C_{\alpha _j}.
		\end{equation*}
		We consider the projection $\pi \co G^M \times _{(G_x)^\C} \bigoplus _{j=1}^{2n-m}\C_{\alpha _j} \to G^M/(G_x)^\C$ (this $\pi$ is essentially the same as $\pi$ in the proof of Proposition \ref{prop:slicetheorem}). We have that 
		\begin{equation*}
			\pi \circ \psi^{-1}(\exp(-rX_v)(y')) = [g' -[\sqrt{-1}v]],
		\end{equation*}
		so $\sqrt{-1}v$ should be in $(\mathfrak{g}_x)^\C+\mathfrak{h}$. Otherwise, it contradicts the assumption that $q'$ is the limit point. Since $(\mathfrak{g}_x)^\C +\mathfrak{h}$ is a complex vector space, $\sqrt{-1}v \in (\mathfrak{g}_x)^\C + \mathfrak{h}$ implies that $v \in (\mathfrak{g}_x)^\C+\mathfrak{h}$. For $v \in \mathfrak{g}$, it follows from the decomposition \eqref{eq:decomposition} and $v \in (\mathfrak{g}_x)^\C + \mathfrak{h}$ that $v \in \mathfrak{g}_x$. Now we suppose that $\psi^{-1}(y)$ is represented as 
		\begin{equation*}
			[g,y_1,\dots ,y_{2n-m}] \in G^M \times _{(G_x)^\C} \bigoplus _{j=1}^{2n-m}\C_{\alpha _j}
		\end{equation*}
		as before. Then, as we calculated before, 
		\begin{equation*}
				\psi^{-1}(c_y^v(r)) = [g,e^{2\pi r\langle \alpha _1,v\rangle}y_1, \dots ,e^{2\pi r\langle \alpha _{2n-m},v\rangle}y_{2n-m}].
		\end{equation*}
		Since $c_y^v(r)$ converges to $y'$ as $r$ approaches $-\infty$, 
		\begin{equation*}
			\lim _{r \to -\infty} [g,e^{2\pi r\langle \alpha _1,v\rangle}y_1, \dots ,e^{2\pi r\langle \alpha _{2n-m},v\rangle}y_{2n-m}] = [g',0,\dots, 0,y'_{\ell+1},\dots ,y'_{2n-m}].
		\end{equation*}
		Thus, $\langle \alpha _j,v\rangle$ should be positive for $j=1,\dots, \ell$ and $\langle \alpha _j,v\rangle$ should be $0$ for $j=\ell+1,\dots , 2n-m$.  This shows $v \in C_I^0$, proving the lemma. 
	\end{proof}
	\begin{coro}\Label{coro:fan}
		\begin{enumerate}
			\item For every $I, J \in \Sigma$, if $I \neq J$, then $C_I^0 \cap C_J^0 = \emptyset$. 
			\item For every $I, J \in \Sigma$, 
			\begin{equation*}
				C_I \cap C_J = C_{I\cap J}.
			\end{equation*}
			\item The collection of cones 
			\begin{equation*}
				\Delta = \{ C_I \mid {I \in \Sigma}\}
			\end{equation*}
			which was introduced Lemma \ref{lemm:unimodular} below is a fan, that is, every face of every cone in $\Delta$ is itself in $\Delta$, and the intersection of every two cones in $\Delta$ is a common face. 
		\end{enumerate}
	\end{coro}
	\begin{proof}
		Part (1) follows from Lemma \ref{lemm:limit} because the sets $N_I^0$ are disjoint. Part (3) follows from Part (2). 
		
		For Part (2), we only need to show the inclusion $C_I \cap C_J \subseteq C_{I \cap J}$ because the opposite inclusion is trivial by definition of $C_I$. Let $v \in C_I \cap C_J$. Let $I' \subseteq I$ and $J' \subseteq J$ be the subsets such that $v \in C_{I'}^0$ and $v \in C_{J'}^0$. Then $v \in C_{I'}^0 \cap C_{J'}^0$ and hence $I'=J'$ by Part (1). Therefore $v \in C_{I'}^0 =C_{J'}^0 \subseteq C_{I \cap J}$, proving the corollary.
	\end{proof}
	\begin{lemm}\Label{lemm:codim}
		For every $I \in \Sigma$, the set $N_I$ is a complex submanifold of $N$ of complex codimension $|I|$, it is connected, and it contains a minimal orbit.
	\end{lemm}
	\begin{proof}
		Fix $I \in \Sigma$. 
		
		Because each of the sets $N_i$, for $i \in I$ is closed in $N$, so is the intersection $N_I$ of these sets. 
		
		$N$ is the union of open subsets $N(G\cdot x)$ and every intersection $N(G\cdot x) \cap N_I$ is a $G^M$-invariant complex submanifold of codimension $I$ in $N$. It remains to show that $N_I$ is connected and contains a minimal orbit. 
		
		Choose any $v \in C_I^0$ (for example, we may take $v = \sum_{i \in I}\lambda _i$), and choose any $y$ in the free $G^M$-orbit in $N$. By Lemma \ref{lemm:limit}, the curve $c_y^v(r)$ converges as $r \to -\infty$; let $y'$ be its limit. Also by Lemma \ref{lemm:limit}, for every minimal orbit $G\cdot x$ such that $N(G\cdot x) \cap N_I \neq \emptyset$, the limit point $y'$ belongs to $N(G\cdot x)$. Because $N_I$ is the union over such $G\cdot x$ of the subsets $N(G\cdot x)\cap N_I$ by definition of $N$, and because each of these subsets is connected and contains $y'$, the union $N_I$ is connected. Also, every minimal orbit $G\cdot x$ such that $N(G\cdot x)\cap N_I \neq \emptyset$ belongs to $N(G\cdot x) \cap N_I$; because the set of such $G\cdot x$'s is nonempty, $N_I$ contains a minimal orbit. The lemma is proved. 
	\end{proof}
	\begin{coro}\Label{coro:pure}
		In the fan $\Delta$, every cone is contained in a $2n-m$-dimensional cone. 
	\end{coro}
	\begin{proof}
		Every cone in the fan $\Delta$ has the form $C_I$ for some $I \in \Sigma$. By Lemma \ref{lemm:codim}, the set $N_I$ contains a minimal orbit;  let $G\cdot x$ be such a minimal orbit. As we saw in Lemma \ref{lemm:unimodular} above, $N(G\cdot x)$ intersects exactly $2n-m$ characteristic submanifolds which contain $G\cdot x$, say $N_j$ for $J \subset \{1,\dots ,k\}$ with $|J|=2n-m$. Then $J \in \Sigma$, and $C_J$ is a $2n-m$-dimensional cone in $\Delta$ that contains $C_I$, proving the corollary. 
	\end{proof}
	Now we consider the quotient map $N \to N/G$ and compose it with $c_q^v$. Then, we have a curve 
	\begin{equation*}
		\overline{c}_y^v \co \R \to N/G.
	\end{equation*}
	For $v_1,v_2 \in \mathfrak{g}$, it follows immediately that $\overline{c}_y^{v_1}=\overline{c}_y^{v_2}$ if and only if $v_1-v_2 \in \sqrt{-1}\mathfrak{g}+\mathfrak{h}$. This motivates us to define an $\R$-linear map
	\begin{equation*}
		\overline{J} \co \mathfrak{g} \to \mathfrak{g}^M /\mathfrak{g} \quad \text{by $\overline{J}(v) = [\sqrt{-1}v]$},
	\end{equation*}
	where $[\sqrt{-1}v] \in \mathfrak{g}^M/\mathfrak{g}$ denotes the equivalence class of $\sqrt{-1}v \in \mathfrak{g}^M$. The kernel of $\overline{J}$ is $(\sqrt{-1}\mathfrak{g} + \mathfrak{h}) \cap \mathfrak{g}$. 
	\begin{lemm}\Label{lemm:Jbar}
		Let $\overline{J}$ be as above. 
		\begin{enumerate}
		\item $\overline{J}$ sends a cone $C_I$ to a cone of the same dimension. 
		\item The image of the fan $\Delta$ by $\overline{J}$ is also a fan, that is, the collection of cones $\{ \overline{J}(C_I) \mid {I \in \Sigma}\}$ is a fan in $\mathfrak{g}^M/\mathfrak{g}$. 
		\end{enumerate}
	\end{lemm}
	\begin{proof}
		For Part (1), by the fact that  $\overline{J}$ is $\R$-linear and Corollary \ref{coro:pure}, we only need to show that the cone $\overline{J}(C_I)$ such that $|I|=2n-m$ has dimension $2n-m$. Suppose $|I| =2n-m$. Then, $N_I$ is a minimal orbit because $N_I$ contains a minimal orbit by Lemma \ref{lemm:codim} and $\dim N_I = 2n-2(2n-m)=2m-2n$. Let $x \in N_I$. Then, $(\lambda _i)_{i \in I}$ is a basis of $\mathfrak{g}_x$ over $\R$ (see Lemma \ref{lemm:unimodular}). Since the cone $C_I$ is spanned by $\lambda _i$ for $i \in I$, it suffices to show that the restriction $\overline{J}|_{\mathfrak{g}_x}$ is injective. Since $\mathfrak{g}^M = \mathfrak{g}\oplus \sqrt{-1}\mathfrak{g}_x$ by the decomposition \eqref{eq:decomposition}, $\overline{J}|_{\mathfrak{g}_x}$ is an isomorphism, proving Part (1).
		
		Suppose that $\overline{J}(C_I^0) \cap \overline{J}(C_J^0) \neq \emptyset$ and $I \neq J$. Then, there exist vectors $u\in C_I^0$ and $v \in C_J^0$ such that $\overline{J}(u)=\overline{J}(v)$, that is, $u-v \in \sqrt{-1}\mathfrak{g}$. Since $u-v \in \sqrt{-1}\mathfrak{g}$, the curves $\overline{c}_y^u(r)$ and $\overline{c}_y^v(r)$ coincide. Since the curves $c_y^u(r)$ and $c_y^v(r)$ converge to points in $N_I^0$ and $N_J^0$ as $r$ approaches $-\infty$ respectively by Lemma \ref{lemm:limit} and $N/G$ is Hausdorff by the compactness of $G$, the curves $\overline{c}_y^u(r)$ and $\overline{c}_y^v(r)$ also converge to points in $N_I^0/G$ and $N_J^0/G$ respectively. This contradicts the fact that the curve $\overline{c}_y^u(r)$ coincides with $\overline{c}_y^v(r)$. Therefore $\overline{J}(C_I^0) \cap \overline{J}(C_J^0) = \emptyset$ for $I \neq J$. Using the same argument as in the proof of Corollary \ref{coro:fan}, Part (2) is proved.
	\end{proof}
	\begin{lemm}\Label{lemm:complete}
		The fan $\overline{J}(\Delta ):=\{\overline{J}(C_I) \mid {I \in \Sigma}\}$ is a complete fan in $\mathfrak{g}^M/\mathfrak{g}$. 
	\end{lemm}
	We will deduce Lemma \ref{lemm:complete} from the compactness of $M$. We begin with proving a special case:
	\begin{lemm}\Label{lemm:specialcase}
		Let $M'$ be a compact connected complex manifold of real dimension $2n$ equipped with a maximal action of a compact torus $G'$ of dimension $2n-1$. Then, $M'$ consists of exactly two minimal orbits and one free ${G'}^{M'}$-orbit. 
	\end{lemm}
	\begin{proof}
		$M'$ contains exactly one free ${G'}^{M'}$-orbit by Lemma \ref{lemm:opendense2}. 
	
		Since $M'$ is a complex manifold, $M'$ is orientable. Consider the dimension of each orbit $G'\cdot x$. By the inequality \eqref{eq:inequality3}, 
		\begin{equation*}
			2(2n-1)-2n =2n-2 \leq \dim G'\cdot x.
		\end{equation*}
		Since $\dim G' =2n-1$, we have that $\dim G'\cdot x=2n-2$ or $2n-1$. 
		
		In case when $\dim G'\cdot x=2n-1$, the orbit $G'^{M'} \cdot x$ has dimension $2n = \dim M'$. Therefore, $x$ is contained in the free $G'^{M'}$-orbit. 		
		In case when $\dim G' \cdot x =2n-2$, the orbit $G' \cdot x$ is a minimal orbit. 
		
		The quotient $M'/G'$ is a real $1$-manifold with boundary; its boundary is exactly the image of the minimal orbits by the quotient map. Because $M'$ is compact and contains at least one minimal orbit, and by the classification of $1$-manifolds, the quotient $M'/G'$ must be a closed segment. Because a closed segment has the boundary which consists of exactly $2$ distinct points, $M'$ contains exactly $2$ minimal orbits, proving the lemma. 
	\end{proof}
	\begin{lemm}\Label{lemm:pseudomfd}
		Every $2n-m-1$-dimensional cone in $\Delta$ is contained in exactly two $2n-m$-dimensional cones in $\Delta$. 
	\end{lemm}
	\begin{proof}
		Let $C_I$ be a $2n-m-1$-dimensional cone in $\Delta$. 
		Let $G_I$ be the subtorus of $G$ that is generated by the circles $G_i $ for $i \in I$. By definition of $G_I$, $\dim G_I = 2n-m-1$. By Lemma \ref{lemm:codim}, $N_I$ is a connected complex manifold of complex dimension $m-n+1$, equipped with an effective action of the compact torus $G/G_I$ of dimension $2m-2n+1$, which preserves the complex structure on $N_I$. We will now show that $N_I$ is compact, and will deduce Lemma \ref{lemm:pseudomfd} from Lemma \ref{lemm:specialcase}. 
		
		First note that $N_I$ is a connected component of the fixed point set of $G_I$ in $N$. This follows from the facts that $N_I$ is connected by Lemma \ref{lemm:codim} and that, for each of the subsets $N(G\cdot x)$, if $N(G\cdot x) \cap N_I \neq 0$ then it is a connected component of the fixed point set of $G_I$ in $N(G\cdot x)$. Let $M_I$ denote the connected component of the fixed point set of $G_I$ in $M$ that contains $N_I$. By definition, $M_I$ is a complex manifold equipped with an effective action of the compact torus $G/G_I$ which preserves the complex structure on $M_I$. By Lemma \ref{lemm:codim}, $N_I$ contains a minimal orbit $G\cdot x$ with respect to the action of $G$ on $M$. Since $\dim G\cdot x= 2\dim G -\dim M$, $\dim G\cdot x= 2\dim G/G_I - \dim M_I$ and hence the action of $G/G_I$ on $M_I$ is maximal. By Lemma \ref{lemm:specialcase}, $M_I$ contains exactly two minimal orbits, say $G\cdot x$ and $G\cdot x'$. The intersection $N(G\cdot x) \cap M_I$, being a $G^M$-invariant tubular neighborhood of $G\cdot x$, must be all of $M_I \setminus G\cdot x'$, because $M^I$ consists of three $G^M$-orbits by Lemma \ref{lemm:specialcase}. Similarly, $N(G\cdot x') \cap M_I$ is all of $M_I \setminus G\cdot x$. This implies that $N_I=M_I$. 
		
		$N_I$ contains exactly two minimal orbits $G\cdot x$ and $G\cdot x'$. The minimal orbit $G\cdot x$ (respectively, $G\cdot x'$) is the intersection of exactly $2n-m$ characteristic submanifolds $N_i$ for $i \in I$ and $N_{i_0}$ for some $i_0 \notin I$ (respectively, $i_1 \notin I$). Therefore, $I \cup \{i_0\} \in \Sigma$ and $I \cup \{i_1\} \in \Sigma$ and hence $C_I \subset C_{I \cup \{ i_0\}} \in \Delta$ and $C_I \subset C_{I \cup \{i_1\}} \in \Delta$. The lemma is proved. 
	\end{proof}
	Now we are ready to prove Lemma \ref{lemm:complete}.
	\begin{proof}[Proof of Lemma \ref{lemm:complete}]
		Let $|\overline{J}(\Delta)|$ denote the union of the cones in $\overline{J}(\Delta ):= \{ \overline{J}(C_I)\mid {I \in \Sigma}\}$, and let $|\overline{J}(\Delta )^{2n-m-2}|$ denote the union of the cones in $\overline{J}(\Delta)$ that have codimension $\geq 2$. The complement $\mathfrak{g}^M/\mathfrak{g}\setminus |\overline{J}(\Delta )^{2n-m-2}|$ is connected, open and dense in $\mathfrak{g}^M/\mathfrak{g}$. 
		
		By Lemma \ref{lemm:pseudomfd}, the union of the relative interiors of the cones of $\overline{J}(\Delta)$ of dimension $2n-m$ and of $2n-m-1$ is open in $\mathfrak{g}^M/\mathfrak{g}$. The union is $|\overline{J}(\Delta )| \setminus |\overline{J}(\Delta )^{2n-m-2}|$. Thus, $|\overline{J}(\Delta )| \setminus |\overline{J}(\Delta )^{2n-m-2}|$ is also open in $\mathfrak{g}^M/\mathfrak{g}\setminus  |\overline{J}(\Delta )^{2n-m-2}|$. 
		
		But because $|\overline{J}(\Delta )|$ is closed in $\mathfrak{g}^M/\mathfrak{g}$, we also have that $|\overline{J}(\Delta )| \setminus |\overline{J}(\Delta )^{2n-m-2}|$ is closed in $\mathfrak{g}^M/\mathfrak{g}\setminus  |\overline{J}(\Delta )^{2n-m-2}|$.
		
		Because $|\overline{J}(\Delta )| \setminus |\overline{J}(\Delta )^{2n-m-2}|$ is open and closed in $\mathfrak{g}^M/\mathfrak{g}\setminus  |\overline{J}(\Delta )^{2n-m-2}|$ and $\mathfrak{g}^M/\mathfrak{g}\setminus  |\overline{J}(\Delta )^{2n-m-2}|$ is connected, we have that $|\overline{J}(\Delta )| \setminus |\overline{J}(\Delta )^{2n-m-2}|$ is either empty or is equal to all of $\mathfrak{g}^M/\mathfrak{g}\setminus  |\overline{J}(\Delta )^{2n-m-2}|$. 
		
		Because, by assumption, $M$ has a minimal orbit, $\Delta$ has at least one $2n-m$-dimensional cone, $\overline{J}(\Delta)$ has at least one $2n-m$-dimensional cone, so $|\overline{J}(\Delta)| \setminus |\overline{J}(\Delta)^{2n-m-2}|$ is not empty. So $|\overline{J}(\Delta)| \setminus |\overline{J}(\Delta)^{2n-m-2}|$ is equal to all of $\mathfrak{g}^M/\mathfrak{g}\setminus  |\overline{J}(\Delta )^{2n-m-2}|$. Taking the closures in $\mathfrak{g}^M/\mathfrak{g}$, we deduce that $|\overline{J}(\Delta)|=\mathfrak{g}^M/\mathfrak{g}$, as required. 
	\end{proof}
	Using $\overline{J}$, we can characterize the fan $\Delta$ and the Lie subalgebra $\mathfrak{h}$. We have two short exact sequences and one projection:
	\begin{equation*}
		\xymatrix{
		0 \ar[r] & \mathfrak{h} \ar@{^(->}[r] \ar@{.>}[d] & \mathfrak{g}^\C \ar[r] \ar[d]^p &\mathfrak{g}^M \ar[r] \ar@{.>}[d]& 0\\
		0 \ar[r] & \ker \overline{J} \ar@{^(->}[r] & \mathfrak{g} \ar[r]_{\overline{J}} & \mathfrak{g}^M/\mathfrak{g} \ar[r] & 0,
		}
	\end{equation*}
	where the surjection $\mathfrak{g}^\C \to \mathfrak{g}^M$ is induced by the quotient map $G^\C \to G^M = G^\C/Z_{G^\C}$ and $p \co \mathfrak{g}^\C \to \mathfrak{g}$ is the first projection $\mathfrak{g}^\C = \mathfrak{g}\oplus \sqrt{-1}\mathfrak{g} \to \mathfrak{g}$. In order to make the diagram above commute, we define $p' \co \mathfrak{g}^M \to \mathfrak{g}^M/\mathfrak{g}$ by $p'(v)=[\sqrt{-1}v]$, where $[\sqrt{-1}v] \in \mathfrak{g}^M/\mathfrak{g}$ denotes the equivalence class of $\sqrt{-1}v \in \mathfrak{g}^M$. Then, the  square of the right side of the diagram becomes commutative. So we have that $p(\mathfrak{h}) \subseteq \ker \overline{J}$. Therefore, we have the following commutative diagram
	\begin{equation*}
		\xymatrix{
		0 \ar[r] & \mathfrak{h} \ar@{^(->}[r] \ar[d]^{p|_{\mathfrak{h}}} & \mathfrak{g}^\C \ar[r] \ar[d]^p &\mathfrak{g}^M \ar[r] \ar[d]^{p'}& 0\\
		0 \ar[r] & \ker \overline{J} \ar@{^(->}[r] & \mathfrak{g} \ar[r]_{\overline{J}} & \mathfrak{g}^M/\mathfrak{g} \ar[r] & 0.
		}
	\end{equation*}
	\begin{lemm}\Label{lemm:snake}
		The map $p|_{\mathfrak{h}} \co \mathfrak{h} \to \ker \overline{J}$ is an isomorphism. In particular, $p|_{\mathfrak{h}}$ is injective. 
	\end{lemm}
	\begin{proof}
		By the snake lemma, we have an exact sequence
		\begin{equation*}
			0 \to \ker (p|_{\mathfrak{h}}) \to \ker p \to \ker p' \to \coker (p|_\mathfrak{h}) \to \coker p = 0
		\end{equation*}
		because $p$ is surjective. In order to show that $p|_{\mathfrak{h}}$ is an isomorphism, it suffices to show that the map $\ker p \to \ker p'$ is an isomorphism. By definition of $p$ and $p'$,  $\ker p  = \sqrt{-1}\mathfrak{g} \subset \mathfrak{g}^\C$ and $\ker p' = \sqrt{-1}\mathfrak{g} \subset \mathfrak{g}^M$. By the decompositions \eqref{eq:decomposition}, we have that $\mathfrak{g}^\C = \sqrt{-1}\mathfrak{g}\oplus \mathfrak{g}_x \oplus \mathfrak{h}$ and $\mathfrak{g}^M = \sqrt{-1}\mathfrak{g}\oplus \mathfrak{g}_x$. Since the quotient map $\mathfrak{g}^\C \to \mathfrak{g}^M$ kills only the factor of $\mathfrak{h}$, its restriction $\ker p \to \ker p'$ is an isomorphism. The lemma is proved. 
	\end{proof}
	\begin{coro}\Label{coro:complete}
		Let $q \co \mathfrak{g} \to \mathfrak{g}/p(\mathfrak{h})$ be the quotient map. Then, the collection $q(\Delta ):= \{ q(C_I)\mid {I \in \Sigma}\}$ of cones is a complete fan in $\mathfrak{g}/p(\mathfrak{h})$. 
	\end{coro}
	\begin{proof}
		By the group homomorphism theorem, the surjective map $\overline{J} \co \mathfrak{g} \to \mathfrak{g}^M/\mathfrak{g}$ descends to the isomorphism $\overline{\overline{J}} \co \mathfrak{g}/\ker \overline{J} \to \mathfrak{g}^M/\mathfrak{g}$. Since $p(\mathfrak{h}) = \ker \overline{J}$ by Lemma \ref{lemm:snake}, the vector space $\mathfrak{g}/\ker \overline{J}$ is exactly $\mathfrak{g}/p(\mathfrak{h})$. By the fact that the quotient map $q \co \mathfrak{g} \to \mathfrak{g}/p(\mathfrak{h})$ coincides with the composition $\left( \overline{\overline{J}}\right)^{-1} \circ \overline{J}$ and by Lemma \ref{lemm:complete}, we have that $q(\Delta )$ is a complete fan in $\mathfrak{g}/p(\mathfrak{h})$, as required. 
	\end{proof}
\section{Transition functions and principal bundles}\Label{sec:bundle}
	Let $M$ be a compact connected complex manifold of complex dimension $n$ equipped with a maximal action of a compact torus $G$ of real dimension $m$, which preserves the complex structure on $M$. For such a manifold, we defined a $G^M$-invariant open dense submanifold 
	\begin{equation*}
		N = \bigcup_{G\cdot x} N(G\cdot x)
	\end{equation*}
	of $M$ in Section \ref{sec:fan}. Let $N_1,\dots ,N_k$ be the characteristic submanifolds of $N$. 
	Since each minimal orbit $G\cdot x$ meets exactly $2n-m$ characteristic submanifold $N_i$ for $i \in I$, $|I| = 2n-m$ and $N_I = \bigcap _{i \in I}N_i$ is connected by Lemma \ref{lemm:codim}, $G\cdot x = N_I$ for some $I \in \Sigma$, $|I|=2n-m$. Conversely, $N_I$ for $I \in \Sigma$, $|I| =2n-m$ is a minimal orbit in $N$. For each $I \in \Sigma$, we set 
	\begin{equation*}
		U_I := \begin{cases}
			N(G\cdot x) & \text{if $N_I = G\cdot x$, that is, $|I|=2n-m$},\\
			\displaystyle \bigcap _{I \subset I' \in \Sigma, |I'|=2n-m} U_{I'} & \text{otherwise}.
		\end{cases}
	\end{equation*}
	By definition of $U_I$'s, $U_I \subseteq U_J$ if and only if $I \subseteq J$. 
	
	By Proposition \ref{prop:slicetheorem}, there exists a $G^M$-equivariantly biholomorphism 
	\begin{equation*}
		\psi  \co G^M \times _{(G_x)^\C}  \bigoplus_{j=1}^{2n-m}\C_{\alpha _j} \to N(G\cdot x), 
	\end{equation*}
	where each $\C_{\alpha _j}$ is a $1$-dimensional representation of $(G_x)^\C$. 
	
	Let $G_I$ be the subtorus of $G$ generated by $G_i$, $i \in I$.
	Let $\lambda _i$ be as \eqref{eq:lambda}. Then, $(\lambda _i)_{i\in I}$ is a $\Z$-basis of $\Hom (S^1,G_I)$ (see Lemma \ref{lemm:unimodular}). 
	Let $(\alpha ^I_i)_{i \in I}$ be the dual basis of $(\lambda _i)_{i \in I}$. By definition of $\lambda _i$, each $\alpha _j$ should be one of $\alpha _i^I$, $i\in I$. So we may replace the domain of definition and the range of $\psi$ by 
	\begin{equation*}
		G^M \times _{(G_I)^C} \bigoplus _{i \in I} \C_{\alpha _i^I} \to U_I
	\end{equation*} 
	if $|I| =2n-m$.
	Let $y$ be a point in the free $G^M$-orbit on $N$. By scaling $\psi$ with an element of $G^M$, we may assume that 
	\begin{equation*}
		\psi ^{-1}(y) = [1, (1)_{i \in I}],
	\end{equation*}
	where $(1)_{i \in I}$ denotes the point in $\bigoplus _{i \in I}\C_{\alpha _i^I}$ whose all coordinates are $1$. 
	
	Now we fix the point $y$ sitting in the free $G^M$-orbit on $N$. For each $I \in \Sigma$ such that $|I|=2n-m$, we have a $G^M$-equivariant biholomorphism 
	\begin{equation*}
		\psi _I \co G^M \times_{(G_I)^\C}\bigoplus _{i \in I}\C_{\alpha _i^I} \to U_I
	\end{equation*}
	such that $\psi ([1,(1)_{i \in I}]) = y$. For arbitrary $J \in \Sigma$, we want to construct such a $G^M$-equivariant biholmorphism $\psi _J$ whose range is $U_J$. We begin with the following lemma.
	\begin{lemm}\Label{lemm:complement}
		Let $I \in \Sigma$, $|I|=2n-m$. Let $I' \subset I$. Then, the preimage of $U_{I'}$ by $\psi _I$ is 
		\begin{equation*}
			\{ [g,(z_i)_{i \in I}] \in G^M\times _{(G_I)^\C}\bigoplus _{i \in I}\C_{\alpha _i^I}\mid z_i \neq 0 \text{ for $i \notin I'$}\}.
		\end{equation*}
	\end{lemm}
	\begin{proof}
		By definition of $U_{I'}$ and $N_{I'}^0$, $N_{I'}^0 \cap U_{I'} \neq \emptyset$. Let $p$ be a point in $N_{I'}^0 \cap U_{I'}$ and let $\psi _I^{-1}(p)$ denote 
		\begin{equation*}
			[g', (z'_i)_{i \in I}] \in G^M\times _{(G_I)^\C}\bigoplus _{i \in I}\C_{\alpha _i^I}.
		\end{equation*}
		Since $p \in N_{I'}^0 \cap U_{I'}$, $z'_i = 0$ for $i \in I'$. Since $U_{I'}$ is  a $G^M$-invariant open subset of $U_I$, we have that 
		\begin{equation*}
			\psi_I^{-1}(U_{I'}) \supseteq \left\{ [g,(z_i)_{i \in I}] \in G^M\times _{(G_I)^\C}\bigoplus _{i \in I}\C_{\alpha _i^I}\mid z_i \neq 0 \text{ for $i \notin I'$}\right\}.
		\end{equation*}
		
		Conversely, for $I \in \Sigma$ with $|I|=2n-m$, $N_J \cap U_I \neq \emptyset $ if and only if $J \subseteq I$. Since $\overline{J}(\Delta )$ is complete by Lemma \ref{lemm:complete}, $\Sigma$ is a sphere triangulation such that the link of each simplex is also a sphere triangulation. Since $\link _{\Sigma }I'$ is a sphere triangulation, there is no vertex $\{j\}$ in $\link _{\Sigma}I'$ such that $\{j\}$ intersects all simplices in $\link _{\Sigma}I'$. Therefore, the intersection
		\begin{equation*}
			\bigcap _{I \in \Sigma, |I|=2n-m, I' \subset I}I
		\end{equation*}
		is exactly $I'$. By definition of $U_{I'}$ and the fact that $N_J \cap U_I \neq \emptyset $ if and only if $J \subseteq I$ for $I \in \Sigma$ with $|I|=2n-m$, we have that $U_{I'} \cap N_J \neq \emptyset$ if and only if $J \subseteq I'$. This means that 
		\begin{equation}\Label{eq:UIUI'}
			U_I \setminus \bigcup _{i \notin I'}N_i \supseteq U_{I'}.
		\end{equation}
		Applying $\psi_I^{-1}$ to \eqref{eq:UIUI'}, we have the opposite inclusion
		\begin{equation*}
			\psi_I^{-1}(U_{I'}) \subseteq \left\{ [g,(z_i)_{i \in I}] \in G^M\times _{(G_I)^\C}\bigoplus _{i \in I}\C_{\alpha _i^I}\mid z_i \neq 0 \text{ for $i \notin I'$}\right\},
		\end{equation*}
		proving the lemma.
	\end{proof}
	\begin{lemm}\Label{lemm:localform}
		Let $I \in \Sigma$, $|I|=2n-m$. Let $I' \subset I$. The open subset 
		\begin{equation*}
			\left\{ [g,(z_i)_{i \in I}] \in G^M\times _{(G_I)^\C}\bigoplus _{i \in I}\C_{\alpha _i^I}\mid z_i \neq 0 \text{ for $i \notin I'$}\right\}
		\end{equation*}
		of $G^M\times _{(G_I)^\C}\bigoplus _{i \in I}\C_{\alpha _i^I}$ is $G^M$-equivariantly biholomorphic to 
		\begin{equation*}
			G^M\times _{(G_{I'})^\C}\bigoplus _{i \in I'}\C_{\alpha _i^{I'}}. 
		\end{equation*}
	\end{lemm}
	\begin{proof}
		Define a map \begin{equation*}
		\Theta _{I', I}\co \left\{ (g,(z_i)_{i \in I}) \in G^M\times \bigoplus _{i \in I}\C_{\alpha _i^I}\mid z_i \neq 0 \text{ for $i \notin I'$}\right\} \to G^M\times \bigoplus _{i \in I'}\C_{\alpha _i^{I'}}
		\end{equation*}
		by
		\begin{equation*}
			\Theta _{I', I}(g,(z_i)_{i \in I}) := \left( g\prod _{i\in I\setminus I'}\lambda _i(z_i), (z_i)_{i \in I'}\right). 
		\end{equation*}
		The map $\Theta_{I',I}$ descends to a $G^M$-equivariant holomorphic map 
		\begin{equation*}
			\theta _{I',I} \co \left\{ [g,(z_i)_{i \in I}] \in G^M\times _{(G_{I})^\C}\bigoplus _{i \in I}\C_{\alpha _i^I}\mid z_i \neq 0 \text{ for $i \notin I'$}\right\} \to G^M\times _{(G_{I'})^\C}\bigoplus _{i \in I'}\C_{\alpha _i^{I'}}
		\end{equation*}
		defined by $\theta_{I',I}([g,(z_i)_{i \in I}])=[g\prod _{i\in I\setminus I'}\lambda _i(z_i), (z_i)_{i \in I'}]$. $\theta _{I', I}$ is well-defined because $[g,(z_i)_{i\in I}] = [g\prod _{i\in I\setminus I'}\lambda _i(z_i),(w_i)_{i \in I}]$, where 
		\begin{equation*}
			w_i = \begin{cases}
				z_i & \text{if $i \in I'$},\\
				1 & \text{otherwise}.
			\end{cases}
		\end{equation*}
		Define another map
		\begin{equation*}
			\Xi_{I',I} \co G^M\times \bigoplus _{i \in I'}\C_{\alpha _i^{I'}}\to \left\{ (g,(z_i)_{i \in I}) \in G^M\times \bigoplus _{i \in I}\C_{\alpha _i^I}\mid z_i \neq 0 \text{ for $i \notin I'$}\right\}
		\end{equation*}
		by $\Xi_{i',I}(g,(z_i)_{i \in I'}):=(g,(w_i)_{i \in I})$, 
		where $w_i$'s are as above. We claim that the map $\Xi_{I',I}$ also descends to a $G^M$-equivariant holomorphic map
		\begin{equation*}
			\xi_{I',I} \co G^M\times _{(G_{I'})^\C}\bigoplus _{i \in I'}\C_{\alpha _i^{I'}} \to \left\{ [g,(z_i)_{i \in I}] \in G^M\times _{(G_{I})^\C}\bigoplus _{i \in I}\C_{\alpha _i^I}\mid z_i \neq 0 \text{ for $i \notin I'$}\right\}
		\end{equation*}
		defined by $\xi_{I',I}([g,(z_i)_{i \in I'}])=[g,(w_i)_{i \in I}]$. To see the well-definedness, take any element $h \in (G_{I'})^\C$. Since $(\lambda _i)_{i \in I'}$ is the $\Z$-basis of $\Hom (S^1,G_{I'}) \subset \Hom (S^1,G_I)$ and $\alpha _i^{I}$ is the dual basis of $(\lambda _i)_{i \in I}$, 
		\begin{equation*}	
			\alpha _i^I (h)=
			\begin{cases}
				\alpha_i^{I'}(h) &\text{if $i \in I'$},\\
				1 & \text{if $i \in I \setminus I'$}.
			\end{cases}
		\end{equation*}
		Thus, by direct computation, 
		\begin{equation*}
			\Xi _{I',I}(gh,(\alpha _i^{I'}(h)^{-1}z_i)_{i \in I'}) = (gh, \alpha _i^I(h)^{-1}w_i)_{i \in I}.
		\end{equation*}
		So $\xi_{I',I}$ is well-defined. 
		
		It follows from direct computation that both compositions $\theta_{I',I}\circ \xi_{I',I}$ and $\xi_{I',I} \circ \theta_{I',I}$ are the identities. Therefore $\theta _{I',I}$ is a $G^M$-equivariant biholomorphic map, as required. 
	\end{proof}
	Let $I'\subset I \in \Sigma$ such that $|I|=2n-m$. By Lemmas \ref{lemm:complement} and \ref{lemm:localform}, the composition 
	\begin{equation*}
		\psi_{I'} ^{-1}:= \theta _{I',I}\circ\psi _I \co U_{I'} \to G^M \times _{(G_{I'})^\C}\bigoplus _{i \in I'}\C_{\alpha_i^{I'}}
	\end{equation*} is a $G^M$-equivariant biholomorphism such that $\psi _{I'}^{-1}(q)=[1,(1)_{i \in I'}]$, where $q$ is the point in the free $G^M$-orbit which we fixed. $\psi_{I'}^{-1}$ does not depend on the choice of $I \in \Sigma$ with $|I|=2n-m$; because $\psi_{I'}^{-1}(y)$ determines the values of all points in the free $G^M$-orbit and the free $G^M$-orbit is open dense in $N$.  
	
	For every $I \in \Sigma$, now we have a $G^M$-equivariant biholomorphism 
	\begin{equation*}
		\psi _I^{-1} \co U_I \to G^M \times _{(G_I)^\C}\bigoplus _{i \in I}\C_{\alpha _i^I}
	\end{equation*}
	such that $\psi _I^{-1}(y)=[1,(1)_{i \in I}]$. If $J \subseteq I \in \Sigma$, then $U_J \subseteq U_I$ and the composition 
	\begin{equation*}
		\psi _I^{-1}\circ \psi_J \co G^M \times _{(G_J)^\C}\bigoplus _{i \in J}\C_{\alpha _i^{J}} \to  G^M \times _{(G_I)^\C}\bigoplus _{i \in I}\C_{\alpha _i^I}
	\end{equation*}
	is a $G^M$-equivariant open inclusion. We represent this explicitly for later use.
	\begin{lemm}\Label{lemm:transition}
		\begin{equation*}
			\psi_I^{-1}\circ \psi_J([g,(z_i)_{i \in J}])=[g,(w_i)_{i \in I}],
		\end{equation*}
		where 
		\begin{equation*}
			w_i = \begin{cases}
				z_i & \text{if $i \in J$,}\\
				1 & \text{otherwise}.
			\end{cases}
		\end{equation*}
	\end{lemm}
	\begin{proof}
		First, we check the well-definedness of the mapping. Let $h \in (G_J)^\C$. Then, as well as the proof of Lemma \ref{lemm:localform},
		\begin{equation*}
			[g, (w_i)_{i \in I}] = [gh,(\alpha_i^I(h)^{-1}w_i)_{i\in I}]=[gh,(w'_i)_{i \in I}],
		\end{equation*}
		where 
		\begin{equation*}
			w'_i= \begin{cases}
				\alpha_i^J(h)^{-1}z_i &\text{if $i \in J$},\\
				1 & \text{otherwise}
			\end{cases}
		\end{equation*}
		because 
		\begin{equation*}
			\alpha _i^I(h) = \begin{cases}
				\alpha _i^J(h) & \text{if $i \in J$},\\
				1 & \text{if $i \in I \setminus J$}. 
			\end{cases}
		\end{equation*}		
		Therefore 
		the mapping is well-defined. 
		
		Clearly, the mapping is $G^M$-equivariant. Since $\psi_I^{-1}\circ \psi_J([1,(1)_{i \in J}])=[1,(1)_{i \in I}]$ and a $G^M$-equivariant map $G^M \times _{(G_J)^\C}\bigoplus _{i \in J}\C_{\alpha _i^{J}} \to  G^M \times _{(G_I)^\C}\bigoplus _{i \in I}\C_{\alpha _i^I}$ such that $[1,(1)_{i \in J}] \mapsto [1,(1)_{i \in I}]$ is unique because $G^M\cdot [1,(1)_{i \in J}]$ is dense, $\psi_I^{-1}\circ \psi_J$ is of the form as in Lemma \ref{lemm:transition}, proving the lemma.
	\end{proof}
	Now we can reconstruct $N$ from $G^M\times _{(G_I)^\C}\bigoplus _{i\in I}\C_{\alpha _i^I}$ for all $I \in \Sigma$. We consider the quotient space
	\begin{equation}\Label{eq:N}
		\bigsqcup_{I \in \Sigma}\left(G^M\times_{(G_I)^\C}\bigoplus _{i\in I}\C_{\alpha _i^I}\right)/\sim
	\end{equation}
	of the disjoint union of $G^M\times _{(G_I)^\C}\bigoplus _{i\in I}\C_{\alpha _i^I}$ for all $I \in \Sigma$, where $\sim$ is the equivalence relation generated by the following:
	\begin{quote}
		For $J \subseteq I$ and $[g,(z_i)_{\in J}] \in G^M \times _{(G_J)^\C}\bigoplus _{i \in J}\C_{\alpha _i^{J}}$, \begin{equation*}
		[g,(z_i)_{i \in J}] \sim [g,(w_i)_{i \in I}] \in G^M\times _{(G_I)^\C}\bigoplus _{i\in I}\C_{\alpha _i^I},
		\end{equation*} where 
		\begin{equation*}
			w_i = \begin{cases}
				z_i & \text{if $i \in J$,}\\
				1 & \text{otherwise}.
			\end{cases}
		\end{equation*}
	\end{quote}
	It follows from Lemma \ref{lemm:transition} that the quotient space $\bigsqcup_{I \in I}\left(G^M\times_{(G_I)^\C}\bigoplus _{i\in I}\C_{\alpha _i^I}\right)/\sim$ is $G^M$-equivariantly biholomorphic to $N$. 
		
	Thanks to this gluing reconstruction of $N$, we can construct a $G^\C$-equivariant principal $Z_{G^\C}$-bundle over $N$ whose total space is a toric variety. Since $G^M = G^\C/Z_{G^\C}$ by definition, we have a $G^\C$-equivariant principal $Z_{G^\C}$-bundle
	\begin{equation*}
		G^\C \times _{(G_I)^\C}\bigoplus _{i\in I}\C_{\alpha _i^I} \to G^M \times _{(G_I)^\C}\bigoplus _{i\in I}\C_{\alpha _i^I}
	\end{equation*}
	for all $I \in \Sigma$. 
	\begin{lemm}\Label{lemm:affine}
		The total space is an affine toric variety corresponding to the cone $C_I \in \Delta$. 
	\end{lemm}
	\begin{proof}
    		Let $\check{C_I}$ be the dual cone of $C_I$ and let $S_I$ denote the semigroup $\check{C_I} \cap \Hom (G,S^1)$. Recall that the affine toric variety corresponding to $C_I$ is the set $\Hom (S_I, (\C, \times))$ of all semigroup homomorphisms. $\Hom (S_I, (\C, \times))$ is regarded as an affine variety whose coordinate ring is $\C [S_I]$. Let $[g,(z_i)_{i \in I}] \in G^\C \times_{(G_I)^\C}\bigoplus_{i \in I}\C_{\alpha_i^I}$. For $\alpha \in S_I \subseteq \Hom (G,S^1) \subseteq \mathfrak{g}^*$, $\langle \alpha, \lambda_i\rangle$ is a nonnegatibe integer. Hence we have a map 
		\begin{equation*}
			\theta \co G^\C \times_{(G_I)^\C}\bigoplus_{i \in I} \C_{\alpha_i^I} \to \Hom (S_I, (\C, \times))
		\end{equation*}
		given by 
		\begin{equation*}
			(\theta ([g,(z_i)_{i\in I}]))(\alpha ) = \alpha (g) \prod_{i \in I}z_i^{\langle \alpha, \lambda_i\rangle}. 
		\end{equation*}
		
		Let $\alpha_i \in S_I$ for $i \in I$ satisfy that $\alpha_i|_{G_I} = \alpha_i^I \in \Hom (G_I,S^1)$. Let $\beta_1,\dots, \beta_{m-|I|}$ be a $\Z$-basis os $S_I \cap (-S_I) \subseteq \Hom (G,S^1)$. Then, $(\alpha_i)_{i \in I}, \beta_1,\dots, \beta_{m-|I|}$ form a $\Z$-basis of $\Hom (G, S^1)$. Since 
		\begin{equation*}
			\langle \alpha_j, \lambda_i \rangle = \begin{cases}
				0 & \text{if $j \neq i$},\\
				1 & \text{if $j =i$}
			\end{cases}
		\end{equation*}
		and $\langle \beta_j,\lambda_i\rangle =0$ for all $i \in I$ and $j=1,\dots, m-|I|$, there exist elements $\mu_1,\dots, \mu_{m-|I|} \in \Hom (S^1,G)$ such that $(\lambda_i)_{i \in I}, \mu_1,\dots, \mu_{m-|I|}$ form the dual basis of $(\alpha_i)_{i \in I}, \beta_1,\dots, \beta_{m-|I|}$. We claim that, for any $\psi \in \Hom (S_I, (\C,\times))$, 
		\begin{equation*}
			\xi (\psi ) :=  \left[ \prod_{j=1}^{m-|I|}\mu_j(\psi(\beta_j)), (\psi(\alpha_i))_{i \in I}\right] \in G^\C \times_{(G_I)^\C} \bigoplus_{i \in I}\C_{\alpha_i^I}
		\end{equation*}
		satisfies that $\theta (\xi (\psi) ) = \psi$ and $\xi (\theta ([g,(z_i)_{i \in I}])) = [g, (z_i)_{i \in I}]$. Since  $(\alpha _i)_{i \in I}, \beta_1,\dots, \beta_{m-|I|}$ are the dual basis of $(\lambda_i)_{i \in I}, \mu_1,\dots, \mu_{m-|I|}$, 
		\begin{equation*}
			\alpha_k\left(\prod _{j =1}^{m-|I|}\mu_j(\psi(\beta_j))\right)\prod_{i \in I}(\psi (\alpha_i))^{\langle \alpha_k, \lambda_i\rangle} = \psi(\alpha_k)
		\end{equation*}
		for $k \in I$ and 
		\begin{equation*}
			\beta_k\left(\prod _{j =1}^{m-|I|}\mu_j(\psi(\beta_j))\right)\prod_{i \in I}(\psi (\alpha_i))^{\langle \beta_k, \lambda_i\rangle} = \psi(\beta_k)
		\end{equation*}
		for $k =1,\dots, m-|I|$, showing that $\theta (\xi (\psi)) = \psi$. For $[g, (z_i)_{i \in i}] \in G^\C \times_{(G_I)^\C}\bigoplus_{i \in I} \C_{\alpha_i^I}$, 
		\begin{equation*}
			\begin{split}
			\xi (\theta ([g, (z_i)_{i \in i}])) &= \left[ \prod_{j=1}^{m-|I|} \mu_j(\beta_j(g)\prod_{i \in I}z_i^{\langle \beta_j,\lambda_i \rangle}), (\alpha_i(g)\prod_{j \in I}z_j^{\langle \alpha_i, \lambda_j\rangle})_{i \in I}\right]\\
			&= \left[ \prod_{j=1}^{m-|I|}\mu_j(\beta_j(g)), (\alpha_i(g)z_i)_{i\in I}\right] \\
			&= \left[ \prod_{i \in I}\lambda_i(\alpha_i(g)) \prod_{j =1}^{m-|I|}\mu_j(\beta_j(g)), (z_i)_{i \in I}\right] \\
			&= \left[g, (z_i)_{i \in I}\right]
			\end{split}
		\end{equation*}
		because $(\alpha_i)_{i \in I},\beta_1,\dots, \beta_{m-|I|}$ are the dual basis of $(\lambda_i)_{i \in I}, \mu_1,\dots, \mu_{m-|I|}$, showing that $\xi(\theta ([g, (z_i)_{i \in I}])) = [g, (z_i)_{i \in I}]$. Clearly, both $\theta$ and $\xi$ are holomorphic and $G^\C$-equivariant. Therefore, these computations show that $\xi$ and $\theta$ are equivariant biholomorphic maps and hence $G^\C \times _{(G_I)^\C}\bigoplus _{i\in I}\C_{\alpha _i^I}$ is an affine toric variety corresponding to the cone $C_I$, proving the lemma. 
	\end{proof}
	If a cone $C_J \in \Delta$ is a face of a cone $C_I \in \Delta$, the inclusion $C_J \subset C_I$ induces an equivariant open embedding of the affine toric variety corresponding to $C_J$ into the affine toric variety corresponding to $C_I$. By Lemma \ref{lemm:affine}, both toric varieties are biholomorphic to 
	\begin{equation*}
		G^\C \times_{(G_J)^\C}\bigoplus _{j \in J}\C_{\alpha _j^J} \quad \text{and}\quad G^\C \times_{(G_J)^\C}\bigoplus _{i \in I}\C_{\alpha _i^I}. 
	\end{equation*}
	\begin{lemm}\Label{lemm:transition2}
		The open embedding 
		\begin{equation*}
			\widetilde{\psi}_I^{-1}\circ \widetilde{\psi}_J \co G^\C \times_{(G_J)^\C}\bigoplus _{j \in J}\C_{\alpha _j^J} \to G^\C \times_{(G_I)^\C}\bigoplus _{i \in I}\C_{\alpha _i^I}
		\end{equation*}
		is given by 
		\begin{equation*}
			\widetilde{\psi}_{I}^{-1}\circ \widetilde{\psi}_J ([g,(z_j)_{j\in J}]) = [g,(w_j)_{j \in I}],
		\end{equation*}
		where 
		\begin{equation*}
			w_i = \begin{cases}
				z_i & \text{if $i \in J$,}\\
				1 & \text{otherwise}.
			\end{cases}
		\end{equation*}
	\end{lemm}
	\begin{proof}
		An almost same argument as the proof of Lemma \ref{lemm:transition} works. We omit the detail.
	\end{proof}
	By the construction of toric variety associated with $\Delta$ and by Lemma \ref{lemm:transition2}, we have a toric variety as a quotient space
	\begin{equation}\Label{eq:toric}
		X(\Delta ) := \bigsqcup _{I \in \Sigma} G^\C \times_{(G_I)^\C}\bigoplus _{i \in I}\C_{\alpha _i^I} / \approx, 
	\end{equation}
	where $\approx$ is the equivalence relation generated by the following: 
	\begin{quote}
		For $J \subseteq I$ and $[g,(z_i)_{\in J}] \in G^\C \times _{(G_J)^\C}\bigoplus _{i \in J}\C_{\alpha _i^{J}}$, \begin{equation*}
		[g,(z_i)_{i \in J}] \approx [g,(w_i)_{i \in I}] \in G^\C \times _{(G_I)^\C}\bigoplus _{i\in I}\C_{\alpha _i^I},
		\end{equation*} where 
		\begin{equation*}
			w_i = \begin{cases}
				z_i & \text{if $i \in J$,}\\
				1 & \text{otherwise}.
			\end{cases}
		\end{equation*}
	\end{quote}
	By comparing the equivalence relation $\sim$ (see \eqref{eq:N}) with $\approx$ (see \eqref{eq:toric}), we have the following lemma. 
	\begin{lemm}\Label{lemm:bundle}
		There exists a $G^\C$-equivariant principal $Z_{G^\C}$-bundle
		\begin{equation*}
			X(\Delta ) \to N,
		\end{equation*}
		where $X(\Delta)$ denotes the toric variety associated with $\Delta$. 
	\end{lemm}
	\begin{coro}\Label{coro:quotient}
		$N$ is $G^M$-equivariantly biholomorphic to $X(\Delta)/Z_{G^\C}$.
	\end{coro}
	\begin{proof}
		The corollary follows from Lemma \ref{lemm:bundle} immediately. 
	\end{proof}
	
\section{Quotients of nonsingular toric varieties}\Label{sec:maintheo}
	Let $M$ be a compact connected complex manifold of complex dimension $n$, equipped with a maximal action of a compact torus $G$ of dimension $m$ which preserves the complex structure on $M$. So far, we defined a $G^M$-invariant open submanifold $N$ and assigned a nonsingular fan $\Delta$ to $N$ in Section \ref{sec:fan}. The fan $\Delta$ and the Lie algebra $\mathfrak{h}$ of the global stabilizers $Z_{G^\C}$ of the complexified action of $G^\C$ on $M$ were characterized by Lemma \ref{lemm:snake} and Corollary \ref{coro:complete}. We saw that $N$ is $G^M$-equivariantly biholomorphic to $X(\Delta )/Z_{G^\C}$ in Corollary \ref{coro:quotient}. 
	
	In this section, we give an inverse correspondence, that is, for a triple $(\Delta, \mathfrak{h}, G)$ of a nonsingular fan $\Delta$ in the Lie algebra $\mathfrak{g}$ of $G$, a Lie subalgebra $\mathfrak{h}$ of $\mathfrak{g}^\C$ and a compact torus $G$, which satisfies certain conditions, we construct a compact complex manifold equipped with a maximal action of a compact torus which preserves the complex structure of the manifold. 
	
	As before, we denote by $G$ a compact torus, by $\mathfrak{g}$ the Lie algebra of $G$, by $\mathfrak{g}^\C$ the complexification $\mathfrak{g}\otimes _{\R}\C \cong \mathfrak{g}\oplus \sqrt{-1}\mathfrak{g}$, and by $G^\C$ the complex Lie group $\mathfrak{g}^\C /\Hom (S^1, G)$. 
	\begin{lemm}\Label{lemm:closed}
		Let $\mathfrak{h}$ be a complex subspace of $\mathfrak{g}^\C$ and let $p \co \mathfrak{g}^\C\to \mathfrak{g}$ be the projection. Suppose that $p|_{\mathfrak{h}}$ is injective. Then, the image $\exp (\mathfrak{h})$ of $\mathfrak{h}$ by the exponential map $\exp \co \mathfrak{g}^\C \to G^\C$ is closed in $G^\C$. 
	\end{lemm}
	\begin{proof}
		Let $p' \co \mathfrak{g}^\C \cong \mathfrak{g} \oplus \sqrt{-1}\mathfrak{g} \to \sqrt{-1}\mathfrak{g}$ be the $2$nd projection. We  see the injectivity of $p'|_\mathfrak{h}$ first. Let $u+\sqrt{-1}v \in \mathfrak{h}$ such that $p'(u+\sqrt{-1}v)=0$, where we represent the element $u+\sqrt{-1}v \in \mathfrak{h}$ with $u,v \in \mathfrak{g}$. Since $p'(u+\sqrt{-1}v)=0$, $v=0$ and hence $u \in \mathfrak{h}$. Since $\mathfrak{h}$ is a complex subspace, $\sqrt{-1}u$ is also an element in $\mathfrak{h}$. Applying the first projection $p$, we have $u=0$ because $p|_{\mathfrak{h}}$ is injective. Therefore, $u+\sqrt{-1}v \in \mathfrak{h}$ should be $0$. So $p'|_{\mathfrak{h}}$ is injective.

		Thus, there exists an $\R$-linear isomorphism 
		\begin{equation*}
			s \co p'(\mathfrak{h}) \to \mathfrak{h}
		\end{equation*}
		such that the composition $p' \circ s$ is the identity. 
		Since $\mathfrak{g}^\C = \mathfrak{g} \oplus \sqrt{-1}\mathfrak{g}$, we have
		\begin{equation*}
			\mathfrak{h} = \{ p\circ s(\sqrt{-1}v) + \sqrt{-1}v \in \mathfrak{g} \oplus \sqrt{-1}\mathfrak{g} \mid \sqrt{-1}v \in p'(\mathfrak{h})\}. 
		\end{equation*}
		By definition of $G^\C$, $G^\C$ can be identified with $G \times \sqrt{-1}\mathfrak{g}$. Under this identification, 
		\begin{equation*}
			\exp (\mathfrak{h}) = \{ (g,\sqrt{-1}v)\in G \times \sqrt{-1}\mathfrak{g} \mid g = \exp (p\circ s(\sqrt{-1}v)),  \sqrt{-1}v \in p'(\mathfrak{h}) \}.
		\end{equation*}
		Since $p'(\mathfrak{h})$ is closed in $\sqrt{-1}\mathfrak{g}$ and hence in $G^\C$, and the mapping $(g,\sqrt{-1}v) \mapsto g\cdot \exp (p\circ s(\sqrt{-1}v))^{-1}$ is continuous, $\exp(\mathfrak{h})$ is closed in $G^\C$, as required.
	\end{proof}
	Let $\mathfrak{h}$ and $p \co \mathfrak{g}^\C \to \mathfrak{g}$ be as Lemma \ref{lemm:closed}. Assume that $p|_{\mathfrak{h}}$ is injective. By Lemma \ref{lemm:closed}, the image of $\mathfrak{h}$ by the exponential map is closed. We denote it by $H$. $H$ is a complex connected Lie subgroup of $G^\C$. 
	
	Let $\Delta$ be a nonsingular fan in $\mathfrak{g}$ and let $\lambda _1,\dots ,\lambda _k \in \Hom (S^1,G)$ be the primitive elements such that   each $\lambda _i$ generates a $1$-dimensional cone in $\Delta$. Define an abstract simplicial complex 
	\begin{equation*}
		\Sigma := \{ I \subseteq \{ 1,\dots ,k\} \mid \text{$C_I=\pos (\lambda _i\mid i \in I) \in \Delta $}\}.
	\end{equation*}
	For each $I \in \Sigma$, we set the subspace 
	\begin{equation*}
		\mathfrak{g}_I := \left\{ \sum _{i \in I}a_i\lambda _i \mid a_i \in \R\right\}
	\end{equation*}
	of $\mathfrak{g}$ which is generated by $\lambda _i$ for $i \in I$. 
	Let $G_I$ be the subtorus of $G$ whose Lie algebra is $\mathfrak{g}_I$.
	Since $\Delta$ is nonsingular, $\lambda _i$ for $i \in I$ form a basis of $\Hom (S^1, G_I)$. Let $\alpha _i^I \in \Hom (G_I,S^1)$ denote the dual basis of $\lambda _i$ for $i \in I$. Then, by Lemma \ref{lemm:affine}, the affine toric variety associated with $C_I$ is $G^\C$-equivariantly biholomorphic to 
	\begin{equation*}
		G^\C \times _{(G_I)^\C}\bigoplus _{i \in I}\C _{\alpha _i^I}.
	\end{equation*}
	\begin{lemm}\Label{lemm:free}
		Let $q \co \mathfrak{g} \to \mathfrak{g}/p(\mathfrak{h})$ be the quotient map. Let $I \in \Sigma$. Suppose that $q(\lambda _i)$ for $i \in I$ are linearly independent. Then, the action of $G^\C$ restricted to $H$ on the affine toric variety associated with the cone $C_I$ is free. 
	\end{lemm}
	\begin{proof}
		It suffices to show that $H \cap (G_I)^\C = \{1\}$ by Lemma \ref{lemm:affine}. 
		Since $H$ and $(G_I)^\C$ are connected and abelian, the exponential maps are surjective. 
		Let $u +\sqrt{-1}v \in \mathfrak{h}$, where $u \in p(\mathfrak{h})$ and $\sqrt{-1}v \in p'(\mathfrak{h})$. By definition of $(\mathfrak{g}_I)^\C$, any element in $(\mathfrak{g}_I)^\C$ can be written as 
		\begin{equation*}
			\sum _{i \in I}(a_i \lambda _i +\sqrt{-1}b_i \lambda _i)
		\end{equation*}
		for some $a_i,b_i \in \R$.
		Suppose that 
		\begin{equation*}
			\exp (u +\sqrt{-1}v) = \exp \left( \sum _{i \in I}(a_i \lambda _i +\sqrt{-1}b_i \lambda _i) \right).
		\end{equation*}
		Then, 
		\begin{equation*}
			u-\sum_{i \in I}a_i\lambda _i + \sqrt{-1}\left( v-\sum_{i\in I}b_i\lambda i\right) \in \Hom (S^1, G).
		\end{equation*}
		Therefore, 
		\begin{equation*}
			v = \sum_{i\in I}b_i\lambda i. 
		\end{equation*}
		Since $\mathfrak{h}$ is a complex subspace of $\mathfrak{g}^\C$, $v - \sqrt{-1}u$ is also an element in $\mathfrak{h}$. Thus, 
		\begin{equation}\Label{eq:blambda}
			\sum _{i\in I}b_i \lambda _i \in p(\mathfrak{h}).
		\end{equation}
		Applying $q$ to \eqref{eq:blambda} and by the assumption that $q(\lambda _i)$ for $i \in I$ are linearly independent, we have that $b_i=0$ for all $i \in I$. Therefore $v=0$. So $u \in \mathfrak{h}$ and $\sqrt{-1}u\in \mathfrak{h}$. Since $p|_{\mathfrak{h}}$ is injective, $u=0$. Therefore $\exp (\sum_{i \in I}a_i\lambda_i+\sqrt{-1}b_i\lambda_i) \in H$ implies that $\exp (\sum_{i \in I}a_i\lambda_i+\sqrt{-1}b_i\lambda_i)=1$. This shows that $H \cap (G_I)^\C = \{1\}$, proving the lemma.  
	\end{proof}
	Lemmas \ref{lemm:closed} and \ref{lemm:free} tell us that the quotient of the affine toric variety associated with $C_I$ by $H$ is biholomorphic to a complex manifold 
	\begin{equation*}
		(G^\C /H) \times _{(G_I)^\C} \bigoplus _{i \in I}\C_{\alpha _i^I}
	\end{equation*} 
	if $q(\lambda _i)$ for $i \in I$ are linearly independent.
	Now we consider the quotient space $X(\Delta )/H$ of the nonsingular toric variety associated with $\Delta$ by $H$. 
	In order to deduce the Hausdorff-ness and the compactness of $X(\Delta )/H$ from conditions on $q(\Delta)$, we prepare a locally defined continuous function. 
	\begin{lemm}\Label{lemm:function}
		Let $\xi \in (\mathfrak{g}/p(\mathfrak{h}))^*$ such that $\langle \xi , q(\lambda _i)\rangle \geq 0$ for all $i \in I$. The function $f _\xi^I \co \mathfrak{g}^\C \times \bigoplus _{i \in I}\C_{\alpha _i^I} \to [0, \infty)$ defined by
		\begin{equation*}
			f_\xi^I (u+\sqrt{-1}v, (z_i)_{i\in I}) := e^{-2\pi\langle \xi, q(v)\rangle}\prod _{i\in I}|z_i|^{\langle \xi, q(\lambda _i)\rangle} 
		\end{equation*}
		descends to a continuous function $\overline{f_\xi ^I} \co (G^\C/H) \times _{(G_I)^\C} \bigoplus _{i \in I}\C_{\alpha _i^I} \to [0, \infty )$.
	\end{lemm}
	\begin{proof}
		It suffices to show the well-definedness of $\overline{f_\xi ^I}$ because the continuity of $f_\xi^I$ and $\overline{f_\xi^I}$ are obvious. 
		
		Clearly, $f^I_\xi$ is invariant under the translation by elements in $\mathfrak{g}$. In particular, it is invariant under the translation by elements in $\Hom (S^1,G)$. 
		Let $u'+\sqrt{-1}v' \in \mathfrak{h}$, where $u', v' \in \mathfrak{g}$. Then, $v' \in p(\mathfrak{h})$ because $v'-\sqrt{-1}u' \in \mathfrak{h}$. Thus, 
		\begin{equation*}
			\begin{split}
			f_\xi^I (u+u'+\sqrt{-1}(v+v'), (z_i)_{i\in I}) &= e^{-2\pi\langle \xi, q(v+v')\rangle}\prod _{i\in I}|z_i|^{\langle \xi, q(\lambda _i)\rangle} \\
			 &= e^{-2\pi\langle \xi, q(v)\rangle}\prod _{i\in I}|z_i|^{\langle \xi, q(\lambda _i)\rangle} \\
			 &= f_\xi^I (u+\sqrt{-1}v, (z_i)_{i\in I}).
			\end{split}
		\end{equation*}
		Let $u'' + \sqrt{-1}v'' \in (\mathfrak{g}_I)^\C$, where $u'', v'' \in \mathfrak{g}_I$. Then, 
		\begin{equation*}
			\begin{split}
			& f^I_\xi (u+u''+\sqrt{-1}(v+v''), (\alpha _i^I(\exp (u''+\sqrt{-1}v''))^{-1}z_i)_{i \in I})\\
			&= e^{-2\pi\langle \xi, q(v+v'')\rangle}\prod _{i \in I}|e^{2\pi\langle \alpha _i^I, v''\rangle}z_i|^{\langle \xi, q(\lambda _i)\rangle}\\
			&= e^{2\pi \langle \xi, -q(v'')+\sum _{i \in I}\langle \alpha _i^I,v''\rangle q(\lambda _i)\rangle }f_i^I(u+\sqrt{-1}v, (z_i)_{i\in I}).
			\end{split} 
		\end{equation*}
		Since $\lambda _i$ for $i \in I$ form a basis of $\Hom (S^1, G_I) \subset \mathfrak{g}_I$ and $\alpha _i^I$ for $i \in I$ are the dual basis, $v''= \sum _{i \in I}\langle \alpha _i^I,v''\rangle \lambda _i$. So the exponent $2\pi \langle \xi , -q(v'')+\sum _{i \in I}\langle \alpha _i^I, v''\rangle q(\lambda _i)\rangle =0$ and hence 
		\begin{equation*}
			\begin{split}
			& f^I_\xi (u+u''+\sqrt{-1}(v+v''), (\alpha _i^I(\exp (u''+\sqrt{-1}v''))^{-1}z_i)_{i \in I})\\
			&= f_i^I(u+\sqrt{-1}v, (z_i)_{i\in I}).
			\end{split} 
		\end{equation*}
		Therefore $f_\xi^I$ descends to $\overline{f_\xi^I}$, as required.
	\end{proof}
	For $\xi \in (\mathfrak{g}/p(\mathfrak{h}))^*$, we define a subfan $\Delta_\xi^+$ of $\Delta$ to be 
	\begin{equation*}
		\Delta _{\xi}^+ := \{ C_I =\pos (\lambda _i \mid i \in I) \in \Delta \mid \langle \xi, q(\lambda _i) \rangle \geq 0 \text{ for all $i \in I$}\}.
	\end{equation*}
	We also define $\Delta _\xi ^-$ to be $\Delta _{-\xi}^+$. Then, the union $\Delta _\xi$ of these subfans $\Delta _\xi^+$ and $\Delta _\xi^-$, and the intersection $\Delta _\xi^0$ of $\Delta _\xi^+$ and $\Delta _\xi^-$ are also  subfans of $\Delta$. 
	
	The toric variety $X(\Delta )$ associated with the fan $\Delta$ is $G^\C$-equivariantly biholomorphic to the quotient space as \eqref{eq:toric}. Hence the quotient space $X(\Delta )/H$ is 
		\begin{equation*}
			X(\Delta )/H = \bigsqcup _{I \in \Sigma }(G^\C/H) \times_{(G^I)^\C} \bigoplus _{i \in I}\C_{\alpha _i^I}/ \sim, 
		\end{equation*}
		where $\sim$ is the equivalence relation generated by the following:
		\begin{quote}
			For $J \subseteq I$ and $[g,(z_i)_{i \in J}] \in (G^\C/H) \times_{(G_J)^\C} \bigoplus _{i \in J} \C _{\alpha _i^J}$, 
			\begin{equation*}
				[g,(z_i)_{i \in J}] \sim [g,(w_i)_{i \in I}] \in (G^\C/H) \times_{(G_I)^\C} \bigoplus _{i \in I} \C _{\alpha _i^I},
			\end{equation*}
			where 
			\begin{equation*}
				w_i = \begin{cases}	
					z_i & \text{if $i \in J$}, \\
					1 & \text{otherwise}.
				\end{cases}
			\end{equation*}
		\end{quote}
		The subfans $\Delta _\xi^+$, $\Delta _\xi^-$, $\Delta_\xi$ and $\Delta _\xi^0$ determine open subsets $X(\Delta _\xi^+)/H$, $X(\Delta _\xi^-)/H$, $X(\Delta_\xi)/H$ and $X(\Delta _\xi^0)/H$ of $X(\Delta )/H$, respectively.  
	\begin{lemm}\Label{lemm:fxi}
		Let $\overline{f_\xi^I}$ be the function as in Lemma \ref{lemm:function} for $I \in \Sigma$ such that $C_I \in \Delta _\xi ^+$. 
		\begin{enumerate}
			\item For $I' \subseteq I$ and $[g,(z_i)_{i \in I'}] \in (G^\C/H)\times _{(G_{I'})^\C} \bigoplus _{i \in I'}\C_{\alpha _i^{I'}}$, 
			\begin{equation*}
				\overline{f_\xi^{I'}}([g,(z_i)_{i \in I'}]) = \overline{f_\xi^I}([g,(w_i)_{i \in I}]).
			\end{equation*}
			Therefore, there exists a continuous function $\overline{f_\xi} \co X(\Delta _\xi^+)/H \to [0, \infty)$ such that 
			\begin{equation*}
				\overline{f_\xi}|_{(G^\C/H) \times _{(G_I)^\C}\bigoplus _{i \in I}\C_{\alpha _i^I}} = \overline{f_\xi^I}
			\end{equation*}
			for all $I$ such that $C_I \in \Delta _\xi^+$. 
			\item $\overline{f_\xi}$ is nowhere zero on $X(\Delta _\xi^0)/H$ and 
			\begin{equation*}
				(\overline{f_\xi})^{-1} = \overline{f_{-\xi}}
			\end{equation*}
			on $X(\Delta _\xi^0)/H$. 
			\item There exists a continuous function $\overline{f_\xi}\cup (\overline{f_{-\xi}})^{-1} \co X(\Delta _\xi)/H \to [0,\infty]$ such that 
			\begin{equation*}
				\overline{f_\xi}\cup (\overline{f_{-\xi}})^{-1}|_{X(\Delta _\xi^+)/H} = \overline{f_\xi} \quad \text{and} \quad \overline{f_\xi}\cup (\overline{f_{-\xi}})^{-1}|_{X(\Delta _\xi^-)/H} = (\overline{f_{-\xi}})^{-1}, 
			\end{equation*}
			here we may regard $0^{-1}$ as $\infty$ and the topology of $[0,\infty]$ is the order topology. 
		\end{enumerate}
	\end{lemm}
	\begin{proof}
		Part (1) and Part (2) follows from the definition of $\overline{f_\xi^I}$ and of the equivalence relation $\sim$. Part (3) follows from Part (2).
	\end{proof}
	Now we are ready to deduce the Hausdorff-ness of $X(\Delta)/H$ from the condition which is $q(\Delta )$ being a fan.
	
	\begin{lemm}\Label{lemm:Hausdorff}
		Suppose that $q(\Delta ) =\{ q(C_I) \mid {I \in \Sigma}\}$ is a fan and $q(\lambda _i)$ for $i \in I$ is linearly independent for each $I \in \Sigma$. Then, the quotient space $X(\Delta)/H$ is Hausdorff. 
	\end{lemm}
			
	\begin{proof}[Proof of Lemma \ref{lemm:Hausdorff}]
		Let $x, y \in X(\Delta )/H$ such that $x \neq y$. If $x$ and $y$ belong to the same open subset 
		\begin{equation*}
			(G^\C/H) \times _{(G_I)^\C}\bigoplus _{i \in I} \C_{\alpha _i^I}, 
		\end{equation*}
		then there exists open subsets $U_x$ and $U_y$ such that $U_x\ni x$, $U_y \ni y$ and $U_x \cap U_y =\emptyset$. 
		
		Now we assume that $x$ and  $y$ do not belong to the same open subset 
		$(G^\C/H)\times _{(G_I)^\C} \bigoplus_{i \in I}\C_{\alpha _i^I}$.
		Suppose that $I$ is a simplex in $\Sigma$ such that 
		\begin{equation*}
			x \in (G^\C/H) \times _{(G_I)^\C}\bigoplus _{i \in I} \C_{\alpha _i^I}
		\end{equation*}
		but 
		\begin{equation*}
			x \notin (G^\C/H) \times _{(G_I')^\C}\bigoplus _{i \in I'} \C_{\alpha _i^{I'}} \quad \text{for all $I' \subsetneq I$}.
		\end{equation*}
		Similarly, suppose that $J$ is a simplex in $\Sigma$ such that 
		\begin{equation*}
			y \in (G^\C/H) \times _{(G_J)^\C}\bigoplus _{j \in J} \C_{\alpha _j^J}
		\end{equation*}
		but 
		\begin{equation*}
			y \notin (G^\C/H) \times _{(G_J')^\C}\bigoplus _{j \in J'} \C_{\alpha _j^{J'}} \quad \text{for all $J' \subsetneq J$}.
		\end{equation*}
		Since $q(\Delta)$ is a fan, there exists an element $\xi \in (\mathfrak{g}/p(\mathfrak{h}))^*$ such that $\langle \xi, q(\lambda _i) \rangle > 0$ for all $i \in I \setminus J$, $\langle \xi, q(\lambda _j) \rangle <0$ for all $j \in J \setminus I$ and $\langle \xi, q(\lambda _i) \rangle =0$ for $i \in I\cap J$. Then,  $C_I \in \Delta _\xi^+$ and $C_J \in \Delta _\xi^-$.
		Let $\overline{f_\xi} \cup (\overline{f_{-\xi}})^{-1}$ be the continuous function as Lemma \ref{lemm:fxi}. By the choice of $\xi$, we have that 
		\begin{equation*}
			\overline{f_\xi} \cup (\overline{f_{-\xi}})^{-1}(x) = 0 \quad \text{and} \quad \overline{f_\xi} \cup (\overline{f_{-\xi}})^{-1}(y) = \infty.
		\end{equation*}
		Since the interval $[0,\infty]$ is Hausdorff, and $x$ and $y$ are distinguished by the continuous function $\overline{f_\xi} \cup (\overline{f_{-\xi}})^{-1}$, there exist open subsets $U_x$ and $U_y$ such that $U_x \ni x$, $U_y \ni y$ and $U_x \cap U_y = \emptyset$. Therefore, $X(\Delta )/H$ is Hausdorff, as required.
	\end{proof}
	By Lemmas \ref{lemm:closed}, \ref{lemm:free} and \ref{lemm:Hausdorff}, if the nonsingular fan $\Delta$ in $\mathfrak{g}$ and the complex subspace $\mathfrak{h} \subset \mathfrak{g}^\C$ satisfies the following conditions
	\begin{enumerate}
		\item the restriction $p|_{\mathfrak{h}}$ of the projection $p \co \mathfrak{g}^\C \to \mathfrak{g}$ is injective, 
		\item the quotient map $q \co \mathfrak{g} \to \mathfrak{g}/p(\mathfrak{h})$ sends\footnote{Here, ``sends" means that $q$ gives a bijection between cones in $\Delta$ and cones in $q(\Delta) $.} a fan $\Delta$ to a fan $q(\Delta)$ in $\mathfrak{g}/p(\mathfrak{h})$,
	\end{enumerate}
	then, the quotient space $X(\Delta)/H$ of the toric variety $X(\Delta )$ associated with $\Delta$ by $H = \exp (\mathfrak{h})$ is a complex manifold\footnote{
		In \cite{Battisti}, it has been shown that $X(\Delta)/H$ is a complex manifold by showing that the action of $H$ on $X(\Delta)$ is proper and holomorphic under the conditions (1) and (2) above. Our argument in this paper is slightly different from \cite{Battisti}.}.
	
	As well as toric varieties, we will deduce the compactness of $X(\Delta)/H$ from the completeness of $q(\Delta)$. $X(\Delta)/H$ is compact if and only if $X(\Delta)/G\cdot H$ is compact because $G$ is a compact torus. So we will show the compactness of $X(\Delta)/G\cdot H$.
	\begin{lemm}\Label{lemm:quotient}
		Let the nonsingular fan $\Delta$ and the complex subspace $\mathfrak{h}$ of $\mathfrak{g}^\C$ satisfy conditions (1) and (2) above. 
		Let $C_I = \pos (\lambda _i \mid i \in I)$ be a cone in $\Delta$ such that $\dim C_I = \dim \mathfrak{g}/p(\mathfrak{h})$. Let $\overline{\alpha _i^I}$ for $i \in I$ be the dual basis of $q(\lambda _i)$. Then, the continuous map 
		\begin{equation*}
			(\overline{f_{\overline{\alpha _j^I}}^I})_{j \in I} \co (G^\C/H) \times _{(G_I)^\C}\bigoplus _{i \in I}\C_{\alpha _i^I} \to \prod _{j \in I}(\R _{\geq 0})_j
		\end{equation*}
		descends to a homeomorphism 
		\begin{equation*}
			F_I \co \left((G^\C/H) \times _{(G_I)^\C}\bigoplus _{i \in I}\C_{\alpha _i^I}\right) /G \to \prod _{i \in I}(\R _{\geq 0})_i.
		\end{equation*}
	\end{lemm}
	\begin{proof}
		For any $\xi \in (\mathfrak{g}/p(\mathfrak{h}))^*$, by definition of $\overline{f_\xi^I}$, the function $\overline{f_\xi^I}$ is $G$-invariant and hence $F_I$ is well-defined. Since $\mathfrak{h} \cap \mathfrak{g} = \{0\}$ and $(G_I)^\C \cap H =\{1\}$ (see the proof of Lemma \ref{lemm:free}), $\mathfrak{g}^\C$ can be decomposed into $\mathfrak{g}\oplus \sqrt{-1}\mathfrak{g}_I\oplus \mathfrak{h}$ because $\dim C_I = \dim \mathfrak{g}/p(\mathfrak{h})$. Therefore $G^\C/H$ can be identified with $G \times \sqrt{-1}\mathfrak{g}_I$ through the exponential map. Using this identification, we have
		\begin{equation*}
			\begin{split}
				\overline{f_{\overline{\alpha _j^I}}^I}([(g,\sqrt{-1}v),(z_i)_{i \in I}]) &= e^{-2\pi\langle \overline{\alpha _j^I}, q(v)\rangle}\prod_{i \in I}|z_i|^{\langle \overline{\alpha _j^I}, q(\lambda _i) \rangle} \\
				&= e^{-2\pi\langle \overline{\alpha _j^I},q(v)\rangle}|z_j| 
			\end{split}
		\end{equation*}
		for $[(g,\sqrt{-1}v), (z_i)_{i \in I}] \in (G\times \sqrt{-1}\mathfrak{g}_I )\times _{(G_I)^\C}\bigoplus _{i \in I}\C_{\alpha _i^I}$. Define 
		\begin{equation*}
			E_I \co \prod _{j \in I}(\R_{\geq 0})_j \to\left((G\times \sqrt{-1}\mathfrak{g}_I) \times _{(G_I)^\C}\bigoplus _{i \in I}\C_{\alpha _i^I}\right) /G
		\end{equation*}
		to be 
		\begin{equation*}
			E_I ((r_j)_{j \in I}) := [(1,0), (r_j)_{j \in I}].
		\end{equation*}
		Clearly, $E_I$ is continuous and the composition $F_I \circ E_I$ is identity. It remains to show that $E_I \circ F_I$ is identity. By direct computation, 
		\begin{equation*}
			\begin{split}
				E_I \circ F_I ([[(g,\sqrt{-1}v), (z_i)_{i \in I}]]) &= E_I ((e^{-2\pi \langle \overline{\alpha _j^I}, q(v)\rangle}|z_j|)_{j \in I})\\
				&=[[(1,0), (e^{-2\pi \langle \overline{\alpha _j^I}, q(v)\rangle}|z_j|)_{j \in I})]],
			\end{split}
		\end{equation*}
		where $[[(g,\sqrt{-1}v), (z_i)_{i \in I}]]\in \left((G\times \sqrt{-1}\mathfrak{g}_I) \times _{(G_I)^\C}\bigoplus _{i \in I}\C_{\alpha _i^I}\right) /G$ denotes the equivalence class of $[(g,\sqrt{-1}v),(z_i)_{i \in I}] \in (G\times \sqrt{-1}\mathfrak{g}_I) \times _{(G_I)^\C}\bigoplus _{i \in I}\C_{\alpha _i^I}$. Since 
		\begin{equation*}
			[[(g,\sqrt{-1}v),(z_i)_{i \in I}] = [[(1,0),(e^{-2\pi\langle \alpha _i^I,v\rangle}|z_i|)_{i \in I}]],
		\end{equation*}
		it suffices to show that $\langle \alpha _i^I,v\rangle = \langle \overline{\alpha _i^I} ,q(v)\rangle$ for $v \in \mathfrak{g}_I$. Since $\lambda _i$ for $i \in I$ form a basis of $\mathfrak{g}_I$, $v = \sum _{i \in I}\langle \alpha _i^I,v\rangle \lambda _i$. Applying $q$, we have that $q(v)=\sum_{i \in I}\langle \alpha _i^I,v\rangle q(\lambda _i)$. Since $q(\lambda _i)$ for $i \in I$ are a basis of $\mathfrak{g}/p(\mathfrak{h})$ and $\overline{\alpha _i^I}$ are its dual, $q(v) = \sum _{i \in I}\langle \overline{\alpha _i^I},q(v)\rangle q(\lambda_i)$. Therefore,
		\begin{equation*}
			\sum_{i \in I}(\langle \alpha_i^I,v\rangle -\langle \overline{\alpha _i^I},q(v)\rangle )q(\lambda _i)=0.
		\end{equation*}
		Since $q(\lambda _i)$ for $i \in I$ are linearly independent, $\langle \alpha_i^I,v\rangle -\langle \overline{\alpha _i^I},q(v)\rangle = 0$. Hence the composition $E_I \circ F_I$ is the identity. Since $F_I$ is continuous and its inverse $E_I$ is also continuous, $F_I$ is a homeomorphism, as required. 
	\end{proof}
	
	\begin{lemm}\Label{lemm:compact}
		Let $\Delta$ and $\mathfrak{h}$ satisfy conditions (1) and (2) above. Suppose that the fan $q(\Delta)$ in $\mathfrak{g}/p(\mathfrak{h})$ is complete. Then, the quotient space $X(\Delta)/H$ is compact. 
	\end{lemm}
	\begin{proof}
		Since $X(\Delta)$ has the open dense orbit $G^\C$ corresponding to the empty set $\emptyset \in \Sigma$, $X(\Delta)/H$ also has the open dense orbit $G^\C/H$. First, we identify the quotient $G^\C/G\cdot H$ of the open dense subset $G^\C/H$ by $G$ with the vector space $\mathfrak{g}/p(\mathfrak{h})$ as follows. To each $[u+\sqrt{-1}v] \in (\mathfrak{g}\oplus \sqrt{-1}\mathfrak{g})/(\mathfrak{g}+\mathfrak{h})$ where $u,v \in \mathfrak{g}$ and $[u+\sqrt{-1}v]$ denotes the equivalence class, we assign $[v] \in \mathfrak{g}/p(\mathfrak{h})$ where $[v]$ denotes the equivalence class of $v \in \mathfrak{g}$. This is well-defined because if $u'+\sqrt{-1}v' \in \mathfrak{h}$ then $v' \in p(\mathfrak{h})$. Moreover, this is a bijective linear map. Since $G^\C/H = \mathfrak{g}^\C/(\mathfrak{h}+\Hom (S^1,G))$, the correspondence $[u+\sqrt{-1}v] \mapsto [v]$ is nothing but an isomorphism $\iota \co G^\C/G\cdot H \to \mathfrak{g}/p(\mathfrak{h})$. 
		
		Since $q$ sends the fan $\Delta$ to a complete fan $q(\Delta)$ in $\mathfrak{g}/p(\mathfrak{h})$, every cone $C_J \in \Delta$ is contained in a cone $C_I \in \Delta$ such that $\dim C_I = \dim \mathfrak{g}/p(\mathfrak{h})$. Hence the open subsets
		\begin{equation*}
			\left((G^\C/H) \times _{(G_I)^\C}\bigoplus _{i \in I}\C_{\alpha _i^I}\right) /G \subset X(\Delta)/G\cdot H
		\end{equation*}
		for $I \in \Sigma$ such that $\dim C_I = \dim \mathfrak{g}/p(\mathfrak{h})$ cover the quotient space $X(\Delta )/G\cdot H$. Remark that each open subset $\left((G^\C/H) \times _{(G_I)^\C}\bigoplus _{i \in I}\C_{\alpha _i^I}\right) /G$ contains $G^\C/G\cdot H$ as an open dense subset. Now we consider the preimage of the cone $q(C_I)$ by $\iota$. Suppose that $\dim C_I  = \dim \mathfrak{g}/p(\mathfrak{h})$ and let $F_I$ be the homeomorphism as Lemma \ref{lemm:quotient}. Then, 
		\begin{equation*}
			F_I (\iota^{-1}(q(v))) = (e^{-2\pi \langle \overline{\alpha_i^I} ,q(v)\rangle})_{i \in I}. 
		\end{equation*}
		Therefore, 
		\begin{equation*}
			F_I (\iota^{-1}(q(C_I)))  = \left\{ (r_i)_{i\in I} \in \prod _{i \in I}(\R_{\geq 0})_i\mid 0<r_i\leq 1\right\} 
		\end{equation*}
		because $\langle \overline{\alpha_i^I},q(v)\rangle \geq 0$ for $q(v) \in q(C_I)$. The closure of right hand side is a cube. In particular, it is compact. Therefore, closure $K_I$ of $\iota ^{-1}(q(C_I))$ in $\left((G^\C/H) \times _{(G_I)^\C}\bigoplus _{i \in I}\C_{\alpha _i^I}\right) /G$ is also compact and hence $K_I$ coincides with the closure of $\iota ^{-1}(q(C_I))$ in $X(\Delta )/G\cdot H$. 
		
		Since $G^\C/G\cdot H$ is the union of $\iota ^{-1}(q(C_I))$ such that $\dim C_I = \dim \mathfrak{g}/p(\mathfrak{h})$ because $q(\Delta)$ is complete, $X(\Delta)/G\cdot H$ is the union of compact subsets $K_I$'s. Since the number of cones in $\Delta$ is finite, $X(\Delta)/G\cdot H$ is a finite union of compact subsets and hence it is compact. The lemma is proved. 
	\end{proof}
	
	\begin{prop}\Label{prop:construction}
		Let $\Delta$ be a nonsingular fan in $\mathfrak{g}$ and let $\mathfrak{h}$ be a complex subspace of $\mathfrak{g}^\C$. Suppose that $\Delta$ and $\mathfrak{h}$ satisfy the following conditions:
		\begin{enumerate}
			\item the restriction $p|_{\mathfrak{h}}$ of the projection $p \co \mathfrak{g}^\C \to \mathfrak{g}$ is injective and 
			\item the quotient map $q \co \mathfrak{g} \to \mathfrak{g}/p(\mathfrak{h})$ sends a fan $\Delta$ to a complete fan $q(\Delta )$ in $\mathfrak{g}/p(\mathfrak{h})$. 
		\end{enumerate}
		Then, the quotient space $X(\Delta)/H$ of the toric variety associated with $\Delta$ by $H= \exp(\mathfrak{h})$ is a compact complex manifold equipped with the action of $G$ which is maximal and preserves the complex structure. 
	\end{prop}
	\begin{proof}
		By Lemmas \ref{lemm:closed}, \ref{lemm:free}, \ref{lemm:Hausdorff} and \ref{lemm:compact}, $X(\Delta)/H$ is a compact complex manifold. It remains to show that the action of $G$ on $X(\Delta)/H$ is maximal. Since $q(\Delta)$ is complete, there exists a cone $C_I \in \Delta$ such that $\dim C_I = \dim \mathfrak{g}/p(\mathfrak{h})$. Therefore there exists a point $\widetilde{x} \in X(\Delta)$ such that $\dim G_{\widetilde{x}} = \dim \mathfrak{g}/p(\mathfrak{h})$. Let $x$ be the point in $X(\Delta)/H$ to which the quotient map $X(\Delta) \to X(\Delta)/H$ sends $\widetilde{x}$. Since $G \cap H = \{1\}$ by condition (1), $G_x = G_{\widetilde{x}}$. Therefore, 
		\begin{equation*}
			\dim G + \dim G_x = 2\dim G-\dim H = \dim X(\Delta)/H,
		\end{equation*}
		that is, the action of $G$ on $X(\Delta)/H$ is maximal, proving the lemma. 
	\end{proof}Now we are in a position to show the first theorem in this paper. 
	\begin{theo}\Label{theo:maintheo}
		Let $M$ be a compact connected complex manifold equipped with a maximal action of a compact torus $G$ which preserves the complex structure on $M$. Let $\Delta$ and $\mathfrak{h}$ be the fan and complex subspace of $\mathfrak{g}^\C$ defined as in Section \ref{sec:fan}. Then, $M$ is $G$-equivariantly biholomorphic to $X(\Delta)/H$. 
	\end{theo}
	\begin{proof}
		Let $N$ be the open submanifold of $M$ defined as in Section \ref{sec:fan}. Let $Z_{G^\C}$ be the global stabilizers of the action of $G^\C$ on $M$. By Corollary \ref{coro:quotient}, $N$ is $G$-equivariantly biholomorphic to the quotient space $X(\Delta)/Z_{G^\C}$. Since $\mathfrak{h}$ is the Lie algebra of the global stabilizer $Z_{G^\C}$ and $Z_{G^\C}$ is connected by Lemma \ref{lemm:stabilizers}, $Z_{G^\C} = \exp (\mathfrak{h}) = H$. By Proposition \ref{prop:construction}, the quotient space $X(\Delta)/H$ is compact. Therefore, $N$ is open and closed in $M$. Since $M$ is connected, $N$ is the whole manifold $M$ and hence $M$ is $G$-equivariantly biholomorphic to the  quotient $X(\Delta)/H$. The theorem is proved. 
	\end{proof}
	\begin{rema}\Label{rema:N}
		As we saw in the proof of Theorem \ref{theo:maintheo}, $N$ coincides with the whole manifold $M$. So the characteristic submanifolds of $N$ are closed submanifolds of $M$. Therefore we can call each of them a characteristic submanifold of $M$. As a result, we can deduce the fan $\Delta$ without passing the open submanifold $N$. 
	\end{rema}

\section{Category equivalence}\Label{sec:category}
Let $\mathcal{C}_1$ denote the class of all $(M,G,y)$ of a compact connected complex manifold $M$ equipped with a maximal action of a compact torus $G$ and a point $y \in M$ such that $G_y = \{1\}$. For $(M_1,G_1,y_1), (M_2,G_2,y_2) \in \mathcal{C}_1$, we say that a holomorphic map $\psi \co M_1 \to M_2$ is a morphism $\psi \co (M_1,G_1,y_1) \to (M_2,G_2,y_2)$ if 
\begin{enumerate}
	\item $\psi(y_1) = y_2$ and 
	\item there exists a group homomorphism $\alpha \co G_1 \to G_2$ such that $\psi$ is $\alpha$-equivariant. 
\end{enumerate}
Obviously $\mathcal{C}_1$ with the hom-sets $\Hom _{\mathcal{C}_1}((M_1,G_1,y_1),(M_2,G_2,y_2))$ forms a category. 

Let $\mathcal{C}_2$ denote the class of all triples $(\Delta, \mathfrak{h}, G)$ of a nonsingular fan $\Delta$ in $\mathfrak{g}$, a complex subspace $\mathfrak{h}$ of $\mathfrak{g}^\C$ such that 
\begin{enumerate}
	\item the restriction of $p|_{\mathfrak{h}}$ of the projection $p \co \mathfrak{g}^\C \to \mathfrak{g}$ is injective and 
	\item the quotient map $q \co \mathfrak{g}  \to \mathfrak{g}/p(\mathfrak{h})$ sends $\Delta$ to a complete fan $q(\Delta)$ in $\mathfrak{g}/p(\mathfrak{h})$, 
\end{enumerate}
and a compact torus $G$ whose Lie algebra is $\mathfrak{g}$. For $(\Delta_1,\mathfrak{h}_1, G_1), (\Delta_2,\mathfrak{h}_2, G_2) \in \mathcal{C}_2$, a morphism $\alpha \co (\Delta_1,\mathfrak{h}_1,G_1) \to (\Delta_2,\mathfrak{h}_2,G_2)$ is a group homomorphism $\alpha \co G_1 \to G_2$ such that 
\begin{enumerate}
	\item the differential $(d\alpha)_1 \co \mathfrak{g}_1 \to \mathfrak{g}_2$ at the unit of $G_1$ induces a morphism $\Delta_1 \to \Delta_2$ of fans and 
	\item $(d\alpha)_1^\C := (d\alpha)_1 \otimes \id_{\C} \co \mathfrak{g}_1^\C \to \mathfrak{g}_2^\C$ satisfies $(d\alpha_1)^\C(\mathfrak{h}_1) \subseteq \mathfrak{h}_2$. 
\end{enumerate}
$\mathcal{C}_2$ with the hom-sets $\Hom_{\mathcal{C}_2}((\Delta_1,\mathfrak{h}_1,G_1),(\Delta_2,\mathfrak{h}_2,G_2))$ also forms a category. In this section, we show that the categories $\mathcal{C}_1$ and $\mathcal{C}_2$ are equivalent. 

In Section \ref{sec:fan}, we assigned a nonsingular fan $\Delta$ in $\mathfrak{g}$ and a complex subspace $\mathfrak{h}$ of $\mathfrak{g}^\C$ to each compact connected complex manifold $M$ equipped with a maximal action of a compact torus $G$ which preserves the complex structure on $M$. This assignment defines the object function $\mathcal{F}_1 \co \mathcal{C}_1 \to \mathcal{C}_2$ by Lemma \ref{lemm:snake} and Corollary \ref{coro:complete}. For simplicity, we often denote by $\mathcal{F}_1(M,G) \in \mathcal{C}_2$ the value $\mathcal{F}_1(M,G,y)$ for $y \in M$ such that $G_y$ is trivial in the sequel. Remark that $\mathcal{F}_1(M,G,y)$ does not depend on the choice of $y$. 

Conversely, by Proposition \ref{prop:construction}, we have the object function $\mathcal{F}_2 \co \mathcal{C}_2 \to \mathcal{C}_1$ via $\mathcal{F}_2(\Delta, \mathfrak{h}, G) = (X(\Delta)/H, G, [1_{X(\Delta)}])$, where $1_{X(\Delta)}$ is the point $[1_{G^\C}, (1)_{i \in I}] \in G^\C \times _{(G_I)^\C} \bigoplus_{i \in I}\C_{\alpha_i^I} \subset X(\Delta)$ (see \eqref{eq:toric}). 

We would like to define a mapping function and use the same notation $\mathcal{F}_1$. To a morphism $\psi \in \Hom _{\mathcal{C}_1}((M_1,G_1,y_1),(M_2,G_2,y_2))$, we define $\mathcal{F}_1(\psi)$ to be the  group homomorphism $\alpha \co G_1 \to G_2$ if $\psi$ is $\alpha$-equivariant. Since $\psi(y_1) = y_2$, and $(G_1)_{y_1}$ and $(G_2)_{y_2}$ both are trivial, $\alpha$ is uniquely determined. 

\begin{lemm}\Label{lemm:F1}
	$\mathcal{F}_1$ satisfies the following: 
	\begin{enumerate}
		\item $\mathcal{F}_1(\id _{M_1}) = \id _G$ for $(M,G,y) \in \mathcal{C}_1$, 
		\item $\mathcal{F}_1(\psi_2 \circ \psi_1) = \mathcal{F}_1(\psi_2) \circ \mathcal{F}_1(\psi_1)$ for morphisms $\psi_1 \in \Hom _{\mathcal{C}_1}((M_1,G_1,y_1),(M_2,G_2,y_2))$ and $\psi_2 \in \Hom _{\mathcal{C}_2}((M_2,G_2,y_2),(M_3,G_3,y_3))$, 
		\item $\mathcal{F}_1(\psi) = \Hom _{\mathcal{C}_2}(\mathcal{F}_1(M_1,G_1),\mathcal{F}_1(M_2,G_2)$ for $\psi \in \Hom _{\mathcal{C}_1}((M_1,G_1,y_1),(M_2,G_2,y_2))$. 
	\end{enumerate}
	In particular, $\mathcal{F}_1 \co \mathcal{C}_1 \to \mathcal{C}_2$ is a covariant functor. 
\end{lemm}

\begin{proof}
		Parts (1) and (2) are obvious. 
		
		For $i=1,2$, let $(M_i,G_i,y_i) \in \mathcal{C}_1$ and suppose that $\mathcal{F}_1(M_i,G_i,y_i)=(\Delta_i,\mathfrak{h}_i,G_i)$. Let $\psi \in \Hom _{\mathcal{C}_1}((M_1,G_1,y_1),(M_2,G_2,y_2))$  and let $\mathcal{F}_1(\psi) = \alpha$. Let $\sigma$ be a cone in $\Delta$ and let $v \in \mathfrak{g}_1$ be an element which belongs to the relative interior of $\sigma$. Let $c_{y_1}^v \co \R \to M_1$ be the curve as Lemma \ref{lemm:limit} above. By Lemma \ref{lemm:limit}, the curve converges to a point in $M_1$ as $r$ approaches to $-\infty$. Moreover,  the limit point, say $x_\sigma$, does not vary with $v$ as long as $v$ belongs to the relative interior of $\sigma$.  Since $\psi \co M_1 \to M_2$ is an $\alpha$-equivariant holomorphic map such that $\psi (y_1)=y_2$, we have that $\psi \circ c_{y_1}^v (r)= c_{y_2}^{(d\alpha)_1(v)}(r)$. Therefore the curve $c_{y_2}^{(d\alpha)_1(v)}(r)$ converges to $\psi (x_\sigma)$ as $r$ approaches to $-\infty$. By Lemma \ref{lemm:limit}, $(d\alpha)_1(v) \in \mathfrak{g}_2$ belongs to the relative interior of a cone in $\Delta_2$ for all $v$ which belongs to the relative interior of $\sigma$. This shows that $(d\alpha)_1$ induces a morphism of fans from $\Delta_1$ to $\Delta_2$. 
		
		Since $\psi$ is $\alpha$-equivariant (with respect to $\alpha \co G_1 \to G_2$) and holomorphic, $\psi$ is an $\alpha$-equivariant holomorphic map with respect to the complexified homomorphism $\alpha \co G_1^\C \to G_2^\C$. Since $\mathfrak{h}_i$ is the Lie algebra of the global stabilizers of the action of $G_i^\C$ on $M_i$,  we have that $(d\alpha)_1^\C(\mathfrak{h}_1) \subseteq\mathfrak{h}_2$. Therefore $\mathcal{F}_1(\psi) =\alpha$ is a morphism in $\Hom _{\mathcal{C}_2}((\Delta_1,\mathfrak{h}_1,G_1),(\Delta_2,\mathfrak{h}_2,G_2))$, proving Part (3). 
	\end{proof}
Let $(\Delta_1,\mathfrak{h}_1,G_1), (\Delta_2,\mathfrak{h}_2,G_2) \in \mathcal{C}_2$ and $\alpha \in \Hom _{\mathcal{C}_2}((\Delta_1,\mathfrak{h}_1,G_1), (\Delta_2,\mathfrak{h}_2,G_2))$. Since $(d\alpha)_1 \co \mathfrak{g}_1 \to \mathfrak{g}_2$ induces a morphism of fans from $\Delta_1$ to $\Delta_2$, it also induces a toric morphism $\widetilde{\psi_\alpha} \co X(\Delta_1) \to X(\Delta_2)$ such that $\widetilde{\psi_\alpha}(1_{X(\Delta_1)}) = 1_{X(\Delta_2)}$. Since $(d\alpha_1)^\C (\mathfrak{h}_1) \subseteq \mathfrak{h}_2$, the toric morphism $\widetilde{\psi_\alpha}$ descends to an $\alpha$-equivariant holomorphic map $\psi_\alpha \co X(\Delta_1)/H_1 \to X(\Delta_2)/H_2$, where $H_i = \exp (\mathfrak{h}_i)$, $i =1,2$. 

Therefore , the morphism $\alpha$ induces a morphism
\begin{equation*}
	\psi_\alpha \in \Hom_{\mathcal{C}_1}((X(\Delta_1)/H_1,G_1,[1_{X(\Delta_1)}]),(X(\Delta_2)/H_2, G_2,[1_{X(\Delta_2)}])). 
\end{equation*}
One can see that the object function $\mathcal{F}_2 \co \mathcal{C}_2 \to \mathcal{C}_1$ together with the mapping function $\mathcal{F}_2$ (use the same notation) which assigns $\psi_\alpha$ to $\alpha$ is a covariant functor. 

\begin{theo}\Label{theo:equivalence}	
	The covariant functors $\mathcal{F}_1$ and $\mathcal{F}_2$ are weak inverse to each other. In particular, $\mathcal{C}_1$ and $\mathcal{C}_2$ are equivalent. 
\end{theo} 
Before the proof of Theorem \ref{theo:equivalence}, we prepare $2$ lemmas. 

\begin{lemm}\Label{lemm:equivariant}
	Let $(M,G,y_1), (M,G,y_2) \in \mathcal{C}_1$. Then, there exists a unique $G$-equivariant biholomorphism $\psi \co M \to M$ such that $\psi(y_1) =y_2$. 
\end{lemm}
\begin{proof}
	It follows from Theorem \ref{theo:maintheo} that the union of free $G$-orbits of $M$ coincides with the open dense $G^\C$-orbit of $M$. Since $y_1$ and $y_2$ sits in the same $G^\C$-orbit, there exists an element $g \in G^\C$ such that $g \cdot y_1 = y_2$. The biholomorphic map $\psi \co p \mapsto g\cdot p$ is what we wanted, proving the lemma. 
\end{proof}
\begin{lemm}\Label{lemm:equivariant2}
	Let $(M_1,G_1,y_1), (M_2,G_2,y_2) \in \mathcal{C}_1$ and let $\psi, \psi' \in \Hom _{\mathcal{C}_1}((M_1,G_1,y_1), (M_2,G_2,y_2))$ be $\alpha$-equivariant for the same $\alpha \co G_1 \to G_2$. Then, $\psi = \psi'$.
\end{lemm}
\begin{proof}
	As in the proof of Lemma \ref{lemm:equivariant}, $\psi$ coincides with $\psi'$ on the open dense $G_1^\C$-orbit of $M_1$ because $\psi(y_1) = \psi'(y_1)$. So $\psi = \psi'$, as required. 
\end{proof}

\begin{proof}[Proof of Theorem \ref{theo:equivalence}]
	Let $(\Delta, \mathfrak{h}, G) \in \mathcal{C}_2$. By Theorem \ref{theo:maintheo}, $\mathcal{F}_1 \circ \mathcal{F}_2 (\Delta, \mathfrak{h}, G) = (\Delta, \mathfrak{h}, G)$. For a morphism $\alpha$ in $\mathcal{C}_2$, the map $\mathcal{F}_2(\alpha) = \psi_\alpha$ is an $\alpha$-equivariant map. Therefore $\mathcal{F}_1\circ \mathcal{F}_2 (\alpha) = \alpha$, showing that $\mathcal{F}_2 \circ \mathcal{F}_1 = \id _{\mathcal{C}_2}$. 
	
	It remains to show the existence of the natural isomorphism $\eta$ between $\mathcal{F}_2 \circ \mathcal{F}_1$ and $\id_{\mathcal{C}_1}$. Let $(M, G, y) \in \mathcal{C}_1$ and let $(\Delta, \mathfrak{h}, G) =\mathcal{F}_1 (M, G,y) \in \mathcal{C}_2$. By Theorem \ref{theo:maintheo}, there exists a $G$-equivariant biholomorphic map $\psi \co M \to X(\Delta)/H$. Then, $\psi (y)$ sits in a free $G$-orbit. By Lemma \ref{lemm:equivariant}, there exists a $G$-equivariant biholomorphic map $\psi ' \co X(\Delta)/H \to X(\Delta)/H$ such that $\psi' \circ \psi(y) = [1_{X(\Delta)}]$. Therefore there exists an isomorphism $\eta_{(M,G,y)} := \psi' \circ \psi \co (M,G,y) \to (X(\Delta)/H, G, [1_{X(\Delta)}])$. By Lemma \ref{lemm:equivariant2}, such an isomorphism $\eta_{(M,G,y)}$ is unique. Let $(M_i,G_i,y_i) \in \mathcal{C}_1$ for $i=1,2$. Let $\psi \in \Hom _{\mathcal{C}_2}((M_1,G_1,y_1),(M_2,G_2,y_2))$ be a morphism and suppose that $\alpha = \mathcal{F}_1(\psi)$. Let $(\Delta_i, \mathfrak{h}_i, G_i) = \mathcal{F}_1(M_i,G_i,y_i)$ for $i=1,2$. Then, the morphism $\eta_{(M_2,G_2,y_2)} \circ \psi \circ \eta_{(M_1,G_1,y_1)}^{-1} \co (X(\Delta_1)/H_1,G_1,[1_{X(\Delta_1)}]) \to (X(\Delta_2)/H_2,G_2,[1_{X(\Delta_2)}])$ coincides with $\psi_\alpha$ by Lemma \ref{lemm:equivariant2}. It turns out that the collection $\eta$ of $\eta_{(M,G,y)}$ for all $(M,G,y) \in \mathcal{C}_1$ is a natural isomorphism between $\mathcal{F}_2 \circ \mathcal{F}_1$ and $\id _{\mathcal{C}_1}$. The theorem is proved. 
\end{proof}

Let $G_1$ and $G_2$ be groups. We say that a map $f \co X_1 \to X_2$ between a $G_1$-set $X_1$ and a $G_2$-set $X_2$ is \emph{weakly equivariant} if there exists a group isomorphism $\alpha \co G_1 \to G_2$ such that $f\circ g = \alpha (g) \circ f$ for all $g \in G_1$. Let $(M_i,G_i,y_i) \in \mathcal{C}_1$ for $i=1,2$. It immediately follows from Theorem \ref{theo:equivalence} that $M_1$ and $M_2$ are weakly equivariantly biholomorphic to each other if and only if $\mathcal{F}_1(M_1,G_1,y_1)$ and $\mathcal{F}_1(M_2,G_2,y_2)$ are isomorphic. 

We can also classify the (non-equivariant) biholomorphism types of compact connected complex manifolds which admit maximal actions of tori in terms of $\mathcal{C}_2$ by using an argument similar to \cite[Lemma 4.1]{Battisti}. 

\begin{theo}\Label{theo:biholo}
	Let $M_1$ and $M_2$ be compact connected complex manifolds, equipped with maximal actions of compact tori $G_1$ and $G_2$ which preserve the complex structures on $M_1$ and $M_2$, respectively. Then, $M_1$ and $M_2$ are biholomorphic to each other if and only if $\mathcal{F}_1(M_1,G_1)$ and $\mathcal{F}_1(M_2,G_2)$ are isomorphic. 
\end{theo}
\begin{proof}
	Since $M_1$ is weakly equivariantly biholomorphic to $M_2$ if and only if $\mathcal{F}_1(M_1,G_1)$ is isomorphic to $\mathcal{F}_1(M_2,G_2)$, the ``if" part is obvious. 
	
	We show the ``only if" part. The group $\Aut (M_i)$ of all biholomorphisms from $M_i$ to $M_i$ itself is a complex Lie group (see \cite[Section 9]{Bochner-Montgomery}). Let $\psi \co M_1 \to M_2$ be a biholomorphism. The biholomorphism $\psi \co M_1 \to M_2$ induces a complex Lie group isomorphism $\psi_* \co \Aut (M_1) \to \Aut (M_2)$. Lert $K_i$ be a maximal connected compact Lie subgroups of $\Aut (M_i)$ which contains $G_i$. Since maximal connected compact Lie subgroups $\psi_*(K_1)$ and $K_2$ of $\Aut (M_2)$ are conjugate to each other (see \cite[Chapter XV, Section 3]{Hochschild}), there exists an element $g \in \Aut(M_2)$ such that $g\psi_*(K_1)g^{-1} = K_2$. Moreover, since $g\psi_*(G_1)g^{-1}$ and $G_2$ are maximal tori in $K_2$, there exists an element $g' \in K_2$ such that $g'g\psi_*(G_1)g^{-1}g'^{-1} = G_2$. Put $g'' := \psi_*^{-1}(g'g)$ and then $g'' \circ \psi \circ g''^{-1} \co M_1 \to M_2$ is a biholomorphism which induces a group isomorphism $(g'' \circ \psi \circ g''^{-1})_*$ whose restriction to $G_1$ is an isomorphism onto $G_2$. Therefore $g'' \circ \psi \circ g''^{-1}$ is an $\alpha$-equivariant biholomorphism for some group isomorphism $\alpha \co G_1 \to G_2$. Let $y_1 \in M_1$ such that $(G_1)_{y_1}$ is trivial. Put $y_2 := g'' \circ \psi \circ g''^{-1}(y_1)$. Then, $(G_2)_{y_2}$ is trivial and $g'' \circ \psi \circ g''^{-1} \co (M_1,G_1,y_1) \to (M_2,G_2,y_2)$ is an isomorphism. Applying $\mathcal{F}_1$, we have that $\mathcal{F}_1(M_1,G_1,y_1)$ is isomorphic to $\mathcal{F}_1(M_2,G_2,y_2)$, proving the theorem. 
\end{proof}

\section{LVMB manifolds and moment-angle manifolds}\Label{sec:LVMBmam}
	In this section, we elaborate on the relationships between complex manifolds with maximal torus actions, LVMB manifolds (introduced by Bosio in \cite{Bosio}) and moment-angle manifolds with invariant complex structures (given by Tambour and Panov-Ustinovsky in \cite{Tambour} and \cite{Panov-Ustinovsky} independently). 
	
	First, we recall the notion of moment-angle complexes briefly. Let $\Sigma$ be an abstract simplicial complex on $\{1,\dots, n\}$ (we do not require each singleton $\{j\}$ to be an element of $\Sigma$). For each simplex $I \in \Sigma$, we set
	\begin{equation*}
		(\overline{\D},S^1)^I := \{ (z_1,\dots, z_n) \in \overline{\D}^n \mid z_j \in S^1 \text{ if $j \notin I$}\} \subset \C^n
	\end{equation*}
	and 
	\begin{equation*}
		\mathcal{Z}_\Sigma := \bigcup _{I \in \Sigma}(\overline{\D},S^1)^I.
	\end{equation*}
	Since each $(\overline{\D},S^1)^I$ is invariant under the natural $(S^1)^n$-action on $\C^n$, so is $\mathcal{Z}_\Sigma$. The topological space $\mathcal{Z}_\Sigma$ is called a \emph{moment-angle complex}. Moment-angle complexes play an important role in toric topology (see \cite[Chapter 6]{Buchstaber-Panov} for details). For examples, $\mathcal{Z}_\Sigma$ is a topological $(n+d)$-manifold when the geometric realization $|\Sigma|$ of $\Sigma$ is homeomorphic to a $(d-1)$-sphere, and becomes a total space of a principal torus bundle over a complete nonsingular toric variety $X(\Delta)$ when $\Sigma$ is isomorphic to the underlying simplicial complex of $\Delta$.  In the case when $\mathcal{Z}_\Sigma$ is a topological manifold, we call $\mathcal{Z}_\Sigma$ \emph{the moment-angle manifold}. In \cite[Construction 3.1 and Theorem 3.3]{Panov-Ustinovsky} and \cite[COROLLARY 5.6]{Tambour}, complex structures invariant under the natural torus actions are constructed under some conditions on $\Sigma$.

	We recall the construction of LVMB manifolds. Let $m, n$ be positive integers such that $2m \leq n$ and let $\mathcal{L} := (\ell_0,\dots, \ell_n)$ be $n+1$ linear forms of $\C ^m$ such that any subfamily of $2m+1$ elements of $\mathcal{L}$ spans $(\C^m)^\vee$ $\R$-affinely, where $(\C^m)^\vee$ denotes the dual space of $\C^m$. Let $\mathcal{E}_{m,n}$ denote a family of $(2m+1)$-subsets of $\{0,\dots, n\}$. For every $P \in \mathcal{E}_{m,n}$, let $\mathcal{L}_P$ denote the corresponding subfamily of $\mathcal{L}$. An \emph{LVMB datum} is a pair $(\mathcal{L}, \mathcal{E}_{m,n})$ which satisfies the following two conditions:
	\begin{enumerate}
		\item For all $P \in \mathcal{E}_{m,n}$ and $i \in \{ 0,\dots, n\}$, there exists $j \in P$ such that $(P \setminus \{j\}) \cup \{i\} \in \mathcal{E}_{m,n}$. 
		\item For all $P, Q \in \mathcal{E}_{m,n}$, the interiors of the convex hulls of $\mathcal{L}_P$ and $\mathcal{L}_Q$ have non-empty intersection. 
	\end{enumerate}
	For each $P \in \mathcal{E}_{m,n}$, we set an open subset 
	\begin{equation*}
		U_P := \{ z =[z_0,\dots, z_n] \in \C P^n \mid z_i \neq 0 \text{ if $i \in P$} \}
	\end{equation*}
	and the union $U := \bigcup_{ P \in \mathcal{E}_{m,n}}U_P$.

	Define an action of $H = \C ^m$ on $U$ by 
	\begin{equation*}
		v \cdot [z_0,\dots, z_n] := [e^{\ell_0(v)}z_0,\dots, e^{\ell _n(v)}z_n] \quad \text{for $v \in \C^m$}. 
	\end{equation*}
	Bosio showed that the action of $H$ on $U$ is free, proper and cocompact. Hence the quotient space $X_{n,m} := U/H$ is a compact complex manifold of complex dimension $n-m$ and called an \emph{LVMB manifold}. 	
	
	Let $G = (S^1)^n$ act on $\C P^n$ by 
	\begin{equation*}
		(g_1,\dots, g_n) \cdot [z_0,\dots, z_n] := [z_0,g_1z_1,\dots, g_nz_n]
	\end{equation*}
	for $(g_1,\dots ,g_n) \in G$ and $[z_0,\dots, z_n] \in \C P^n$. 
	By definition of $U$, $U$ is invariant under the action of $G$. Since the action of $G$ commutes with the action of $H$, the action of $G$ on $U$ descends to an action on the LVMB manifold $X_{n,m} = U/H$. 	
	\begin{prop}\Label{prop:LVMBmax}
		Any LVMB manifold carries a maximal torus action which preserves the complex structure.
	\end{prop}
	\begin{proof}
		We use the same notations as above. Obviously, the $G$-action on $X_{n,m}=U/H$ preserves the complex structure on $X_{n,m}$. We need to show that there exists a point $x \in X=U/H$ such that $n+\dim G_x = 2(n-m)$. Since the quotient map $\pi : U \to X$ is isovariant (that is, $G_z =G_{\pi(z)}$ for any $z \in U$), it suffices to find a point $z \in U$ such that $\dim G_z =n-2m$. Let $P \in \mathcal{E}_{n,m}$. Put 
		\begin{equation*}
			z_i := \begin{cases}
				0 & \text{if $i \notin P$,}\\
				1 & \text{if $i \in P$.}
			\end{cases}
		\end{equation*}
		Then, the point $z := [z_0, \dots, z_n] \in \C P^n$ sits in $U_P$. In particular, $z \in U$. 
		If $0 \in P$, then $z_0 = 1$. In this case, the isotropy subgroup $G_z$ is the subgroup of elements $g \in G$ such that $g_i =1$ for all $i \in P \setminus \{0\}$, and hence $\dim G_z = n-2m$. If $0 \notin P$, then $z_0=0$. In this case, the isotropy subgroup $G_z$ is generated by the elements $g \in G$ such that $g_i=1$ for all $i \in P$ and the diagonal elements $(g_0,\dots, g_0) \in G$. Therefore, in the case when $0 \notin P$, $\dim G_z = n-2m$, too. This together with the case when $0 \in P$ shows that $\dim G_z =n-2m$, proving the proposition. 
		\end{proof}
		The following theorem characterizes LVMB manifolds in complex manifolds with maximal torus actions.
		\begin{theo}\Label{theo:LVMBchar}
			\begin{enumerate}
				\item Let $X_{n,m}$ be an LVMB manifold and let $(\Delta, \mathfrak{h},G) = \mathcal{F}_1(X_{n,m},G)$. Then, the fan $\Delta$ is a subfan of the fan of $\C P^n$. 
				\item Conversely, for $(\Delta, \mathfrak{h}, G) \in \mathcal{C}_2$ such that $\Delta$ is isomorphic to a subfan of the fan of $\C P^n$, the quotient space $X(\Delta)/H$ is an LVMB manifold. Namely, there exists an LVMB datum such that the corresponding LVMB manifold is the same as $X(\Delta)/H$. 
			\end{enumerate}
		\end{theo}
		\begin{proof}
			The Part (1) follows from the construction of the LVMB manifolds immediately. The Part (2) is due to \cite[Theorem 2.2]{Battisti}.
		\end{proof}
		
		For an LVMB datum $(\mathcal{L}, \mathcal{E}_{m,n})$, we say that an integer $i \in \{0,\dots, n\}$ is \emph{indispensable} if $i \in P$ for all $P \in \mathcal{E}_{m,n}$ (cf.~\cite{Bosio}, \cite{Meersseman} and \cite{Meersseman-Verjovsky}). There are $2$ types of LVMB data; one is the case with an indispensable integer and the other one is the case with no indispensable integer. 
		We shall clarify the relationships between indispensable integers and characteristic submanifolds. Let $(\mathcal{L}, \mathcal{E}_{m,n})$ be an LVMB datum and let $X_{n,m}$ be the corresponding LVMB manifold. Let $G=(S^1)^n$ act on $X_{n,m}$ as Proposition \ref{prop:LVMBmax} above. Let $(\Delta, \mathfrak{h}, G) = \mathcal{F}_1(X_{m,n},G)$. Suppose that $i$ is an indispensable integer. Then, each $U_P \subset \C P^n$ for $P \in \mathcal{E}_{m,n}$ does not contain points $z=[z_0,\dots, z_n]$ such that $z_i =0$. In particular, the fixed point set $U^{G_i}$ of $G_i$ in $U$ is empty, where
		\begin{equation*}
			G_i := \begin{cases}
				\{ (g,\dots, g) \in G \mid g \in S^1\} & \text{if $i =0$,}\\
				\{ (\underbrace{1,\dots,1}_{i-1}, g, \underbrace{1,\dots, 1}_{n-i}) \in G \mid g \in S^1\} & \text{otherwise}. 
			\end{cases}
		\end{equation*}
		Therefore, $X_{n,m}^{G_i} = \emptyset$. In particular, the fan $\Delta$ does not contain the $1$-cone $\tau_i$, where 
		\begin{equation*}
			\tau_i := \begin{cases}
				\pos (-e_1-\dots -e_n) & \text{if $i=0$,}\\
				\pos (e_i) & \text{otherwise,}
			\end{cases}
		\end{equation*}
		where $e_i$ denotes the $i$-th standard basis vector of $\R^n$ and $\pos (A)$ denotes the cone spanned by elements in $A$.
		The converse is also true. Namely, $i$ is indispensable if and only if $\Delta$ does not contain $\tau_i$. In particular, the number of indispensable points tells us the fundamental group of the LVMB manifold. Let $\mathcal{F}_1(X_{n,m}, G) = (\Delta,\mathfrak{h},G)$. Since $X_{n,m} \cong X(\Delta)/H$ and $H$ is contractible, $X_{n,m}$ is homotopy equivalent to $X(\Delta)$. Therefore, the fundamental group $\pi_1(X_{n,m})$ is isomorphic to $\Z^{k-1}$ if the LVMB datum has $k$ indispensable integers and $k>0$. In case when $k=0$, $X_{n,m}$ is simply connected. 
		
		Let $(\mathcal{L}, \mathcal{E}_{m,n})$ be an LVMB datum such that $0$ is indispensable. 
		Then, it follows from the observation above that the fan $\Delta$ is a subfan of the fan of $\C^n$. In this case, we can regard $U$ as an open subset of $\C^n$, not only of $\C P^n$. In fact, for any $z =[z_0,\dots, z_n] \in U$, $z_0  \neq 0$ and hence we have a biholomorphic map $f$ from $U$ onto an open subset of $\C^n$ via 
		\begin{equation*}
			f([z_0,\dots, z_n]) := \left( \frac{z_1}{z_0},\dots, \frac{z_n}{z_0}\right).
		\end{equation*}
		We shall see the image $f(U)$ of $U$ by $f$ for later use. 
		We define an abstract simplicial complex $\Sigma$ on $\{1,\dots, n\}$ given by
		\begin{equation*}
			\Sigma := \{ I \mid I \subseteq \{0,\dots, n\} \setminus P,\ ^\exists P \in \mathcal{E}_{m,n}\}. 
		\end{equation*}
		For each $I \in \Sigma$, we set 
		\begin{equation*}
			U(I) := \{ w =(w_1,\dots, w_n) \in \C ^n \mid w_i \neq 0 \text{ if $i \notin I$}\} 
		\end{equation*}
		and 
		\begin{equation*}	
			U(\Sigma) := \bigcup_{I \in \Sigma} U(I). 
		\end{equation*}
		Then, by definition of $f$, we have that $f(U_P) = U(\{0,\dots, n\} \setminus P)$ for each $P \in \mathcal{E}_{m,n}$. Therefore $f(U)$ coincides with $U(\Sigma)$. It turns out that the LVMB manifold $X_{n,m}$ is a quotient of the open subset $U(\Sigma)$ by an action of $\C^m$. 

		We shall state the following Lemma for later use.
		\begin{lemm}\Label{lemm:LVMB0indis}
			\begin{enumerate}
				\item Let $(\mathcal{L}, \mathcal{E}_{m,n})$ be an LVMB datum such that $0$ is indispensable and let $X_{n,m}$ be the LVMB manifold corresponding to $(\mathcal{L},\mathcal{E}_{m,n})$. Let $(\Delta, \mathfrak{h},G) = \mathcal{F}_1(X_{n,m},G)$. Then, $\Delta$ is a subfan of the fan of $\C^n$.
				\item Conversely, for $(\Delta, \mathfrak{h}, G) \in \mathcal{C}_2$ such that $\Delta$ is isomorphic to a subfan of the fan of $\C^n$, the quotient space $X(\Delta)/H$ is an LVMB manifold which comes from an LVMB datum such that $0$ is indispensable. 
			\end{enumerate}
		\end{lemm}
		
		By comparing Lemma \ref{lemm:LVMB0indis} with \cite[COROLLARY 5.6]{Tambour} and \cite[Construction 3.1 and Theorem 3.3]{Panov-Ustinovsky}, it follows that the LVMB manifold which comes from an LVMB datum such that $0$ is indispensable is equivariantly homeomorphic to a moment-angle manifold $\mathcal{Z}_\Sigma$. The following is an application of our classification theorem Theorem \ref{theo:maintheo}.

		\begin{theo}\Label{theo:mam}
			An even dimensional moment-angle manifold $\mathcal{Z}_\Sigma$ admits an invariant complex structure if and only if $\Sigma$ is star-shaped\footnote{A $(d-1)$-simplicial complex whose geometric realization is homeomorphic to $(d-1)$-sphere is said to be  \emph{star-shaped} if $\Sigma$ is isomorphic to the underlying simplicial complex of a simplicial complete fan in $\R ^d$.} and if and only if it is biholomorphic to an LVMB manifold which comes from an LVMB datum such that $0$ is indispensable. 
		\end{theo}
		\begin{proof}
			The ``if" part follows from \cite[COROLLARY 5.6]{Tambour} and \cite[Construction 3.1 and Theorem 3.3]{Panov-Ustinovsky}. We will show the ``only if" part. Let $\Sigma$ be an abstract simplicial complex on $\{1,\dots, n\}$ such that $|\Sigma|$ is a $(d-1)$-sphere. Suppose that the moment-angle manifold $\mathcal{Z}_\Sigma$ has an invariant complex structure and fix it. First, we claim that the $(S^1)^n$-action on $\mathcal{Z}_\Sigma$ is maximal. Since $\Sigma$ is homeomorphic to an $(d-1)$-sphere, there exists $I \in \Sigma$ such that the cardinarity $|I|$ of $I$ is equal to $d$. Put 
			\begin{equation*}
				x_i := \begin{cases}
					0 & \text{if $i \in I$}\\
					1 & \text{if $i \notin I$}.
				\end{cases}
			\end{equation*}
			Then, $x:=(x_1,\dots, x_n) \in (\D,S^1)^I \subset \mathcal{Z}_\Sigma$. The isotropy subgroup of $(S^1)^n$ at $x$ is the subgroup of the elements $g=(g_1,\dots, g_n)$ such that $g_i =1$ if $i \notin I$. Therefore $\dim G_x = d$. Since $\dim \mathcal{Z}_\Sigma = n+d$, the action of $(S^1)^n$ on $\mathcal{Z}_\Sigma$ is maximal. 
			
			Now we identify the Lie algebra of  $S^1$ with $\R$ through the differential of the exponential map $t \mapsto e^{2\pi\sqrt{-1}t}$. Then, the nonsingular fan $\Delta$ for $\mathcal{F}_1(\mathcal{Z}_\Sigma,(S^1)^n) = (\Delta, \mathfrak{h}, (S^1)^n) \in \mathcal{C}_2$ is as follows. The fan $\Delta$ in $\R^m$ is the collection of cones 
	\begin{equation*}
		\{ \pos (s_ie_i \mid i \in I) \mid I \in \Sigma\}
	\end{equation*}
	where $s_i =\pm 1$. Clearly, $\Delta$ is isomorphic to a subfan of the fan of $\C^n$. 
	The Lie algebra of $((S^1)^n)^\C = (\C^*)^n$ is identified with $\C^m$ throuth the exponential map. The complex subspace $\mathfrak{h}$ of $\R^m \otimes _\R \C = \C^m$ satisfies the conditions
	\begin{enumerate}
		\item the restriction $p|_{\mathfrak{h}}$ of the projection to the real part $p \co \C^m \to \R^m$ is injective,
		\item the quotient map $q \co \R^m \to \R^m/p(\mathfrak{h})$ sends $\Delta$ to a complete fan in $\R^m/p(\mathfrak{h})$
	\end{enumerate}
	by Lemma \ref{lemm:snake} and Corollary \ref{coro:complete}. By condition (2), the geometric realization $|\Sigma|$ of the simplicial complex $\Sigma$ should be homeomorphic to a sphere and $\Sigma$ should be star-shaped. 
	
	Since the fan $\Delta$ is isomorphic to a subfan of a fan of $\C^n$, it follows from Lemma \ref{lemm:LVMB0indis} that $\mathcal{Z}_\Sigma$ is an LVMB manifold which comes from an LVMB datum such that $0$ is indispensable. The theorem is proved. 
		\end{proof}
		\begin{rema}
		In \cite{Lu-Panov}, a moment-angle complex $\mathcal{Z}_\Sigma$ for a simplicial poset $\Sigma$ is introduced and has nice properties as well as the case when $\Sigma$ is a simplicial complex. For example, $\mathcal{Z}_\Sigma$ is also a manifold when the geometric realization of $\Sigma$ is homeomorphic to a sphere. One can see that  the action of $(S^1)^m$ on $\mathcal{Z}_\Sigma$ is maximal. However, our classification Theorem \ref{theo:maintheo} yields that $\Sigma$ should be a simplicial complex when $\mathcal{Z}_\Sigma$ admits a complex structure invariant under the action of $(S^1)^m$. 
	\end{rema}
\section{Equivariant principal bundles}\Label{sec:epb}
	In this section, we study on the equivariant principal bundles in complex manifolds with maximal torus actions. 
	It has been shown that certain LVM manifolds carry principal bundles over projective quasi-regular toric varieties. Conversely, starting from any quasi-regular projective toric variety, a principal bundle over it whose total space is an LVM manifold can be constructed (see \cite{Meersseman-Verjovsky} for details). It turns out that any projective smooth toric variety can be obtained from an LVMB manifold as a quotient space. 
	
	The main purpose in this section is to show that LVMB manifolds appear as building blocks of complex manifolds with maximal torus actions. We begin with the following general consequence. 

	\begin{theo}\Label{theo:principal}
		Let $(\Delta_i,\mathfrak{h}_i,G_i) \in \mathcal{C}_2$ for $i=1,2$  and let $\alpha \in \Hom _{\mathcal{C}_2}((\Delta_1,\mathfrak{h}_1,G_1), (\Delta_2,\mathfrak{h}_2,G_2))$. Let $G' \subset G_1$ denote the kernel of $\alpha$ (as a group homomorphism). Let $H_i$ denote the image of $\mathfrak{h}_i$ by the exponential map for $i=1,2$. Then, 
		\begin{enumerate}
			\item the induced $\alpha$-equivariant map $\mathcal{F}_2(\alpha) = \psi_\alpha \co X(\Delta_1)/H_1 \to X(\Delta_2)/H_2$ is $G'$-invariant.
			\item $\psi_\alpha$ is a principal $G'$-bundle if and only if $\alpha$ is surjective, and $(d\alpha)_1$ gives a one-to-one correspondence from the primitive generators of $1$-cones in $\Delta_1$ to the primitive generators of $1$-cones in $\Delta_2$. 
		\end{enumerate}
	\end{theo}
	\begin{proof}
		Part (1) follows from the fact that the toric morphism $\widetilde{\psi_\alpha} \co X(\Delta_1) \to X(\Delta_2)$ is $G'$-invariant.
		
		For Part (2), suppose that $\psi_\alpha$ is a principal $G'$-bundle. Then, $\alpha \co G_1 \to G_2$ is surjective; otherwise, $\psi_\alpha$ is not surjective. For each characteristic submanifold of $X(\Delta_1)/H_1$, the image of it by $\psi_\alpha$ is also a characteristic submanifold of $X(\Delta_2)/H_2$ and vice versa. Let $N_1,\dots, N_k$ be the characteristic submanifolds of $X(\Delta_1)/H_1$. For each point $x \in N_i$, the normal vector space $T_x(X(\Delta_1)/H_1)/T_xN_i$ is isomorphic to $T_{\psi_\alpha (x)}(X(\Delta_2)/H_2)/T_{\psi_\alpha(x)}(\psi_\alpha(N_i))$ through the differential $(d\psi_\alpha)_x$. Let $\lambda_i \in (S^1,G_1)$ be the homomorphism such that 
		\begin{equation*}
			(d\lambda_i(g))_x(\xi ) = g \xi \quad \text{for $g \in S^1$, $\xi \in T_x(X(\Delta_1)/H_1)/T_xN_i$},
		\end{equation*}
		that is, the primitive generator of the $1$-cone corresponding to $N_i$. Applying $(d\psi_\alpha)_x$, we have 
		\begin{equation*}
			(d\psi_\alpha)_x((d\lambda_i(g))_x(\xi))=g(d\psi_\alpha)_x(\xi).
		\end{equation*}
		Since $\psi_\alpha$ is $\alpha$-equivariant, we have
		\begin{equation*}
			(d(\alpha\circ \lambda_i(g)))_x((d\psi_\alpha)_x(\xi)) = g(d\psi_\alpha)_x(\xi).
		\end{equation*}
		Therefore, the primitive generator of the $1$-cone in $\Delta_2$ corresponding to $\psi_\alpha(N_i)$ is $\alpha\circ\lambda_i = (d\alpha)_1(\lambda_i)$, proving the ``only if" part. 
		
		To show the ``if" part, suppose that $\alpha$ is surjective and $(d\alpha)_1$ gives a one-to-one correspondence from the primitive generators of $1$-cones in $\Delta_1$ to the primitive generators of $1$-cones in $\Delta_2$. Let $p_i \co \mathfrak{g}_i^\C \to \mathfrak{g}_i$ be the projection and let $q_i \co \mathfrak{g}_i \to \mathfrak{g}_i/p_i(\mathfrak{h}_i)$ be the quotient map for $i=1,2$. Since $(\Delta_i,\mathfrak{h}_i, G_i) \in \mathcal{C}_2$, $q_i(\Delta_i)$ is a complete fan in $\mathfrak{g}_i/p_i(\mathfrak{h}_i)$. Since $\alpha \in \Hom _{\mathcal{C}_2}((\Delta_1,\mathfrak{h}_1,G_1),(\Delta_2,\mathfrak{h}_2,G_2))$ and $\alpha \co G_1\to G_2$ is surjective as a group homomorphism, $(d\alpha)_1 \co \mathfrak{g}_1 \to \mathfrak{g}_2$ is surjective and $(d\alpha)_1^\C (\mathfrak{h}_1) \subseteq \mathfrak{h}_2$. Therefore $(d\alpha)_1$ induces a surjective linear map $\overline{(d\alpha)_1} \co \mathfrak{g}_1/p_1(\mathfrak{h}_1) \to \mathfrak{g}_2/p_2(\mathfrak{h}_2)$. Let $\lambda_1,\dots, \lambda_k$ be the primitive generators of $1$-cones in $\Delta_1$. Then, $\nu_1=(d\alpha)_1(\lambda_1),\dots, \nu_k=(d\alpha)_1(\lambda_k)$ are the primitive generators of $1$-cones in $\Delta_2$. Put 
		\begin{equation*}
			\Sigma := \left\{ I \subseteq \{1,\dots, k\} \mid \pos(\lambda_i \mid i \in I) \in \Delta_1\right\}
		\end{equation*}
		Then, $\Delta_1 = \{ C_I := \pos (\lambda_i \mid i \in I) \mid {I \in \Sigma}\}$ and $q_1(\Delta_1) = \{ q_1(C_I)=\pos(q_1(\lambda_i)\mid i \in I) \mid {I \in \Sigma}\}$. Since $\alpha$ is a morphism, $(d\alpha)_1(C_I)$ is contained in a cone in $\Delta_2$ for each $I \in \Sigma$. Therefore, $\overline{(d\alpha)_1}(q_1(C_I))$ is contained in a cone in $q_2(\Delta_2)$. This together with the fact that $\overline{(d\alpha)_1}$ is surjective shows that the image $\overline{(d\alpha)_1}(q_1(\Delta_1))$ of the complete fan $q_1(\Delta_1)$ by $\overline{(d\alpha)_1}$ is a refinement of the complete fan $q_2(\Delta_2)$, in particular, $\overline{(d\alpha)_1}$ is an isomorphism. However, the number of $1$-cones in $\Delta_1$ is same as $1$-cones in $\Delta_2$ and hence $\overline{(d\alpha)_1}(q_1(\Delta_1))=q_2(\Delta_2)$. Therefore, $\Delta_2 =\{ \pos(\nu_i \mid i \in I)\}_{I \in \Sigma}$. 
		
		Let us think of the toric varieties $X(\Delta_1)$ and $X(\Delta_2)$ as the quotient spaces (see \eqref{eq:toric})
		\begin{equation*}
			X(\Delta_1) = \bigsqcup_{I \in \Sigma} G_1^\C \times_{G_{1,I}^\C} \bigoplus_{i \in I} \C_{\alpha_i^I} / \approx 
		\end{equation*}
		and 
		\begin{equation*}
			X(\Delta_2) = \bigsqcup_{I \in \Sigma} G_2^\C \times_{G_{2,I}^\C} \bigoplus_{i \in I} \C_{\beta_i^I} / \approx,
		\end{equation*}
		where $\alpha_i^I$ (respectively, $\beta_i^I$) for $i \in I$ denote the dual basis of $\lambda_i$ (respectively, $\nu_i$), $i \in I$. Then, the restriction of the toric morphism $\widetilde{\psi_\alpha}$ to each affine toric variety $G_1^\C \times_{G_{1,I}^\C} \bigoplus_{i \in I} \C_{\alpha_i^I}$ is given by 
		\begin{equation*}
			\widetilde{\psi_\alpha}([g,(z_i)_{i\in I}]) = [\alpha(g),(z_i)_{i \in I}] \in G_2^\C \times_{G_{2,I}^\C} \bigoplus_{i \in I} \C_{\beta_i^I}. 
		\end{equation*}
		Therefore the $\alpha$-equivariant holomorphic map $\psi_\alpha \co X(\Delta_1)/H_1 \to X(\Delta_2)/H_2$ is given by 
		\begin{equation*}
			{\psi_\alpha}([[g],(z_i)_{i\in I}]) = [[\alpha(g)],(z_i)_{i \in I}] \in (G_2^\C/H_2) \times_{G_{2,I}^\C} \bigoplus_{i \in I} \C_{\beta_i^I}
		\end{equation*}
		for each $[[g],(z_i)_{i \in I}]$ in the open subset $(G_1^\C/H_1) \times_{G_{1,I}^\C} \bigoplus_{i \in I} \C_{\alpha_i^I}$ of $X(\Delta_1)/H_1$. 
		Since the homomorphism $G_1^\C/H_1 \to G_2^\C/H_2$, induced by the complexified homomorphism $\alpha \co G_1^\C \to G_2^\C$, induces an isomorphism from $G_{1,I}$ to $G_{2,I}$ for each $I \in \Sigma$, 
		the kernel of the induced homomorphism $G_1^\C/H_1 \to G_2^\C/H_2$, say $\widetilde{G'}$, acts 
		on the preimage of each point $y \in X(\Delta)_2/H_2$ by $\psi_\alpha$ simply transitively. Therefore the $\alpha$-equivariant holomorphic map ${\psi_\alpha} \co X(\Delta_1)/H_1 \to X(\Delta_2)/H_2$ is a principal $\widetilde{G'}$-bundle. Since $X(\Delta_1)/H_1$ is compact, $\widetilde{G'}$ is also compact. Therefore $\widetilde{G'}$ is contained in $G_1$ (regarded as a subgroup of $G_1^\C /H_1$) and hence $\widetilde{G'}$ is the kernel $G'$ of $\alpha \co G_1 \to G_2$, proving the ``if" part. 
	\end{proof}
	\begin{theo}\Label{theo:universal}
		For any compact connected complex manifold $M$ equipped with a maximal action of a compact torus $G$ which preserves the complex structure on $M$, there exist an LVMB manifold $X_{n,m}$ and an equivariant principal holomorphic bundle $\psi \co X_{n,m} \to M$. In particular, every compact complex manifold with a maximal torus action is a quotient of an LVMB manifold. 
	\end{theo}
	\begin{proof}
		Let $(\Delta, \mathfrak{h}, G) = \mathcal{F}_1(M,G,y)$ and let $\lambda_1,\dots, \lambda_k$ be the primitive generators of $1$-cones in $\Delta$. Let $n$ be a positive integer such that $n-\dim G$ is even and there exists a surjective $\R$-linear map $f \co \R^n \to \mathfrak{g}$ such that $f(e_i) = \lambda_i$ for $i=1,\dots, k$ and $f(\Z^n) \subseteq \Hom (S^1,G)$, where $e_i$ denotes the $i$-th standard basis vector in $\R^m$. Since $f(e_i) \in \Hom (S^1,G)$ for all $i$ and $f$ is surjective, $f$ descends to a surjective group homomorphism $\alpha \co (S^1)^n \to G$ such that $(d\alpha)_1 =f$. Let $\Sigma$ be the simplicial complex on $\{1,\dots, n\}$ defined as 
		\begin{equation*}
			\Sigma := \{ I \subset \{1,\dots, n\} \mid \pos (\lambda_i \mid i \in I) \in \Delta\}.
		\end{equation*}
		Define a subfan $\Delta'$ of the fan of $\C^n$ as 
		\begin{equation*}
			\Delta' := \{ \pos (e_i \mid i\in I)\mid I \in \Sigma\}.
		\end{equation*}
		We want to find a complex vector subspace $\mathfrak{h}'$ of $\C^n$ such that $p'|_{\mathfrak{h}'}$ is injective, $(d\alpha)_1^\C(\mathfrak{h}') \subseteq \mathfrak{h}$ and the quotient map $q' \co \R^n \to \R^n/p'(\mathfrak{h}')$ sends $\Delta'$ to a complete fan in $\R^n/p'(\mathfrak{h}')$, where $p'$ denotes the projection from $\C^n$ to its real part. If such a vector subspace $\mathfrak{h}'$ exists, then $(\Delta', \mathfrak{h}', (S^1)^n) \in \mathcal{C}_2$ and the manifold $X(\Delta')/H'$ is an LVMB manifold $X_{n,m}$ by Lemma \ref{lemm:LVMB0indis}, where $H'$ is the image of $\mathfrak{h}'$ by the exponential map. By Theorem \ref{theo:principal}, the induced $\alpha$-equivariant map $\psi_\alpha \co X(\Delta')/H' \to X(\Delta)/H$ is a holomorphic principal equivariant $\ker \alpha$-bundle. Since $X(\Delta)/H$ is equivariantly biholomorphic to $M$, we have that there exists a principal holomorphic bundle $\psi \co X_{n,m} \to M$. 
		
		Now we show the existence of such a vector subspace $\mathfrak{h}'$. Let $p \co \mathfrak{g}^\C \to \mathfrak{g}$ be the projection and let $q \co \mathfrak{g} \to \mathfrak{g}/p(\mathfrak{h})$ be the quotient map. Let $w_1,\dots, w_\ell$ be a complex basis of $\mathfrak{h}$. Put 
		\begin{equation*}
			u_i := p(w_i) \quad \text{for $i =1,\dots, \ell$}
		\end{equation*}
		and 
		\begin{equation*}
			v_i := p(\sqrt{-1}w_i) \quad \text{for $i=1,\dots, \ell$}. 
		\end{equation*}
		Then, $u_1,\dots, u_\ell, v_1,\dots, v_\ell$ form a basis of $p(\mathfrak{h})$ because $p|_{\mathfrak{h}}$ is injective. 
		Since $f$ is surjective, there exist $u_i'$ and $v_i'$ in $\R^n$ such that $f(u_i') =u_i$ and $f(v_i') = v_i$ for $i=1,\dots, \ell$. Each of $u_i'$ and $v_i'$ belongs to $\ker q \circ f$ because $q(u_i) = q(v_i) =0$. Since $n-\dim G$ is even and $f$ is surjective, the dimension of $\ker f$ is even, say $2m-2\ell$. Let $u_{\ell+1}', \dots, u_m',v_{\ell+1}',\dots, v_m'$ be a basis of $\ker f$ and put 
		\begin{equation*}
			w_i' := u_i' +\sqrt{-1}v_i' \in \C^n
		\end{equation*}  
		for $i=1,\dots, m$. Let $\mathfrak{h}'$ be the complex vector subspace of $\C^n$ spanned by $w_i'$ for $i=1,\dots, m$. 
		
		To see the injectivity of $p'|_{\mathfrak{h}'}$, consider the linear combination 
		\begin{equation*}
			\sum_{i=1}^m a_iw_i', \quad a_i \in \C. 
		\end{equation*}
		Let $b_i$ and $\sqrt{-1}c_i$ be the real and imaginary parts of $a_i$, respectively. Suppose that $p'(\sum _{i=1}^ma_iw_i') =0$. Then,
		\begin{equation*}
			\sum _{i=1}^mb_iu_i'-c_iv_i' = 0.
		\end{equation*}
		By applying $f$, we have that 
		\begin{equation*}
			\sum _{i=1}^\ell b_iu_i -c_iv_i =0. 
		\end{equation*}
		Since $u_1,\dots, u_\ell, v_1,\dots, v_\ell$ are a basis of $p(\mathfrak{h})$, we have that $a_i =b_i+\sqrt{-1}c_i=0$ for $i=1,\dots, \ell$. Therefore
		\begin{equation*}
			\sum _{i=\ell+1}^mb_iu_i'-c_iv_i' = 0. 
		\end{equation*}
		Since $u_{\ell+1}',\dots, u_m',v_{\ell+1}',\dots, v_m'$ are a basis of $\ker f$, we have that $a_i =b_i+\sqrt{-1}c_i=0$ for $i=\ell+1,\dots, m$. This together with $a_i =b_i+\sqrt{-1}c_i=0$ for $i=1,\dots, \ell$ shows that $p'|_{\mathfrak{h}'}$ is injective. 
		
		To see that $(d\alpha)_1^\C(\mathfrak{h}') \subseteq \mathfrak{h}$, we shall see that $(d\alpha)_1^\C(w_i') \in \mathfrak{h}$ for $i=1,\dots, m$. Since $f= (d\alpha)_1$, we have that 
		\begin{equation*}
			(d\alpha)_1^\C (w_i') = \begin{cases}
				u_i+\sqrt{-1}v_i = w_i & \text{for $i=1,\dots, \ell$},\\
				0 & \text{for $i=\ell+1,\dots, m$}. 
			\end{cases}
		\end{equation*}
		Therefore $(d\alpha)_1^\C (\mathfrak{h}') = \mathfrak{h}$. 
		
		It remains to show that the quotient map $q' \co \R^n \to \R^n /p'(\mathfrak{h}')$ sends $\Delta'$ to a complete fan in $\R^n /p'(\mathfrak{h}')$.  Since $p\circ (d\alpha)_1^\C=(d\alpha)_1\circ p'$, $f=(d\alpha)_1$ induces an $\R$-linear map $\overline{f} \co \R^n/p'(\mathfrak{h}') \to \mathfrak{g}/p(\mathfrak{h})$ and we have a commutative diagram 
		\begin{equation*}
			\xymatrix{
				\R^n  \ar[r]^{f} \ar[d]^{q'} & \mathfrak{g} \ar[d]^{q}\\
				\R^n/p'(\mathfrak{h}') \ar[r]^{\overline{f}} & \mathfrak{g}/p(\mathfrak{h}).
			}
		\end{equation*}
		Since $q\circ f$ sends $\Delta'$ to a complete fan $q(\Delta)$ in $\mathfrak{g}/p(\mathfrak{h})$, it suffices to show that $\overline{f}$ is an isomorphism. Since $f$ is surjective, so is $\overline{f}$. Let $u \in \R^n$ and suppose that $f(u) \in p(\mathfrak{h})$. Then, $u \in \ker q \circ f$. Therefore, $u$ is a linear combination of $u_1',\dots, u_m', v_1',\dots, v_m'$, because $f(u_1'),\dots, f(u_\ell'), f(v_1'),\dots, f(v_\ell')$ are a basis of $p(\mathfrak{h})$ and $u_{\ell+1}',\dots, u_m',v_{\ell+1}',\dots, v_m'$ are a basis of $\ker f$. Suppose that
		\begin{equation*}
			u = \sum _{i=1}^m b_iu_i'-c_iv_i',
		\end{equation*}
		where $b_i$ and $c_i$ are some real number for $i=1,\dots, m$. Put $a_i := b_i +\sqrt{-1}c_i$. Then, 
		\begin{equation*}
			p'\left( \sum_{i=1}^m a_iw_i'\right) = \sum _{i=1}^m b_iu_i'-c_iv_i' = u. 
		\end{equation*}
		Since $\mathfrak{h}'$ is spanned by $w_i'$ for $i=1,\dots, m$, we have that $u \in p'(\mathfrak{h}')$, showing that $\overline{f}$ is injective. The theorem is proved. 
	\end{proof}
	\begin{rema}
		We may assume that the corresponding LVMB datum of the LVMB manifold $X_{n,m}$ appeared in Theorem \ref{theo:universal} has the indispensable integer $0$ because the fan $\Delta'$ is a subfan of the fan of $\C^n$ in the proof. As we saw in Lemma \ref{lemm:LVMB0indis} below, the LVMB manifold $X_{n,m}$ is homeomorphic to a moment-angle manifold coming from a star-shaped sphere. 
	\end{rema}
	
\section{Nondegenerate $\C^n$-actions on complex $n$-manifolds}\Label{sec:nondeg}
	The notion of nondegenerate $\R^n$-actions on real $n$-manifolds is introduced by Zung and Minh (see \cite{Zung-Minh} for details). 
	In this section, we will briefly review a part of their works and give a complete classification of complex manifolds of complex dimension $n$ equipped with holomorphic nondegenerate actions of $\C^n$ as an application of Theorem \ref{theo:maintheo}.

	Let $M$ be an $n$-dimensional smooth manifold, equipped with an $\R^n$-action $\rho \co \R^n \times M \to M$. For each $v \in \R^n$, let $X_v$ denote the fundamental vector field generated by $v$, that is, 
	\begin{equation*}
		X_v(x) = \frac{d}{dt}\rho (tv \cdot x)|_{t=0}
	\end{equation*}
	for $x \in M$. A point $x \in M$ is said to be a \emph{singular point} if the dimension of the vector space 
	\begin{equation*}
		\mathfrak{k}_x := \{ v \in \R^n \mid X_v(x)=0\}
	\end{equation*}
	is greater than $0$. At a singular point $x$, we can consider the isotropy representation $\pi _x \co \mathfrak{k}_x \to GL (T_xM/T_x(\R^n\cdot x))$. The singular point $x$ is said to be \emph{nondegenerate} if the differential $(d\pi_x)_0$ of $\pi_x$ at $0 \in \mathfrak{k}_x$ is injective, that is, the whole image by $(d\pi_x)_0$ is a Cartan subalgebra of $\mathfrak{gl} (T_xM/T_x(\R^n\cdot x))$. The action $\rho$ is called \emph{nondegenerate} if any singular point is nondegenerate. 
	
	Let $\rho \co \R ^n \times M \to M$ be a nondegenerate action. According to \cite[Theorem 2.6]{Zung-Minh}, at a singular point $x$, there exist unique nonnegative integers $h,e$, a basis $(v_1,\dots, v_n)$ of $\R^n$ and a local coordinate system $(x_1,\dots, x_n)$ in a neighborhood  of $x$ such that the vector fields are represented as 
	\begin{equation*}
		\begin{cases}
			X_{v_i} = x_i \frac{\partial}{\partial x_i} & \text{for $i=1,\dots, h$},\\
			X_{v_{h+2j-1}} =x_{h+2j-1}\frac{\partial}{\partial x_{h+2j-1}}+ x_{h+2j}\frac{\partial}{\partial x_{h+2j}}, \\
			X_{v_{h+2j}} = x_{h+2j-1}\frac{\partial}{\partial x_{h+2j}}- x_{h+2j}\frac{\partial}{\partial x_{h+2j-1}}  & \text{for $j=1,\dots, e$}, \\
			X_{v_k} = \frac{\partial}{\partial x_k}& \text{for $k=h+2e+1,\dots, n$}.
		\end{cases}
	\end{equation*}
	$h$ is called the \emph{number of hyperbolic component} and $e$ is called the \emph{number of elbolic component} at $x$. The pair $(h,e)$ of nonnegative integers is called the \emph{HE-invariant} at $x$. 
	
	When the HE-invariant at $x$ is $(h,e)$, the dimension of $\R^n$-orbit through $x$ is $n-h-2e$. Therefore, $\R^n \cdot x$ is diffeomorphic to $\R^r \times (S^1)^t$ for some $r,t$ with $r+t=n-h-2e$. The quadruple $(h,e,r,t)$ of nonnegative integers is called the \emph{HERT-invariant} at $x$.

	Let $Z_\rho$ be the global stabilizers, that is, 
	\begin{equation*}
		Z_\rho = \{ v \in \R^n \mid \rho (v,x) =x \text{ for all $x \in M$}\}. 
	\end{equation*}
	Since $\rho$ is nondegenerate, $Z_\rho$ is a discrete subgroup of $\R^n$. Therefore $G := \R Z_\rho/Z_\rho$, a compact torus, acts on $M$ effectively. The dimension of $G$ is said to be the \emph{toric degree} of the action $\rho$. Suppose that $M$ is connected. Then, for any point $x$, the HERT-invariant at $x$ tells us the toric degree; the toric degree of $\rho$ is equal to $e+t$, where $(h,e,r,t)$ is the HERT-invariant at arbitrary point $x \in M$ (see \cite[Theorem 3.4]{Zung-Minh}). In case when $M$ is compact, there exists a compact $\R^n$-orbit (see \cite[Proposition 2.22]{Zung-Minh}). Namely, there exists a point $x$ such that the HERT-invariant $(h,e,r,t)$ at $x$ satisfies $r=0$. Moreover if $x$ does not admit a hyperbolic component, that is, $h=0$, then $2e+t=n$.

	The action $\rho \co \R^n \times M \to M$ is called \emph{elbolic} if each singular point $x \in M$ does not admit hyperbolic components. The relation between maximal torus actions and elbolic actions is as follows:
	\begin{lemm}\Label{lemm:elbolic-maximal}
		Let $M$ be a compact connected manifold of dimension $n$, equipped with an elbolic $\R^n$-action. Let $G$ be the compact torus as above. Then, the action of $G$ on $M$ is maximal. 
	\end{lemm}
	\begin{proof}
		As we mentioned above, there exists a point $x \in M$ such that the HERT-invariant at $x$ is $(0,e,0,t)$, $2e+t= n$. So there exist a basis $(v_1,\dots, v_n)$ of $\R^n$ and a local coordinate system $(x_1,\dots, x_n)$ in a neighborhood at $x$ such that the vector fields are represented as 
		\begin{equation*}
			\begin{cases}
				X_{v_{2j-1}}= x_{2j-1}\frac{\partial}{\partial x_{2j-1}} + x_{2j}\frac{\partial}{\partial x_{2j}},\\
				X_{v_{2j}} = x_{2j-1}\frac{\partial}{\partial x_{2j}} - x_{2j-1}\frac{\partial}{\partial x_{2j}}& \text{for $j=1,\dots, e$},\\
				X_{v_{k}}= \frac{\partial}{\partial x_k} & \text{for $k=2e+1,\dots, n$}. 
			\end{cases}
		\end{equation*}
		Let $v = \sum _{i =1}^n a_iv_i$, $a_i \in \R$. Suppose that $v \in \R Z_\rho$. Then, the vector field $X_v$ is almost periodic. So $a_{2j-1} = 0$ for all $j=1,\dots, e$ (otherwise, $X_v$ is not almost periodic). The vector subspace 
		\begin{equation*}
			V := \left\{ \sum_{i=1}^n a_iv_i \mid a_i = 0 \text{ for $i=2j-1$, $j=1,\dots, e$}\right\}
		\end{equation*}
		has dimension $e+t$ and contains $\R Z_\rho$. Since the toric degree, the dimension of $\R Z_\rho$, is given by $e+t$, we have that $V=\R Z_\rho$. Therefore for $v \in \R Z_\rho$, $X_v(x)=0$ if and only if $v$ is a linear combination of $v_{2j}$, $j=1,\dots, e$. This shows that the dimension of the isotropy subgroup $G_x$ at $x$ is equal to $e$. It follows from $2e+t=n$ and $\dim G=e+t$ that $\dim G+ \dim G_x = \dim M$. Since the action of $G$ on $M$ is effective, the equality $\dim G+\dim G_x = \dim M$ shows that the $G$-action on $M$ is maximal, as required. 
	\end{proof}
	Since $\C^n$ is isomorphic to $\R^{2n}$ as real Lie groups, we can consider holomorphic nondegenerate $\C^n$-actions on complex manifolds of complex dimension $n$. In the case of holomorphic nondegenerate $\C^n$-action, each point $x$ does not admit hyperbolic components.
	\begin{lemm}\Label{lemm:holomorphic-elbolic}
		Let $M$ be a complex manifold of complex dimension $n$, equipped with a holomorphic nondegenerate $\C^n$-action $\rho \co \C^n \times M \to M$. Then, the action $\rho$ is elbolic. 
	\end{lemm}
	\begin{proof}
		Let $x$ be a singular point in $M$ with respect to the action $\rho \co \C^n \times M \to M$. Let $(h,e)$ be the HE-invariant at $x$. Suppose that $h>0$. Then, there exist an element $v \in \C^n$ and a local coordinate system $(x_1,\dots, x_{2n})$ in a neighborhood at $x$ such that the vector field $X_v$ is represented as 
		\begin{equation}\Label{eq:hyperbolic}
			X_v = x_1 \frac{\partial}{\partial x_1}.
		\end{equation}
		Let $J$ be the complex structure on $M$. Since the action $\rho \co \C^n \times M \to M$ is holomorphic, the section 
		\begin{equation*}
			\frac{1}{2}(X_v - \sqrt{-1}JX_v) \co M \to TM \otimes \C
		\end{equation*}
		is a holomorphic vector field. So the zero locus of $X_v$ has even codimension. This contradicts the fact that the zero locus of $X_v$ has codimension $1$ by \eqref{eq:hyperbolic}. Therefore $h=0$. Since each singular point $x$ does not admit hyperbolic components, the action $\rho \co \C ^n \times M \to M$ is elbolic, as required. 
	\end{proof}
	Combining Lemma \ref{lemm:holomorphic-elbolic} with Lemma \ref{lemm:elbolic-maximal}, we have that if a compact connected complex manifold $M$ of complex dimension $n$ admits a holomorphic nondegenerate $\C^n$-action, then $M$ admits a maximal action of a compact torus $G$ which preserves the complex structure on $M$. This allows us to classify compact connected complex manifolds equipped with holomorphic nondegenerate $\C^n$-actions with Theorems \ref{theo:maintheo} and \ref{theo:biholo}.
	\begin{theo}\Label{theo:nondeg}
		Let $M$ be a compact connected complex manifold of complex dimension $n$. $M$ admits a holomorphic nondegenerate $\C^n$-action if and only if $M$ admits an action of a compact torus $G$ which is maximal and preserves the complex structure on $M$. 
	\end{theo}
	\begin{proof}
		The ``only if" part follows from Lemmas \ref{lemm:elbolic-maximal} and \ref{lemm:holomorphic-elbolic} immediately. 
		
		In order to show the ``if" part, suppose that $M$ admits an action of a compact torus $G$ which is maximal and preserves the complex structure on $M$. Then, we have the complexified action $G^\C \times M \to M$ of the action of $G$. 
		
		Let $G^M$ be as in Section \ref{sec:complexified}. Through the exponential map, the Lie algebra $\mathfrak{g}^M$ of $G^M$ acts on $G^M$ holomorphically and effectively. Since $M$ has an open dense $G^\C$-orbit, it also has an open dense $G^M$-orbit. Therefore $\mathfrak{g}^M$ is isomorphic to $\C^n$. Let $N$ be as in Section \ref{sec:fan}. For any point $p \in N$, it follows from  Proposition \ref{prop:slicetheorem} that $p$ is nondegenerate with respect to the action $\mathfrak{g}^M \times M \to M$. As we saw in Remark \ref{rema:N}, $N$ coincides with the whole manifold $M$ and hence any singular point with respect to the action $\mathfrak{g}^M \times M \to M$ is nondegenerate. This shows that the $\mathfrak{g}^M$-action is nondegenerate, proving the ``if" part. 
		
	\end{proof}
\section{Possibility of being K\"{a}hler}\Label{sec:kaehler}
	Let $M$ be a compact connected complex manifold equipped with a maximal action of a compact torus $G$ which preserves the complex structure on $M$. In this section, we give a necessary condition for a compact connected complex manifold $M$ to admit a K\"{a}hler metric in terms of the nonsingular fan $\Delta$, where $\mathcal{F}_1(M,G) = (\Delta,\mathfrak{h},G)$. 
	\begin{theo}\Label{theo:kaehler}
		Let $(M,G,y) \in \mathcal{C}_1$ and let $(\Delta, \mathfrak{h}, G) = \mathcal{F}_1(M,G,y)$. Let $\mathfrak{f}$ be the subspace of the Lie algebra $\mathfrak{g}$ of $G$ which is spanned by generators of all $1$-cones in $\Delta$. Suppose that $M$ admits a K\"{a}hler metric. Then, $\dim \mathfrak{f} = \dim M - \dim G$. 
	\end{theo}
	\begin{proof}
		Let $h$ be a K\"{a}hler metric on $M$. By taking average of $h$ with the action of $G$, we may assume that $h$ is $G$-invariant without loss of generality. In this case, all elements of $G$ act on $M$ as isometries.
		
		Let $\lambda_1,\dots, \lambda_k \in \Hom(S^1,G)$ be the primitive generators of $1$-cones in $\Delta$. Denote by $G_i$ the circle subgroup $\lambda_i(S^1)$. By definition, $G_i$ fixes a characteristic submanifold of $M$ pointwise. In particular, the fixed point set $M^{G_i}$ of the action of $G_i$ is not empty. By Frankel's theorem (see \cite[LEMMA 2]{Frankel}), it follows from $M^{G_i} \neq \emptyset$ that $G_i$ acts on $M$ in Hamiltonian fashion, that is, the interior product of the K\"{a}hler form and the fundamental vector field of $G_i$ is an exact $1$-form. Therefore, the subtorus $F$ generated by $G_i$ for all $i$ acts on $M$ in Hamiltonian fashion, too. By definition of $F$ and $\mathfrak{f}$, the Lie algebra of $F$ is exactly $\mathfrak{f}$. Let $v$ be a vector in $\mathfrak{f}$ such that $\{\exp (tv) \mid t \in \R\}$ is dense in $F$. Let $\omega$ be the K\"{a}hler form. Since the action of $F$ on $M$ is Hamiltonian, there exists a smooth function $f \co M \to \R$ such that $df = \iota _{X_v}\omega$.
		Since $M$ is compact, there exists a point $x \in M$ which maximizes $f$. At the point $x$, $(df)_x = 0$. Since the K\"{a}hler form $\omega$ is nondegenerate, $(df)_x = 0$ implies that $X_v$ vanishes at $x$. It turns out that $x$ is a fixed point of the action restricted to $\{\exp (tv) \mid t \in \R\}$. By definition of $v$, $x$ is a fixed point of the action of $F$. 
		Since the action of $G$ on $M$ is effective, so is the action of $F$. Since $F \subseteq G_x$, we have that 
		\begin{equation*}
			\dim F \leq \dim G_x \leq \dim M - \dim G
		\end{equation*}
		by \eqref{eq:inequality2}. 
		Therefore we have that $\dim \mathfrak{f} \leq \dim M - \dim G$. 
		
		The opposite inequality follows from Corollary \ref{coro:pure}. Therefore, $\dim \mathfrak{f} = \dim M - \dim G$, as required. 
	\end{proof}
	\begin{theo}\Label{theo:projective}
		Let $M$ be a compact connected complex manifold equipped with a maximal action of a compact torus $G$ which preserves the complex structure on $M$. Suppose that $M$ admits a K\"{a}hler metric. Then, $M$ is a total space of a holomorphic fiber bundle over a compact complex torus whose fibers are projective nonsingular toric varieties. 
	\end{theo}
	\begin{proof}
		Let $(\Delta,\mathfrak{h},G) = \mathcal{F}_1(M,G)$. By Theorem \ref{theo:maintheo}, $M$ is $G$-equivariantly biholomorphic to $X(\Delta)/H$, where $H = \exp(\mathfrak{h})$. We will show that $X(\Delta)/H$ is a total space of a holomorphic fiber bundle over a compact complex torus whose fibers are projective nonsingular toric varieties. 
		
		Let $p \co \mathfrak{g}^\C \to \mathfrak{g}$ be the projection and let $q \co \mathfrak{g} \to \mathfrak{g}/p(\mathfrak{h})$ be the quotient map. Since $(\Delta,\mathfrak{h},G) \in \mathcal{C}_2$, the restriction $p|_\mathfrak{h}$ is injective and $q$ sends $\Delta$ to a complete fan in $\mathfrak{g}/p(\mathfrak{h})$. By Theorem \ref{theo:kaehler}, all cones in $\Delta$ are contained in a subspace $\mathfrak{f}$ of dimension $\dim M - \dim G$. We claim that the restriction $q|_\mathfrak{f} \co \mathfrak{f} \to \mathfrak{g}/p(\mathfrak{h})$ is an isomorphism. The surjectivity of $q|_\mathfrak{f}$ follows from the fact that $\mathfrak{f}$ contains all cones in $\Delta$ and the fact that $q(\Delta)$ is complete. To show the injectivity of $q|_\mathfrak{f}$, it suffices to show that $\dim \mathfrak{g}/p(\mathfrak{h}) = \dim \mathfrak{f}$ because $q|_\mathfrak{f}$ is surjective. It follows from the injectivity of $p|_\mathfrak{h}$ that $\dim \mathfrak{h} = \dim p(\mathfrak{h})$. Since $M$ is biholomorphic to $X(\Delta)/H$, we have
		\begin{equation*}
			\begin{split}
				\dim \mathfrak{g}/p(\mathfrak{h}) &=  \dim \mathfrak{g}- \dim \mathfrak{h}\\
				&= \dim M - \dim G\\
				&= \dim \mathfrak{f} .
			\end{split}
		\end{equation*}
		This together with the surjectivity of $q|_\mathfrak{f}$ shows that $q|_\mathfrak{f}$ is an isomorphism. Since $q(\Delta)$ is complete in $\mathfrak{g}/p(\mathfrak{h})$, $\Delta$ is also complete in $\mathfrak{f} \subseteq \mathfrak{g}$. Let $\Delta_\mathfrak{f}$ denote the restriction of all cones in $\Delta$ to $\mathfrak{f}$. Let $E$ be a subtorus of $G$ which is a complement of $F = \exp (\mathfrak{f})$ in $G$ and let $\mathfrak{e}$ denote its Lie algebra. Under this notation, $\Delta$ is the product of the complete fan $\Delta_\mathfrak{f}$ and the fan $\Delta_0$ consisting of only the origin $\{0\}$ in $\mathfrak{e}$. So the toric variety $X(\Delta)$ associated with $\Delta$ is the product of the complete nonsingular toric variety $X(\Delta_\mathfrak{f})$ and the algebraic torus $X(\Delta_0) = E^\C$. 
		
		Since $E$ is a complement of $F$, we have decompositions $\mathfrak{g} = \mathfrak{e}\oplus \mathfrak{f}$ and $\mathfrak{g}^\C = \mathfrak{e}^\C \oplus \mathfrak{f}^\C$. Let $p_{\mathfrak{e}^\C} \co \mathfrak{g}^\C \to \mathfrak{e}^\C$ be the projection. We claim that $H_{\mathfrak{e}^\C} := \exp (p_{\mathfrak{e}^\C}(\mathfrak{h}))$ is a closed subgroup of $E^\C$. By Lemma \ref{lemm:closed}, it suffices to show that the restriction $p|_{p_{\mathfrak{e}^\C}(\mathfrak{h})} \co p_{\mathfrak{e}^\C}(\mathfrak{h}) \to \mathfrak{e}$ is injective. To show this, we will show that the composition $p \circ p_{\mathfrak{e}^\C} |_{\mathfrak{h}} \co \mathfrak{h} \to e$ is injective. Since $\mathfrak{g}^\C = \mathfrak{g} \oplus \sqrt{-1}\mathfrak{g}$ and $\mathfrak{g} = \mathfrak{e}\oplus \mathfrak{f}$, we may represent each element $w$ of $\mathfrak{h}$ as $w= u_\mathfrak{e}+u_\mathfrak{f}+\sqrt{-1}v_\mathfrak{e}+\sqrt{-1}v_\mathfrak{f}$ for some $u_\mathfrak{e},v_\mathfrak{e} \in \mathfrak{e}$ and $u_\mathfrak{f},v_\mathfrak{f} \in \mathfrak{f}$. Suppose that $p \circ p_{\mathfrak{e}^\C}(w) =0$. Then, $u_\mathfrak{e} = 0$ and hence $w = u_{\mathfrak{f}}+\sqrt{-1}v_{\mathfrak{e}}+\sqrt{-1}v_{\mathfrak{f}}$. Since the restriction $p|_\mathfrak{h} \co \mathfrak{h} \to \mathfrak{g}$ of $p \co \mathfrak{g}^\C \to \mathfrak{g}$ to $\mathfrak{h}$ is injective, $w=0$ if and only if $p(w)=u_\mathfrak{f}=0$. Since $q|_\mathfrak{f} \co \mathfrak{f} \to \mathfrak{g}/p(\mathfrak{h})$ is an isomorphism, it follows from $p(w)=u_f$ and $w\in \mathfrak{h}$ that $u_f=0$. Therefore $w=0$ and hence $p \circ p_{\mathfrak{e}^\C} |_{\mathfrak{h}} \co \mathfrak{h} \to e$ is injective. This shows that  $H_{\mathfrak{e}^\C}$ is a closed subgroup of $E^\C$. 
		It also follows from the injectivity of $p_{\mathfrak{e}^\C}|_{\mathfrak{h}}$ that $H_{\mathfrak{e}^\C}$ does not contain a circle subgroup. Moreover, $\dim H_{\mathfrak{e}^\C}=\dim \mathfrak{h}=2\dim G-\dim M=2\dim G - (\dim F + \dim G)=\dim E$. In particular, $\dim H_{\mathfrak{e}^\C} = \frac{1}{2}\dim E^\C$. Therefore, the quotient $E^\C/H_{\mathfrak{e}^\C}$ is a compact complex torus. 
		
		The first projection $X(\Delta)=E^\C \times X(\Delta_\mathfrak{f}) \to E^\C$ and $p_{\mathfrak{e}^\C}$ induce a holomorphic fiber bundle 
		\begin{equation*}
			\psi \co X(\Delta)/H= (E^\C \times X(\Delta_\mathfrak{f}))/H \to E^\C/H_{\mathfrak{e}^\C}
		\end{equation*}
		whose typical fiber is the complete nonsingular toric variety $X(\Delta_\mathfrak{f})$. Since $X(\Delta)/H$ admits a K\"{a}hler metric, the fiber $X(\Delta_\mathfrak{f})$ also admits a K\"{a}hler metric because each fiber is a complex submanifold of $X(\Delta)/H$. For any complete nonsingular toric variety, it being projective is equivalent to it admitting  a K\"{a}hler metric. Therefore the typical fiber $X(\Delta_\mathfrak{f})$ is a projective nonsingular toric variety, proving the theorem. 
	\end{proof}
	Now we study the topology of $M$ which admits a K\"{a}hler metric. 
	\begin{theo}\Label{theo:trivial}
		Let $M$ be a compact connected complex manifold equipped with a maximal action of a compact torus $G$ which preserves the complex structure on $M$. Suppose that $M$ admits a K\"{a}hler metric. Then, $M$ is equivariantly diffeomorphic to the product of a nonsingular projective toric variety and a compact torus.
	\end{theo}
	\begin{proof}
		We use the same notations as the proof of Theorem \ref{theo:projective} and will show that $X(\Delta)/H$ is diffeomorphic to $E \times X(\Delta_\mathfrak{f})$. Since $E$ and $H_{\mathfrak{e}^\C}$ generates $E^\C$ and $H_{\mathfrak{e}^\C} \cap E =\{ 1_{E^\C}\}$, there exist unique elements $g_E \in E$ and $g_{H_{\mathfrak{e}^\C}} \in H_{\mathfrak{e}^\C}$ such that $g_{E^\C} = g_Eg_{H_{\mathfrak{e}^\C}}$ for any element $g_{E^\C} \in E^\C$. 
		Let $p_{\mathfrak{f}^\C} \co \mathfrak{g}^\C=\mathfrak{e}^\C \oplus \mathfrak{f}^\C \to \mathfrak{f}^\C$ be the projection. Let $H_{\mathfrak{f}^\C} := \exp (p_{\mathfrak{f}^\C}(\mathfrak{h})) \subseteq F^\C$. 
		Since the exponential map $p_{\mathfrak{e}^\C}(\mathfrak{h}) \to H_{\mathfrak{e}^\C}$ is an isomorphism, the composition $p_{\mathfrak{f}^\C} \circ (p_{\mathfrak{e}^\C}|_{p_{{\mathfrak{e}^\C}(\mathfrak{h})}})^{-1} \co p_{\mathfrak{e}^\C}(\mathfrak{h}) \to p_{\mathfrak{f}^\C}(\mathfrak{h})$ induces a group homomorphism $\theta \co H_{\mathfrak{e}^\C} \to H_{\mathfrak{f}^\C}$. 
		We define $\phi \co (E^\C \times X(\Delta_{\mathfrak{f}}))/H \to E \times X(\Delta_\mathfrak{f})$ to be $\phi ([g_{E^\C}, x]) := (g_E, \theta(g_{H_{\mathfrak{e}^\C}})^{-1}\cdot x)$, where $[g_{E^\C},x] \in (E^\C \times X(\Delta_\mathfrak{f}))/H$ denotes the equivalence class of $(g_{E^\C}, x) \in E^\C \times X(\Delta_{\mathfrak{f}})$. Clearly, $\phi$ is well-defined. $\phi$ has the inverse $\phi ^{-1} \co E \times X(\Delta_{\mathfrak{f}}) \to (E^\C \times X(\Delta_\mathfrak{f}))/H$ defined as $\phi^{-1}(g_E,x) = [g_E,x]$. Both $\phi$ and $\phi^{-1}$ are smooth, so $\phi$ is a diffeomorphism and $X(\Delta)/H$ is diffeomorphic to the product $X(\Delta_\mathfrak{f}) \times E$, as required.  
	\end{proof}
	Related with Theorems \ref{theo:kaehler} and \ref{theo:trivial} , we may consider  generalizations. One direction is symplectic version. Let $M$ be a symplectic manifold equipped with an action of a compact torus $G$ which is maximal and preserves the symplectic form on $M$. If $G$ acts on $M$ freely, then $M$ is diffeomorphic to a compact torus. If the fixed point set $M^G$ is nonempty, then the $G$-action on $M$ is Hamiltonian and $M$ is a toric manifold (see \cite[Remark 2.~(4)]{Ishida-Karshon} for details). We already know these extreme cases, but do not know the other cases. 
	
	\begin{prob}
		Let $M$ be a compact connected symplectic manifold equipped with an action of a compact torus $G$, which is maximal and preserves the symplectic form on $M$. Is it true that $M$ is diffeomorphic  to a product of a projective nonsingular toric variety and a compact torus?
	\end{prob}
	The other direction is locally conformal K\"{a}hler version. A Hermitian manifold $(M,h)$ is said to be a \emph{locally conformal K\"{a}hler} manifold if the hermitian metric $h$ is locally conformal to a K\"{a}hler metric, that is, there exists an open cover $\{U_\lambda\}$ of $M$ and positive valued functions $f_\lambda \co U_\lambda \to \R_{>0}$ such that $f_\lambda h$ is a K\"{a}hler metric on $U_\lambda$. A typical example of locally conformal K\"{a}hler manifold is a diagonal Hopf manifold. Let $\widetilde{\alpha} \in \C \setminus \R$ and put $\alpha := e^{\sqrt{-1}\widetilde{\alpha}}$. Define an equivalence relation $\sim$ on $\C^n\setminus \{ 0\}$ to be
	\begin{equation*}
		(z_1,\dots, z_n) \sim (\alpha^kz_1,\dots, \alpha^kz_n) \quad \text{for all $k \in \Z, (z_1,\dots, z_n) \in \C^n\setminus \{0 \}$}.
	\end{equation*} 
	The quotient space $M := \C^n\setminus \{0\}/\sim$ which is called a diagonal Hopf manifold carries the Hermitian metric 
	\begin{equation*}
		h = \frac{1}{\sum_{j=1}^n|z_j |^2}\sum _{i =1}^n dz_i \otimes d\overline{z_i}
	\end{equation*}
	which is locally conformal to a K\"{a}hler metric. Moreover, $h$ is invariant under the action of $(S^1)^n$ on $M$ to which the standard $(S^1)^n$-action on $\C^n\setminus \{0\}$ descends. Also, $h$ is invariant under the action of $S^1 = \R /\Z$ given by $[z_1,\dots, z_n]\mapsto [e^{\sqrt{-1}\widetilde{\alpha}t}z_1,\dots, e^{\sqrt{-1}\widetilde{\alpha}t}z_n]$ for $t \in \R$, where $[z_1,\dots,z_n]\in M$ denotes the equivalence class of $(z_1,\dots, z_n) \in \C ^n \setminus \{0\}$. The action of $G=(S^1)^n \times S^1$ on $M$ is maximal. In fact, the orbit through $[1,0,\dots,0]$ is a real $2$-dimensional torus which is minimal. 
	\begin{prob}\Label{prob:lck}
		Characterize a compact connected complex manifold $M$ equipped with a maximal action of a compact torus $G$ which admits a ($G$-invariant) locally conformal K\"{a}hler metric in terms of $\mathcal{F}_1(M,G)$.
	\end{prob}
	We may expect that a solution of Problem \ref{prob:lck} will provide a lot of concrete examples of locally conformal K\"{a}hler manifolds, as well as the correspondence between K\"{a}hler toric manifolds and Delzant polytopes (see \cite{Delzant}).

\section*{Acknowledgements} The author would like to thank Yael Karshon for fruitful and stimulating discussions. He also would like to thank Laurent Battisti, Mikiya Masuda and Taras Panov for valuable comments.

\end{document}